\documentclass[11pt]{article}
\usepackage{times}
\usepackage{a4wide}
\usepackage{amssymb}
\usepackage{amsmath}
\usepackage{amsthm}
\usepackage{mathtools}
\usepackage{graphicx}
\usepackage{enumerate}
\usepackage{bbold}
\usepackage{url}
\usepackage{enumitem}
\usepackage{enumerate}
\usepackage{scalerel}
\usepackage{hyperref}
\usepackage{caption}
\usepackage{subcaption}

\newcommand*{\medcap}{\mathbin{\scalebox{1.5}{\ensuremath{\cap}}}}%

\usepackage[margin=1.1in]{geometry}
\newtheorem{theorem}{Theorem}[section]

\newtheorem{lemma}[theorem]{Lemma}
\newtheorem{proposition}[theorem]{Proposition}
\newtheorem{hypotheses}[theorem]{Hypotheses}

\theoremstyle{definition}
\newtheorem{exmp}{Example}[section]
\newtheorem{defn}{Definition}[section]

\DeclareMathOperator{\Int}{Int}

\DeclareMathOperator{\MA}{MA}
\DeclareMathOperator{\dom}{dom}
\DeclareMathOperator{\cvx}{cvx}
\DeclareMathOperator{\Ric}{Ric}
\DeclareMathOperator{\Ent}{Ent}
\DeclareMathOperator{\AM}{AM}
\DeclareMathOperator{\lin}{\hspace{.2mm}lin}
\begin{document}

\title{Monge--Amp\`{e}re Iteration}
\date{}
\author{Ryan Hunter\thanks{Partially supported by the BSF grant 2012236 and the NSF grants DMS-1515703 and DMS-1440140.}}

\maketitle

\begin{abstract}
In a recent paper, Darvas--Rubinstein proved a convergence result for the K\"{a}hler--Ricci iteration, which is a sequence of recursively defined complex Monge--Amp\`{e}re equations.  We introduce the Monge--Amp\`{e}re iteration to be an analogous, but more  general, sequence of recursively defined real Monge--Amp\`{e}re second boundary value problems, and we establish sufficient conditions for its convergence.  We then determine two cases where these conditions are satisfied and provide geometric applications for both.  First, we give a new proof of Darvas and Rubinstein's theorem on the convergence of the Ricci iteration in the case of toric Kahler manifolds, while at the same time generalizing their theorem to general convex bodies.  Second, we introduce the affine iteration to be a sequence of prescribed affine normal problems and prove its convergence to an affine sphere, giving a new approach to an existence result due to Klartag.  
\tableofcontents{}
\end{abstract}

\section{Introduction}

Let $(M,g_0)$ be a Riemannian manifold.  A Ricci iteration, in the sense of Rubinstein \cite{Rubinstein}, is a sequence of Riemmannian metrics $\{g_i\}_{i\in\mathbb{N}}$ which satisfy
\[
\Ric (g_{i+1}) = g_i,
\]
where $\Ric(g)$ denotes the Ricci curvature of the metric $g$.  The Ricci iteration is a dynamical system on the space of metrics which breaks up the Einstein equation $\Ric(g)=g$ into a sequence of prescribed curvature problems.  If a Ricci iteration converges to an Einstein metric $g$, then the metrics $\{g_i\}_{i\in\mathbb{N}}$ are approximations of the Einstein metric $g$.  Thus, the Ricci iteration can be thought of as a tool for finding numerical approximations of Einstein metrics.

The Ricci iteration was introduced by Rubinstein \cite{Rubinstein2,Yanirthesis,Rubinstein} and has been since studied by a number of authors \cite{BBGEZ,Darvas,GKY,Keller,Pulemotov,YanirSS}. On K\"{a}hler manifolds the Ricci iteration is a sequence of complex Monge--Amp\`{e}re equations.  We introduce the following analogous sequence of real Monge--Amp\`{e}re equations.

A \textit{Monge--Amp\`{e}re iteration} is a sequence of convex functions $\{ \varphi_i\}_{i\in\mathbb{N}}$ which satisfy
\begin{equation}\label{introiteration}
\begin{cases}
    \dfrac{\det(\nabla^2 \varphi_{i+1})}{\lambda(A)} = \dfrac{h\circ \varphi_i}{\|h\circ \varphi_i\|_1} &\text{on }\mathbb{R}^n\\
    \nabla \varphi_{i+1}(\mathbb{R}^n) = A.
\end{cases}
\end{equation}
where $\lambda$ is Lebesgue measure, $\|\cdot\|_1$ is the $L^1$ norm on $\mathbb{R}^n$, $h$ is a function of one variable, and $A$ is a bounded, convex set.  Our goal is to provide sufficient conditions on $h$ and $A$ so a Monge--Amp\`{e}re iteration converges to a convex solution $\varphi$ to the \textit{Monge--Amp\`{e}re second boundary value problem} given by 
\begin{equation}\label{intropde}
\begin{cases}
    \dfrac{\det(\nabla^2 \varphi)}{\lambda(A)} = \dfrac{h\circ \varphi}{\|h\circ \varphi\|_1} &\text{on }\mathbb{R}^n\\
    \nabla \varphi(\mathbb{R}^n) = A.
\end{cases}
\end{equation}

We will go into details shortly, but let us mention at the outset three main difficulties as compared to the by now well studied complex Monge-Amp\`{e}re equations arising from the Kahler--Ricci iteration \cite{Darvas, Rubinstein}. First, we must deal with a second boundary value condition. Second, we allow rather general right hand sides. Third, we must deal with the noncompactness  of $\mathbb{R}^n$.  Existence and uniqueness of solutions to equation (\ref{intropde}) have been studied in the case $h(t)=e^{-t}$ by Berman and Berndtsson \cite{BB} and in the case $h(t)=t^{-(n+q)}$ for $q>1$ by Klartag \cite{Klartag}.  Thus, we will focus on these two functions for $h$.  Solutions to equation (\ref{introiteration}) are only unique up to an additive constant, so we must first normalize the solutions before we can obtain convergence. For this reason we define a \textit{normalized Monge--Amp\`{e}re iteration} to be a sequence of convex functions $\{\varphi_i\}_{i\in\mathbb{N}}$ solving
\begin{equation}\label{MAiteration}
\begin{cases}
    \dfrac{\det(\nabla^2 \varphi_{i+1})}{\lambda(A)} = \dfrac{h\circ \varphi_i}{\|h\circ \varphi_i\|_1} &\text{on }\mathbb{R}^n\\
    \nabla \varphi_{i+1}(\mathbb{R}^n) = A\\
    \int_A \varphi_{i+1}^* \,d\lambda = -\tau,
\end{cases}
\end{equation}
where $\varphi^*(y) = \sup\,\{\langle x,y\rangle - \varphi(x) \mid x\in\mathbb{R}^n\,\}$ is the Legendre transform of $\varphi$, and $\tau$ is a constant.    We likewise define the \textit{normalized Monge--Amp\`{e}re second boundary value problem} to be 
\begin{equation}\label{ourpde}
\begin{cases}
    \dfrac{\det(\nabla^2 \varphi)}{\lambda(A)} = \dfrac{h\circ \varphi}{\|h\circ \varphi\|_1} &\text{on }\mathbb{R}^n\\
    \nabla \varphi(\mathbb{R}^n) = A\\
    \int_A \varphi^* \,d\lambda = -\tau.
\end{cases}
\end{equation}

Our main theorem establishes the convergence of the normalized Monge--Amp\`{e}re iteration in the cases $h(t)=e^{-t}$ and $h(t)=t^{-(n+q)}$ for $q>1$.  In order to unify the proofs, we define hypotheses which are satisfied for both functions $h$ and prove that these hypotheses are sufficient for the convergence of the normalized iteration.  We also believe our methodical proof reveals some of the inherent structure of the problem.

We define the space of continuous functions with at most linear growth
\begin{equation}\label{classC}
\mathcal{C}_{\lin} = \big\{\,f: \mathbb{R}^n \to (\tau,\infty) \mid f\text{ continuous, and } f(x)\leq C \,(1+|x|)\text{ for some constant }C\,\big\},
\end{equation}
and the space of probability measures with finite first moments
\[
    \mathcal{P}_1 = \left\{\,\mu \in \mathcal{P} \mid \int_{\mathbb{R}^n} |x| \,d\mu < \infty\,\right\}.
\]
The natural pairing between $\mathcal{C}_{\lin}$ and $\mathcal{P}_1$ is
\[
\langle f,\mu\rangle = \int_{\mathbb{R}^n} f\,d\mu.
\]
We define the Monge--Amp\`{e}re measure of a convex function $\varphi$ on a Borel set $U$ by
\[
\MA(\varphi)(U) = \lambda( \partial \varphi(U)),
\]
where $\partial \varphi$ is the subgradient of $\varphi$.  The Monge--Amp\`{e}re measure is an extension of the measure $\det(\nabla^2 \varphi)\,\lambda$ to any non-smooth convex function.  Its definition is due to Alexandrov \cite{Alexandrov}, and a thorough treatment can be found in Rauch--Taylor \cite{RT}.

We define the functional $\mathcal{F}: \mathcal{C}_{\lin} \to \mathbb{R}$ by
\begin{equation}\label{Fdef}
    \mathcal{F}(f) = H^{-1}\big(\|H\circ f\|_1\big)\hspace{3mm}\text{where}\hspace{3mm}H(t) = \int_t^\infty h\,d\lambda,
\end{equation}
and a dual functional $\mathcal{G}:\mathcal{P}_1 \to \mathbb{R}$ by
\begin{equation}\label{Gdef}
\mathcal{G}(\mu) = \inf \,\big\{ \,g\big(\langle f,\mu\rangle,\,\mathcal{F}(f)\big) \mid f\in \mathcal{C}_{\lin}\, \big\},
\end{equation}
for a function $g$ of two variables.  When $h(t)=e^{-t}$ we choose $g(s,t)=s-t$, and when $h(t)=t^{-(n+q)}$ we choose $g(s,t)=s/t$ for $s>0$. In general, $g$ is chosen so that the following hypotheses are satisfied:

\begin{hypotheses}\label{hypB} \hspace{1cm}
\begin{enumerate}[label=(B\arabic*)]
\item $h$ is smooth, positive, decreasing, and there exist $p>0$ and $C>0$ such that $h(t) \leq C\,t^{-(n+p+1)}$ for $t\gg 1$\label{B1}.
\item If smooth solutions to equation (\ref{ourpde}) exist, then they are unique up to translations of $\mathbb{R}^n$\label{E}.
\item $g(s,t)$ is differentiable, decreasing in $t$, and satisfies $g(s,g(s,t))=t$\label{B2}.
\item If $\{\varphi_i\}_{i\in\mathbb{N}}$ is a sequence of smooth, strictly convex functions solving equation (\ref{MAiteration}), then 
\[
g\Big(\Big\langle \varphi_{i+1},\,\scaleobj{.9}{\dfrac{\MA(\varphi_{i+1})}{\lambda(A)}}\Big\rangle,\,\mathcal{G}\Big(\scaleobj{.9}{\dfrac{\MA(\varphi_{i+1})}{\lambda(A)}}\Big)\Big) \leq \mathcal{F}(\varphi_i) \text{ for all }i.
\]\label{B3}
\end{enumerate}
\end{hypotheses}

In Subsection \ref{hypexplanation} we explain the role of Hypotheses \ref{hypB} in the proof of convergence of the normalized Monge--Amp\`{e}re iteration.  Hypotheses \ref{B1} is a mild technical assumption, and \ref{E} is necessary for convergence. Hypotheses \ref{B2} and \ref{B3} imply
\[
\big\{\mathcal{F}(\varphi_i)\big\}_{i\in\mathbb{N}} \text{ and }\Big\{ g\Big(\Big\langle \varphi_{i+1},\,\scaleobj{.9}{\dfrac{\MA(\varphi_{i+1})}{\lambda(A)}}\Big\rangle,\,\mathcal{G}\Big(\scaleobj{.9}{\dfrac{\MA(\varphi_{i+1})}{\lambda(A)}}\Big)\Big)\Big\}_{i\in\mathbb{N}}
\]
are decreasing sequences along the Monge--Amp\`{e}re iteration. In Subsection \ref{hypexplanation} we explain the utility of these decreasing subsequences.

We will always assume
\begin{equation}\label{Aconditions}
A\subset \mathbb{R}^n \text{ is open, bounded, convex, and }\, \int_A y_i\,d\lambda=0\, \text{ for }i=1,\ldots,n,
\end{equation}
meaning the barycenter of $A$ lies at the origin.  Our main theorem is:

\begin{theorem}\label{introthm}
Assume $A$ satisfies (\ref{Aconditions}) and equation (\ref{MAiteration}) satisfies Hypotheses \ref{hypB}.  Let $\{\varphi_i\}_{i\in\mathbb{N}}$ be a normalized Monge--Amp\`{e}re iteration solving equation (\ref{MAiteration}), and let $\widetilde{\varphi}_i(x) = \varphi_i(x+a_i)$ for $a_i$ such that $\varphi_i(a_i)=\inf \varphi_i$.  Then, there exists a smooth, convex solution $\varphi$ to equation (\ref{ourpde}) such that $\widetilde{\varphi}_i$ converges to $\varphi$ on compact sets in every $C^{k,\alpha}$ norm.
\end{theorem}

We take a moment here to contrast the proof of Theorem \ref{introthm} with the proof of the convergence of the K\"{a}hler-Ricci iteration in Darvas--Rubinstein \cite{Darvas}.  The proof of Theorem \ref{introthm} relies upon a classical variational argument concerning minimizers of a certain functional.  Darvas and Rubinstein's proof relies upon a more exotic notion of weak convergence in an infinite dimensional Finsler geometry introduced by Darvas \cite{Darvas2}.

We prove Hypotheses \ref{hypB} are satisfied when $h(t)=e^{-t}$ in Subsection \ref{iteration2}, and when $h(t)=t^{-(n+q)}$ for $q>1$ in Subsection \ref{iteration1}.  These are the two cases for which Hypothesis \ref{E} has been established in earlier work.

For each step of the Monge--Amp\`{e}re iteration, the function $\varphi_{i+1}$ arises as the solution to an optimal transport problem.  Specifically, if $\varphi_{i+1}$ solves equation (\ref{MAiteration}), then
\[
(\nabla \varphi_{i+1})_{\#} \bigg( \dfrac{h\circ \varphi_i}{\|h \circ \varphi_i\|_1} \,\lambda \bigg) = \dfrac{\lambda \big|_A}{\lambda(A)}, 
\]
and $\nabla \varphi_{i+1}$ minimizes the cost
\[
\int_{\mathbb{R}^n} |x-T(x)|^2 \,\dfrac{h\circ \varphi_i}{\|h \circ \varphi_i\|_1} \,d\lambda
\]
over all maps $T:\mathbb{R}^n \to A$ which also push forward the probability measure $\dfrac{h\circ \varphi_i}{\|h\circ \varphi_i\|_1}\,\lambda$ to the uniform measure on $A$.  Numerical approximations of $\varphi_{i+1}$ are facilitated by this optimal transportation interpretation, as is shown in the work of Lindsey and Rubinstein \cite{Lindsey} and references therein.  This demonstrates the potential of using a Monge--Amp\`{e}re iteration to numerically approximate solutions of equation (\ref{ourpde}).

{\bf Organization.} In Section \ref{GA} we outline two geometric applications of Theorem \ref{introthm} in toric K\"{a}hler geometry and affine differential geometry.  In Section \ref{thmpf} we  motivate Hypotheses \ref{hypB} and prove Theorem \ref{introthm}.  In Section \ref{RIsection} we show Monge--Amp\`{e}re iterations for $h(t)=e^{-t}$ correspond to Ricci iterations on toric K\"{a}hler manifolds, and we prove Theorem \ref{introthm} implies the convergence of the Ricci iteration to K\"{a}hler Einstien metrics when they exist. In Section \ref{AIsection} we introduce the affine iteration to be a sequence of prescribed affine normal problems in affine differential geometry which correspond to a Monge--Amp\`{e}re iteration with $h(t)=t^{-(n+2)}$. We prove Theorem \ref{introthm} implies the convergence of the affine iteration to affine spheres when they exist.

\section{Geometric applications}\label{GA}

The Monge--Amp\`{e}re iteration is inspired by the Ricci iteration.  Darvas and Rubinstein \cite{Darvas} proved a convergence result for the Ricci iteration on K\"{a}hler manifolds, and our first geometric application is to show Theorem \ref{introthm} recovers a special case of their result for toric K\"{a}hler manifolds.  Our second geometric application is to affine differential geometry, where we introduce the affine iteration and show Theorem \ref{introthm} proves the affine iteration converges to an affine sphere.

\subsection{K\"{a}hler--Ricci iteration}

In Section \ref{iteration2} we prove Hypotheses \ref{hypB} are satisfied for $h(t)=e^{-t}$. After showing the hypotheses are satisfied, we apply Theorem \ref{introthm} to prove the following:

\begin{theorem}\label{MAiteration2}
Assume $A\subset\mathbb{R}^n$ satisfies (\ref{Aconditions}), and fix $\tau\in\mathbb{R}$ and the function $h(t)=e^{-t}$.  If $\{\phi_i\}_{i\in\mathbb{R}}$ is a normalized Monge--Amp\`{e}re iteration, then there exist constants $a_i\in \mathbb{R}^n$ such that $\widetilde{\phi}_i(x) = \phi_i(x+a_i)$ converges to $\phi$, a smooth convex solution to equation (\ref{ourpde}), on compact sets in every $C^{k,\alpha}$ norm.
\end{theorem}

In the rest of Section \ref{RIsection} we interpret Theorem \ref{MAiteration2} in terms of the convergence of the K\"{a}hler--Ricci iteration on toric K\"{a}hler manifolds which we outline here.

Let $X$ be a compact K\"{a}hler manifold which is Fano, meaning it has positive first Chern class.  A K\"{a}hler--Ricci iteration is a sequence of K\"{a}hler metrics $\{\omega_i\}_{i\in\mathbb{N}}$ in the first Chern class of $X$ solving the sequence of prescribed Ricci curvature problems
\begin{equation}\label{introRI}
\Ric(\omega_{i+1}) = \omega_i.
\end{equation}
The K\"{a}hler--Ricci iteration can be thought of as a discretization of the K\"{a}hler--Ricci flow, which is given by $\partial_t\,\omega=-\Ric(\omega)+\omega$ when $c_1(X)$ is positive.  Each step of the K\"{a}hler--Ricci iteration admits a unique, smooth solution by the Calabi-Yau Theorem \cite{Yau}.  Darvas--Rubinstein \cite[Theorem~1.2]{Darvas} prove that when $X$ admits a K\"{a}hler--Einstein metric, there exist automorphisms $g_i$ such that $g_i^*\,\omega_i$ converges smoothly to a K\"{a}hler--Einstein metric.  

A compact K\"{a}hler manifold $X$ of complex dimension $n$ is toric if it admits an effective holomorphic $\mathbb{C}^*{}^n$ action with an open, dense orbit $X_0$. They are characterized by certain compact, convex polytopes $P \subset \mathbb{R}^n$, and K\"{a}hler metrics $\omega$ which are invariant under the action of $(S^1)^n\subset\mathbb{C}^*{}^n$ can be described in terms of a potential in the open orbit $X_0\simeq \mathbb{C}^*{}^n$:
\begin{equation}\label{potential}
\omega \big|_{\mathbb{C}^*{}^n} = \sqrt{-1} \partial \overline{\partial}\,\phi.
\end{equation}
The $(S^1)^n$-invariance implies $\phi$ only depends on the real variables $x_i = \log(|z_i|^2)$, and $\phi:\mathbb{R}^n \to \mathbb{R}$ is a smooth, strictly convex function satisfying $\nabla \phi(\mathbb{R}^n) = \Int P$.  The K\"{a}hler--Ricci iteration can be written in terms of these potentials as the Monge--Amp\`{e}re iteration
\begin{equation}\label{first}
\begin{cases}
    \dfrac{\det(\nabla^2 \phi_{i+1})}{\lambda(P)} =\dfrac{e^{-\phi_i}}{\|e^{-\phi_i}\|_1}\\
    \nabla \phi_{i+1}(\mathbb{R}^n) = \Int P.
\end{cases}
\end{equation}

By work of Wang--Zhu \cite{WangZhu}, K\"{a}hler--Einstein metrics exist on toric Fano manifolds if and only if the barycenter of $P$ lies at the origin.  In Section \ref{proofofRI} we use this result along with Theorem \ref{MAiteration2} to provide an alternate proof of the following special case of Darvas--Rubinstein \cite[Theorem~1.2]{Darvas}:

\begin{theorem}\label{introthm3}
Let $X$ be a smooth, toric Fano manifold admitting a K\"{a}hler--Einstein metric.  If  $\{\omega_i\}_{i\in\mathbb{N}}$ is a K\"{a}hler--Ricci iteration on $X$, then there exists a sequence of automorphisms $\{g_i\}_{i\in\mathbb{N}}$ such that $g_i^*\omega_i$ converges smoothly to a K\"{a}hler--Einstein metric.
\end{theorem}

Conversely, the results of Darvas and Rubinstein can be used to prove the analytic Theorem \ref{MAiteration2} in the special case when $A=\Int\,P$ for a Delzant polytope $P$.  Thus, Theorem \ref{MAiteration2} generalizes the analytic results of Darvas--Rubinstein \cite{Darvas} to all bounded, convex sets.

\subsection{Affine iteration}

In Section \ref{iteration2} we prove Hypotheses \ref{hypB} are satisfied for $h(t)=t^{-(n+p+1)}$ when $p>0$. After showing the hypotheses are satisfied, we apply Theorem \ref{introthm} to prove the following:

\begin{theorem}\label{MAiteration1}
Assume $A\subset\mathbb{R}^n$ satisfies (\ref{Aconditions}), and fix $p>0$, $\tau>0$, and the function $h(t)=t^{-(n+p+1)}$.  If $\{\phi_i\}_{i\in\mathbb{N}}$ is a normalized Monge--Amp\`{e}re iteration , then there exist constants $a_i\in \mathbb{R}^n$ such that $\widetilde{\phi}_i(x) = \phi_i(x+a_i)$ converges to $\phi$, a smooth convex solution to equation (\ref{ourpde}), on compact sets in every $C^{k,\alpha}$ norm.
\end{theorem}

Theorem \ref{MAiteration1} recovers the existence of smooth solutions to equation (\ref{intropde}) with $h(t)=t^{-(n+p+1)}$ for $p>0$ which was first proven by Klartag \cite{Klartag}.  The proof of Theorem \ref{MAiteration1} relies upon the uniqueness of smooth solutions to equation (\ref{intropde}) with $h(t)=t^{-(n+p+1)}$ for $p>0$.  We cite Klartag's proof of the uniqueness of weak solutions, but uniqueness of smooth solutions could be proven more easily.

In the rest of Section \ref{AIsection} we consider the case $p=1$ and interpret Theorem \ref{MAiteration1} in terms of affine differential geometry.

Affine differential geometry is concerned with immersions $f:M^n \hookrightarrow \mathbb{R}^{n+1}$ and their properties which are \textit{equiaffine}, meaning they are invariant under volume preserving affine transformations of the form $x\mapsto A\,x + v$ for $A\in Sl_{n+1}\mathbb{R}$ and $v\in\mathbb{R}^{n+1}$.  In Subsection \ref{affineimmersions} we define the affine normal $\xi:M \to T\,\mathbb{R}^{n+1} \big|_{f(M)}$ to be a unique equiaffine transversal vector field.  

An important object of study in affine geometry are \textit{affine spheres} which are immersions $f$ satisfying
\[
f(x) + c\, \xi = x_0,
\]
for a constant $c\in\mathbb{R}$ and $x_0 \in \mathbb{R}^n$. On an affine sphere, the affine normals, once scaled by a constant, all meet at a point $x_0$ called the \textit{center} of the affine sphere.  When $c>0$ they are called \textit{elliptic affine spheres}.  When the manifold $M$ is compact, the only affine spheres are ellipsoids as proven by Blaschke \cite{Blaschke} for $n=2$ and Deicke \cite{Deicke} in higher dimensions.  By further work of Calabi \cite{Calabi} and Cheng--Yau \cite{CY}, completeness of an elliptic affine sphere implies compactness, and so the only complete examples of elliptic affine spheres are ellipsoids.  

To allow for a larger class of elliptic affine spheres we consider incomplete immersions.  Specifically we study immersions which are the graph of the Legendre transform of convex functions.  In Subsection \ref{leggraphs} we introduce the \textit{Legendre graph immersion over $A$} to be the immersion $f_\phi: \mathbb{R}^n \hookrightarrow \mathbb{R}^{n+1}$ given by
\[
f_\phi(x) = \big(\nabla \phi(x), \langle x,\nabla \phi(x) \rangle  -\phi(x) \big) = \big(\nabla \phi(x) , \,\phi^*(\nabla \phi(x))\big),
\]
for $\phi:\mathbb{R}^n \to \mathbb{R}$, a smooth, strictly convex function such that  $\nabla \phi(\mathbb{R}^n)=A$, a bounded convex set.  We introduce the following sequence of prescribed affine normal problems for Legendre graph immersions over $A$.

\begin{defn}
An \textit{affine iteration over A} is a sequence $\big\{f_i = \big(\nabla \phi_i(x),\,\phi_i^*(\nabla \phi_i(x))\big)\big\}_{i\in\mathbb{N}}$ of Legendre graph immersions over $A$ satisfying 
\[
\begin{cases}
\xi_{i+1}(x) = -c_{i+1}\,f_{i}(x)\\
\int_A \phi_i^*\,d\lambda=-\tau
\end{cases}
\]
for positive real numbers $\tau$ and $\{c_i\}_{i\in\mathbb{N}}$, where $\xi_{i+1}$ is the affine normal of $f_{i+1}$.
\end{defn}

In Proposition \ref{prescribednormal} we show results of McCann \cite{McCann} and Brenier \cite{Brenier2} on optimal transport imply there is a unique affine iteration over $A$ for every $f_0$ which is an initial Legendre graph immersion over $A$.

In Section \ref{affineiterationsection} we show a normalized Monge--Amp\`{e}re iteration $\{\phi_i\}_{i\in\mathbb{N}}$ with $h(t) = t^{-(n+2)}$ corresponds to an affine iteration $f_i$ over $A$. We use Theorem \ref{MAiteration1} to prove the following theorem on the convergence of the affine iteration.

\begin{theorem}\label{affineiterationconverges}
Assume $A\subset \mathbb{R}^n$ satisfies (\ref{Aconditions}).  If $\{f_i\}_{i\in\mathbb{N}}$ is an affine iteration over $A$, then there exist matrices $M_i \in Sl_{n+1}\mathbb{R}$ such that $M_i\cdot f_i(\mathbb{R}^n)$ converge smoothly to an elliptic affine sphere with center at the origin.
\end{theorem}

\section{Proof of Theorem \ref{introthm}}\label{thmpf}

This section is organized as follows.  In Subsection \ref{hypexplanation} we motivate Hypotheses \ref{hypB} and briefly explain their main role in the proof.  In Subsection \ref{outline} we outline the proof of Theorem \ref{introthm} and break it down into five steps.  In Sections 3.3\,--\,3.7 we prove each step of the outline.

\subsection{Explanation of Hypotheses \ref{hypB}}\label{hypexplanation}
\textit{Hypothesis \ref{B1}} :

If $\varphi$ solves equation (\ref{ourpde}), then $h\circ \varphi$ is in $L^1(\mathbb{R}^n)$.  Thus, it is natural to stipulate a decay condition on $h$ to guarantee $\|h\circ \varphi\|_1 < \infty$.  Assume $h$ is a positive, decreasing function such that
\[
h(t) \leq C\,t^{-(n+p+1)} \text{, for }p>1 \text{ when }t\gg1.
\]
If $\varphi$ is convex and $\nabla \varphi(\mathbb{R}^n)$ contains the origin in its interior, then $\varphi(x) \geq r\,|x|$ for some $r$ as $|x| \to \infty$. Since $h$ is decreasing, $h(\varphi(x)) \leq h(r\,|x|)$ for large $|x|$, and the asymptotic bound implies $\|h\circ \varphi\|_1 < \infty$. The bounds on $h$ also imply the bounds
\[
0<H(t) \leq c\,t^{-(n+p)} \text{ when }t\gg1,
\]
for $H(t) = \int_t^\infty h\,d\lambda$.  This implies $\mathcal{F}(f) = H^{-1}\big(\|H\circ f\|_1\big) < \infty$ for $f\in\mathcal{C}_{\lin}$.  The smoothness and positivity of $h$ are necessary to prove regularity of solutions to equation (\ref{MAiteration}).

\noindent \textit{Hypothesis \ref{E}} :

If the solutions to equation (\ref{ourpde}) were not unique up to translations, then a translated normalized Monge--Amp\`{e}re iteration $\widetilde{\varphi}_i(x) = \varphi_i(x+a_i)$, for $a_i$ such that $\varphi_i(a_i)=\inf \varphi_i$, could have two subsequences converging to different solutions of equation (\ref{ourpde}). 

\noindent \textit{Hypothesis \ref{B2} and \ref{B3}} :

The definition of $\mathcal{G}$ implies 
\[
\mathcal{G}(\mu) \leq g\big(\langle f,\mu\rangle ,\,\mathcal{F}(f)\big)
\]
for all $f$ in $\mathcal{C}_{\lin}$ and $\mu$ in $\mathcal{P}_1$.  When $g$ satisfies Hypothesis \ref{B2}, we can apply $g(\langle f,\mu\rangle, \cdot)$ to both sides of the equation to see
\begin{equation}\label{dualityhalf}
g\big(\langle f,\mu \rangle,\,\mathcal{G}(\mu)\big) \geq \mathcal{F}(f).
\end{equation}
If $\{\varphi_i\}$ is a sequence of convex functions solving the Monge--Amp\`{e}re iteration, then Hypothesis \ref{B3} and equation (\ref{dualityhalf}) imply
\begin{equation}\label{introdecreasingsequence}
\mathcal{F}(\varphi_i) \geq g\Big(\Big\langle \varphi_{i+1},\,\scaleobj{.9}{\dfrac{\MA(\varphi_{i+1})}{\lambda(A)}}\Big\rangle,\,\mathcal{G}\Big(\scaleobj{.9}{\dfrac{\MA(\varphi_{i+1})}{\lambda(A)}}\Big)\Big) \geq \mathcal{F}(\varphi_{i+1}).
\end{equation}
Thus $\big\{\mathcal{F}(\varphi_i)\big\}_{i\in\mathbb{N}}$ and $\Big\{ g\Big(\,\Big\langle \varphi_{i+1},\,\scaleobj{.9}{\dfrac{\MA(\varphi_{i+1})}{\lambda(A)}}\Big\rangle,\,\mathcal{G}\Big(\scaleobj{.9}{\dfrac{\MA(\varphi_{i+1})}{\lambda(A)}}\Big)\Big)\Big\}_{i\in\mathbb{N}}$ are decreasing sequences.  The idea of the proof of Theorem \ref{introthm} is to show continuity for these two decreasing functionals, so any limit $\varphi_i \to \varphi$ will achieve equality between them, and then we show that equality is only achieved for solutions of equation (\ref{ourpde}).

\noindent \textit{Conditions (\ref{Aconditions}) on $A$} :

If $\varphi$ is a solution of equation (\ref{ourpde}), then for $i=1,\ldots,n$
\[
\int_A y_i\,d\lambda = \int_{\mathbb{R}^n} \dfrac{\partial \varphi}{\partial x_i} \,\det(\nabla^2 \varphi) \,d\lambda = \lambda(A) \int_{\mathbb{R}^n} \dfrac{\partial \varphi}{\partial x_i} \, \dfrac{h \circ \varphi}{\|h\circ \varphi\|_1}\,d\lambda = \lambda(A)\,\|h\circ \varphi\|_1^{-1} \int_{\mathbb{R}^n} \dfrac{\partial}{\partial x_i} (H\circ \varphi)\,d\lambda,
\]
where $H$ is an antiderivative of $h$.  The decay condition on $h$ implies, after an integration by parts, that the last integral is $0$, so the barycenter of $A$ must lie at the origin.   This shows the necessity of 
\[
\int_A y_i\,d\lambda=0\, \text{ for }i=1,\ldots,n.
\]
This condition is used in the proof of Lemma \ref{tau} which gives a lower bound for a convex function $f$ in terms of $\lambda(A)$ and $\tau$ when $f$ satisfies $\int_A f^*\,d\lambda = -\tau$.

\subsection{Proof outline}\label{outline}

First, we'll fix notation for the proof.  Assume $A \subset \mathbb{R}^n$ satisfies (\ref{Aconditions}).

\begin{center}
\begin{tabular}{c|c}
$\{\varphi_i\}$ & smooth, convex solutions to the normalized Monge--Amp\`{e}re iteration (\ref{MAiteration})\\[.5ex]
$\{a_i\}$ & points such that $\varphi_i(a_i) = \inf \varphi_i$\\[.5ex]
$\{\widetilde{\varphi}_i\}$ & the translated sequence $\widetilde{\varphi}_i(x) = \varphi_i(x+a_i)$ \\[.5ex]
\end{tabular}
\end{center}

\noindent \textit{Step 1: Uniform growth estimate}

We prove the translated sequence satisfies 
\[
\dfrac{\tau}{\lambda(A)} + r\,|x| \leq \widetilde{\varphi}_i(x) \leq C+R\,|x|,
\]
where $C$, $r$, and $R$ depend on $A$, $\tau$, $h$, and the initial function $\varphi_0$ starting the iteration. The lower bound only depends on $\int_A \varphi_i^*\,d\lambda = -\tau$.   The constant $C$ in the upper bound depends on $\{\varphi_i\}$ solving the Monge--Amp\`{e}re iteration, but the $R\,|x|$ term only comes from $\nabla \varphi_i(\mathbb{R}^n)=A \subset B_R$ for some constant $R$ depending only on $A$. 

\noindent \textit{Step 2: Subsequence convergence and subgradient limits}

The pointwise boundedness from step 1 implies any subsequence $\{\widetilde{\varphi}_{i'}\}$ will have a further subsequence 
\[
\widetilde{\varphi}_{i''} \to \varphi
\]
converging uniformly on compact subsets of $\mathbb{R}^n$ to a convex function $\varphi$ by compactness properties of convex functions.  We use $\int_A \widetilde{\varphi}_i^*\,d\lambda = -\tau$ to show $\varphi$ also has subgradient image $\partial \varphi(\mathbb{R}^n)=A$ up to a set of measure 0.

\noindent \textit{Step 3: Convergence of the Monge--Amp\`{e}re measures}

We prove that if a subsequence $\varphi_{i'}$ converges to $\varphi$ uniformly on compact sets, then their Monge--Amp\`{e}re measures converge in Wasserstein distance,
\[
 \MA(\widetilde{\varphi}_{i'}) \rightarrow_1 \MA(\varphi),
\]
which means $\MA(\widetilde{\varphi}_{i'})$ converges to $\MA(\varphi)$ weakly, and $\{\MA(\widetilde{\varphi}_i)\}_{i\in\mathbb{N}}$ satisfy the tightness condition
\[
\lim_{R\to\infty} \limsup_{i'\to\infty} \,\bigg\{\int_{|x|\geq R} |x|\,\MA(\varphi_{i'})\bigg\}=0.
\]
The weak convergence is a standard consequence of the uniform convergence on compact sets. The proof of the tightness condition relies on a uniform bound of the form $| a_{i+1}-a_i|\leq C$, which is essentially a bound on the rate the sequence $\{\varphi_i\}$ can drift horizontally.

\noindent \textit{Step 4: $\varphi$ minimizes $g\Big(\Big\langle\, \cdot\,, \,\scaleobj{.9}{\dfrac{\MA(\varphi)}{\lambda(A)}} \Big\rangle, \,\mathcal{F}(\cdot)\Big)$.}

Recall the definition of $\mathcal{G}$:
\[
\mathcal{G}(\mu) = \inf \big\{ \,g\big(\langle f,\mu\rangle,\,\mathcal{F}(f)\big) \mid f\in \mathcal{C}_{\lin}\, \big\}.
\]
We prove that $\varphi$ is the minimizer in the definition of $\mathcal{G}\Big(\scaleobj{0.9}{\dfrac{\MA(\varphi)}{\lambda(A)}}\Big)$. The proof relies upon the convergence $\MA(\widetilde{\varphi}_{i'}) \rightarrow_1 \MA(\varphi)$ and the upper semicontinuity of $\mathcal{G}$ with respect to this convergence.

\noindent \textit{Step 5: $\{\widetilde{\varphi}_i\}$ converges smoothly to $\varphi$, which solves the Monge--Amp\`{e}re equation (\ref{ourpde})}.

First we show $\varphi$, the uniform limit of a convergent subsequence $\{\widetilde{\varphi}_{i'}\}$, is a solution to equation (\ref{ourpde}). We take variations $\varphi_{\epsilon}$ and differentiate $g\Big(\Big\langle \varphi_\epsilon, \scaleobj{.9}{\dfrac{\MA(\varphi)}{\lambda(A)}}\Big\rangle,\,\mathcal{F}(\varphi_\epsilon)\Big)$ to show that equation (\ref{ourpde}) is the Euler-Lagrange equation for this functional, and $\varphi$ is a weak solution.  The smoothness of $\varphi$ results from the $C^{2,\alpha}$ estimates of Cafarelli and elliptic regularity.  The subgradient image $\partial \varphi(\mathbb{R}^n)=A$ up to a set of measure $0$ is then upgraded to $\nabla \varphi(\mathbb{R}^n)=A$.  We also prove $\int_A \varphi^*\,d\lambda=-\tau$.

To prove convergence of the whole sequence $\{\widetilde{\varphi}_i\}$, we show the limit $\varphi$ of every convergent subsequence $\widetilde{\varphi}_{i'}$ is unique.  $\varphi$ is a smooth solutions to equation (\ref{ourpde}), so it is unique up to translations by assumption.  The condition $\widetilde{\varphi}_i(0) = \inf \widetilde{\varphi}_i$ implies $\varphi(0) = \inf \varphi$, so it follows that $\varphi$ is unique.  Since every subsequence has a further subsequence which converges and the limits are the same, it follows that $\widetilde{\varphi}_i$ converges to $\varphi$. The smooth convergence is a consequence of Caffarelli's $C^{2,\alpha}$ estimates and elliptic regularity.

\subsection{Uniform growth estimate}

In this section we will prove step 1 of Subsection \ref{outline}:

\begin{proposition}\label{uniformbounds}
Assume $A$ satisfies (\ref{Aconditions}) and the Monge--Amp\`{e}re iteration (\ref{MAiteration}) satisfies Hypotheses \ref{hypB}.  Then 
\begin{equation}\label{propvarphibounds}
\dfrac{\tau}{\lambda(A)} + r\,|x| \leq \widetilde{\varphi}_i(x) \leq C + R\,|x|.
\end{equation}
\end{proposition}

Before proving Proposition \ref{uniformbounds} we'll prove a lemma which shows the bounds (\ref{propvarphibounds}) imply bounds on $\|h \circ \widetilde{\varphi}_i\|_1$ and $\|H\circ \widetilde{\varphi}_i\|_1$ which will be important for subsequent steps of the proof of Theorem \ref{introthm}.

\begin{lemma}\label{Hfbound}
Assume $h$ satisfy Hypothesis \ref{B1}, and let $H$ be defined by equation (\ref{Fdef}).  Let $f:\mathbb{R}^n \to (a,\infty)$ be a convex function.  If $f(x) \geq a + r\,|x|$ for $r>0$, then there exists a constant $c$ depending on $a$, $h$, and $r$ such that
\begin{equation}\label{hfupper}
\|h\circ f\|_1 \leq c \hspace{5mm}\text{and}\hspace{5mm} \|H \circ f\|_1 \leq c.
\end{equation}
If $f(x) \leq C + R\,|x|$ for $R>0$, then there exists a constant $c$ depending on $C$, $h$, and $R$ such that
\begin{equation}\label{hflower}
\|h\circ f\|_1 \geq c>0 \hspace{5mm}\text{and}\hspace{5mm} \|H \circ f\|_1 \geq c>0.
\end{equation}

\begin{proof}
By Hypothesis \ref{B1}, $h$ is positive and decreasing, and there exists $\rho >0$ such that $h(t) \leq C\,t^{-(n+p+1)}$ when $t\geq \rho$.  This implies $H(t) = \int_t^\infty h\,d\lambda$ is positive and decreasing, and $H(t) \leq C'\,t^{-(n+p)}$ when $t\geq \rho$. 

Firstly, assume $f(x) \geq a + r\,|x|$.  Since $h$ and $H$ are decreasing, it follows that
\[
h\circ f(x) \leq h\big(a + r\,|x|\big) \hspace{5mm}\text{and}\hspace{5mm} H\circ f(x) \leq H\big(a + r\,|x|\big).
\]
Thus we can estimate
\begin{align*}
\|h\circ f\|_1 &\leq  n\,\omega_n \int_0^\infty h(a + r\,s)\,s^{n-1}\,ds \\
&\leq \omega_n \,\Big(\dfrac{\rho - a}{r}\Big)^n\,h(\tau) + n\,\omega_n \int_{(\rho-a)/r}^\infty C\,(a + r\,s)^{-(n+p+1)}\,s^{n-1}\,ds.
\end{align*}
The last integral converges because the integrand is less than $C'\,s^{-(2+p)}$ for large $s$ and $p>0$.  This provides the upper bound for $\|h\circ f\|_1$ depending only on $h$, $a$, and $r$.  The bound for $\|H\circ f\|_1$ is identical, except the last integrand will be less than $C' \,s^{(n+p)}$ for large $s$.  Thus we can see that $p>0$ is optimal for the convergence of $\|H \circ f\|_1$.  

Secondly, assume $f(x) \leq C + R\,|x|$.  Since $h$ and $H$ are decreasing, it follows that
\[
h\circ f(x) \geq h\big(C + R\,|x|\big) \hspace{3mm}\text{and}\hspace{3mm} H\circ f(x) \geq H(C+R\,|x|).
\]
Thus we can estimate
\begin{align*}
    \|h\circ f\|_1 &\geq n\,\omega_n \int_0^\infty h(C+R\,s)\,s^{n+1}\,ds\\
    &\geq  \omega_n \,\Big(\dfrac{\rho - C}{R}\Big)^n\,h\Big(\dfrac{\rho-C}{R}\Big) >0.
\end{align*}
The lower bound for $\|H\circ f\|_1$ is completely analagous.
\end{proof}

\end{lemma}

Now we return to the proof of Proposition \ref{uniformbounds}. We begin with two basic lemmas about the Legendre transform and translated sequences.

\begin{lemma}\label{tildeprops}
Assume $A$ satisfies (\ref{Aconditions}).  Let $f$ be a convex function such that $\int_A \varphi^*\,d\lambda = -\tau$. Assume $f(a_i) = \inf_{x\in\mathbb{R}^n} \{f(x)\}$, and let $\widetilde{f}(x) = f(x+a_i)$ be a translation of $f$.  Then
\[
\inf_{x\in\mathbb{R}^n} \big\{\widetilde{f}(x)\big\} = \widetilde{f}(0)\,,\hspace{5mm} \inf_{y\in A} \big\{\widetilde{f}^*(y)\big\}  = \widetilde{f}^*(0) = -\widetilde{f}(0)\,,\,\text{and}\hspace{5mm} \int_A \widetilde{f}^*\,d\lambda = -\tau.
\]
\end{lemma}

\begin{proof}
By definition of $\widetilde{f}$,
\[
\widetilde{f}(0) = f(a_i) = \inf_{x\in\mathbb{R}^n} \{f(x)\} = \inf_{x\in\mathbb{R}^n}\big\{ \widetilde{f}(x)\big\}.
\]
By the definition of the Legendre transform,
\[  
\widetilde{f}^*(0) = \sup_{x\in\mathbb{R}^n} \big\{-\widetilde{f}(x)\big\} =- \widetilde{f}(0).
\]
In order to show $\widetilde{f}^*(0) = \inf \big\{\widetilde{f}^*\big\}$ we note
\[
\widetilde{f}^*(y) = \sup_{x\in\mathbb{R}^n} \big\{ \langle x,y\rangle - \widetilde{f}(x) \big\} \geq -\widetilde{f}(0)=\widetilde{f}^*(0).
\]

The Legendre transforms of $\widetilde{f}$  and $f$ are related by
\[
\widetilde{f}^*(y) = \sup_{x\in\mathbb{R}^n} \big\{\langle x,y\rangle - f(x+a_i) \big\}= \sup_{x\in\mathbb{R}^n} \big\{ \langle x-a_i,y\rangle - f(x)\big\} = f^*(y)-\langle a_i,y\rangle,
\]
which implies
\[
\int_A \widetilde{f}^*\,d\lambda = \int_A \big(f^*(y) -\langle a_i,y\rangle\big)\,d\lambda = -\tau,
\]
because the barycenter of $A$ is at the origin. 
\end{proof}

\begin{lemma}\label{tau}
Assume $A$ satisfies (\ref{Aconditions}), and let $f$ be a convex function on $\mathbb{R}^n$. If $\int_A f^*\,d\lambda=-\tau<\infty$, then $f(x) \geq \tau/\lambda(A)$ for all $x$.
\end{lemma}
\begin{proof}
If there were a point $x_0$ in $\mathbb{R}^n$ where $f(x_0) <\tau/\lambda(A)$, then $f^*(y) = \sup \big\{\,\langle x,y \rangle -f(x) \mid x \in \mathbb{R}^n\, \big\} > \langle x_0,y\rangle -\tau/\lambda(A)$.  Since the barycenter of $A$ lies at the origin, 
\begin{equation}\label{tauequation}
\int_A f^*(y)\,d\lambda > \int_A \bigg(\langle x_0,y\rangle -\dfrac{\tau}{\lambda(A)} \bigg)\,d\lambda = -\tau,
\end{equation}
which contradicts $\int_A f^*\,d\lambda=-\tau$.
\end{proof}

\noindent \textit{Substep 1:} $\dfrac{\tau}{\lambda(A)} + r\,|x| \leq \widetilde{\varphi}_i(x)$

We recall the convex analysis Lemma 2.6 from Klartag \cite{Klartag}.

\begin{lemma}\label{Klartaglemma}
Let $A\subset \mathbb{R}^n$ be convex with the origin in its interior.  There exists $r>0$, depending on $A$, with the following property: Let $\psi: A \to \mathbb{R}$ be convex and integrable. Assume $\psi(0)=\inf \psi$, and $\int_A \psi \,d\lambda \leq 0$.  Then for $y \in A$, 
\[
\psi(y) \leq \psi(0)/2 \hspace{5mm}\text{when }|y|\leq r.
\]
\end{lemma}
\noindent We illustrate this lemma below for $n=1$.

\begin{center}
\includegraphics{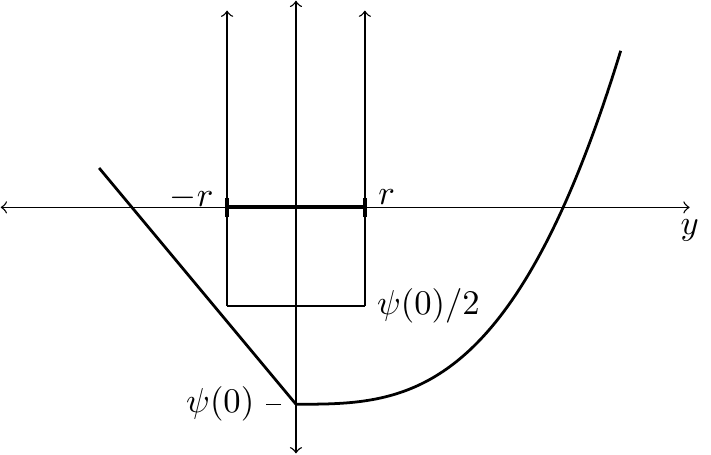}
\end{center}

The $r$ in Lemma \ref{Klartaglemma} will become the $r$ in $\dfrac{\tau}{\lambda(A)} + r\,|x| \leq \widetilde{\varphi}_i(x)$.

\begin{lemma}\label{lowerbound}
Assume $A$ satisfies (\ref{Aconditions}).  Let $f$ be a convex function such that $\inf \{f \}= f(0)$ and $\int_A f^*\,d\lambda = -\tau$.  Then there exists an $r$ depending on $A$ such that
\begin{equation}\label{lowerboundeq}
f(x) \geq \dfrac{1}{2}\Big(f(0)+\scaleobj{.9}{\dfrac{\tau}{\lambda(A)}}\Big) +r\,|x| \geq \tau + r\,|x|.
\end{equation}
\end{lemma}
\begin{proof}
Adding $\tau/\lambda(A)$ to $f^*$ implies
\[
\int_A \bigg(f^* + \dfrac{\tau}{\lambda(A)}\bigg)\,d\lambda = 0.
\]
Thus $f^*+ \tau/\lambda(A)$ satisfies all the hypotheses of Lemma \ref{Klartaglemma}, so there exists an $r$ depending only on $A$ such that $f^*(y) + \tau/\lambda(A) \leq \dfrac{1}{2}\big(f^*(0) + \tau/\lambda(A)\big)$ when $|y|\leq r$. Equivalently,
\begin{equation}\label{aa}
f^*(y) \leq -\dfrac{1}{2}\Big(f(0) + \scaleobj{.9}{\dfrac{\tau}{\lambda(A)}}\Big) + \mathbb{1}_{B_r},
\end{equation}
where
\[
\mathbb{1}_D(y) = \begin{cases} 0 &y\in D\\ \infty & y \notin D \end{cases}
\]
is the convex indicator function of a convex set $D$.
The order reversing property of the Legendre transform implies 
\begin{align*}
f(x) &= \sup_{y \in A}\, \big\{\langle x,y\rangle - f^*(y) \big\} \geq \sup_{y\in\mathbb{R}^n} \Big\{ \langle x,y\rangle - \Big(-\dfrac{1}{2}\Big(f(0) + \scaleobj{.9}{\dfrac{\tau}{\lambda(A)}}\Big) + \mathbb{1}_{B_r}\Big) \Big\} \\ 
&=\dfrac{1}{2}\Big(f(0)+\scaleobj{.9}{\dfrac{\tau}{\lambda(A)}}\Big)+ \sup_{|y|\leq r} \big\{\langle x,y\rangle \big\}=  \dfrac{1}{2}\Big(f(0)+\scaleobj{.9}{\dfrac{\tau}{\lambda(A)}}\Big) + r\,|x|,
\end{align*}
which is the first inequality in equation (\ref{lowerboundeq}). Lemma \ref{tau} implies $f(0)\geq \tau/\lambda(A)$, which yields the second inequality in equation (\ref{lowerboundeq}).
\end{proof}

Lemma \ref{tildeprops} implies $\inf \{\widetilde{\varphi}_i\} = \widetilde{\varphi}_i(0)$ and $\int_A \widetilde{\varphi}_i^*\,d\lambda=-\tau$, so we can apply Lemma \ref{lowerbound} to $\widetilde{\varphi}_i$ to prove substep 1.

\noindent \textit{Substep 2:} $C_1 \leq \mathcal{F}(\varphi_i) = \mathcal{F}(\widetilde{\varphi}_i) \leq C_2$ 

Before proving the upper bound in equation (\ref{propvarphibounds}), we must prove $\mathcal{F}(\varphi_i)$ is bounded.  $\mathcal{F}$ is translation invariant, so $\mathcal{F}(\widetilde{\varphi}_i) = \mathcal{F}(\varphi_i)$, and it is sufficient to show the boundedness of either.  First we show $\mathcal{F}$ is bounded below.

\begin{lemma}\label{Fboundedbelow}
Assume $A$ satisfies (\ref{Aconditions}) and $h$ satisfies Hypothesis \ref{B1}.  Let $\mathcal{F}$ be the functional defined in (\ref{Fdef}). There exists a constant $C$ depending on $A$, $\tau$, $p$, and $h$ such that if $f$ is convex and $\int_A f^*\,d\lambda=-\tau$, then $\mathcal{F}(f) \geq C$.
\end{lemma}
\begin{proof}
Recall the definition (\ref{Fdef}) of $\mathcal{F}$:
\[
\mathcal{F}(f) = H^{-1}\big(\|H\circ f\|_1\big)\hspace{3mm}\text{where}\hspace{3mm}H(t) = \int_t^\infty h\,d\lambda.
\]
By Hypothesis \ref{B1}, $h$ is positive, so $H$ is a strictly decreasing function.  Thus $H^{-1}$ is strictly decreasing as well, so if we show $\| H \circ f\|_1 \leq C$, then the lower bound for $\mathcal{F}(f)$ will follow. 

$\|H\circ f\|_1$ is invariant under translations of $f$ so we can assume without loss of generality that $\inf \{ f \} = f(0)$.  Lemma \ref{tildeprops} implies $\int_A f^*\,d\lambda = -\tau$ is unchanged by translations.  Lemma \ref{lowerbound} implies the lower bound
\[
f(x) \geq \scaleobj{.9}{\dfrac{\tau}{\lambda(A)}} + r\,|x|.
\]
Then, Lemma \ref{Hfbound} implies $\|H \circ f\|_1 \leq C$ as desired.  
\end{proof}

Te hypotheses of Lemma \ref{Fboundedbelow} are satisfied by $\varphi_i$, so $\mathcal{F}(\widetilde{\varphi}_i) = \mathcal{F}(\varphi_i) \geq C$.  Now we prove the upper bound for $\mathcal{F}$ along the Monge--Amp\`{e}re iteration.

\begin{lemma}\label{functionallimit}
Assume the Monge--Amp\`{e}re iteration satisfies Hypotheses \ref{hypB}.  Then
\begin{equation}\label{functionalinequality}
\mathcal{F}(\varphi_{i-1}) \geq g\Big(\Big\langle \varphi_i,\,\scaleobj{.9}{\dfrac{\MA(\varphi_i)}{\lambda(A)}}\Big\rangle,\,\mathcal{G}\Big(\scaleobj{.9}{\dfrac{\MA(\varphi_i)}{\lambda(A)}}\Big)\Big) \geq \mathcal{F}(\varphi_i),
\end{equation}
and there exists $\beta$ finite such that 
\begin{equation}\label{functionallimiteq}
\beta = \lim_{i\to\infty} \mathcal{F}(\varphi_i) = \lim_{i\to\infty} g\Big(\Big\langle \varphi_i,\,\scaleobj{.9}{\dfrac{\MA(\varphi_i)}{\lambda(A)}}\Big\rangle,\,\mathcal{G}\Big(\scaleobj{.9}{\dfrac{\MA(\varphi_i)}{\lambda(A)}}\Big)\Big)
\end{equation}
\end{lemma}

\begin{proof}
The first inequality in equation (\ref{functionalinequality}) is exactly hypothesis \ref{B3}.  Since $\varphi_i$ is a convex function with $\nabla \varphi_i(\mathbb{R}^n)=A$ it follows that $\varphi_i \in \mathcal{C}_{\lin}$, defined in (\ref{classC}).  Thus by the definition of $\mathcal{G}$,
\[
\mathcal{G}\Big(\scaleobj{.9}{\dfrac{\MA(\varphi_i)}{\lambda(A)}}\Big) \leq g\Big(\Big\langle \varphi_i,\,\scaleobj{.9}{\dfrac{\MA(\varphi_i)}{\lambda(A)}} \Big\rangle,\,\mathcal{F}(\varphi_i)\Big).
\]
After applying $g\Big(\Big\langle \varphi_i,\scaleobj{.9}{\dfrac{\MA(\varphi_i)}{\lambda(A)}} \Big\rangle,\cdot\,\Big)$ to both sides of the equation, Hypothesis \ref{B2}, that $g(s,t)$ is a decreasing involution, implies
\[
g\Big(\Big\langle \varphi_i,\scaleobj{.9}{\dfrac{\MA(\varphi_i)}{\lambda(A)}} \Big\rangle,\,\mathcal{G}\Big(\scaleobj{.9}{\dfrac{\MA(\varphi_i)}{\lambda(A)}}\Big)\Big) \geq \mathcal{F}(\varphi_i),
\]
proving the second inequality in equation (\ref{functionalinequality}).  The two inequalities show that 
\[
\big\{\mathcal{F}(\varphi_i)\big\} \hspace{3mm}\text{and}\hspace{3mm} \Big\{g\Big(\Big\langle \varphi_i,\scaleobj{.9}{\dfrac{\MA(\varphi_i)}{\lambda(A)}}\Big\rangle, \,\mathcal{G}\Big(\scaleobj{.9}{\dfrac{\MA(\varphi_i)}{\lambda(A)}}\Big)\Big)\Big\}
\]
are decreasing sequences.  Lemma \ref{Fboundedbelow} implies $\mathcal{F}(\varphi_i)$ is bounded below, so both decreasing sequences converge to a common finite value $\beta$ as in equation (\ref{functionallimiteq}).
\end{proof}

The two previous lemmas imply $C_1 \leq \mathcal{F}(\varphi_i) = \mathcal{F}(\widetilde{\varphi}_i) \leq C_2$.

\noindent \textit{Substep 3:} $\widetilde{\varphi}_i(x) \leq C+R\,|x|$

Now we use the boundedness of $\mathcal{F}(\widetilde{\varphi}_i)$ to prove the upper bound for $\widetilde{\varphi}_i$.

\begin{lemma}\label{upperbound}
Assume $A$ satisfies conditions (\ref{Aconditions}) and the Monge--Amp\`{e}re iteration satisfies Hypotheses \ref{hypB}.  Then, 
\[
\widetilde{\varphi}_i(x) \leq C+R\,|x|.
\]
for constants $C$ and $R$ depending on $\tau$, $A$, and $h$.
\end{lemma}
\begin{proof}
 $\mathcal{F}(\widetilde{\varphi}_i) = H^{-1}\big(\|H\circ \widetilde{\varphi}_i\|_1\big) \leq c$, and $H$ is positive and decreasing, so
\[
0<H(c)\leq\|H\circ \widetilde{\varphi}_i\|_1.
\]
We will first show that there exists $C$, independent of $i$, such that $\widetilde{\varphi}_i(0)\leq C$.  Lemma \ref{lowerbound} implies
\[
\widetilde{\varphi}_i(x) \geq \dfrac{1}{2}\Big(\widetilde{\varphi}_i(0)+\dfrac{\tau}{\lambda(A)}\Big) + r\,|x|.
\]
As in the proof of Lemma \ref{Hfbound}, the bound for $h$ from hypothesis \ref{B1} shows there exists $C$, $R$, and $p>0$ such that
\[
\|H\circ \widetilde{\varphi}_i\|_1 \leq \int_{B_R} H\bigg(\dfrac{1}{2}\Big(\widetilde{\varphi}_i(0)+\scaleobj{.9}{\dfrac{\tau}{\lambda(A)}}\Big)\bigg) \,d\lambda + C\int_{\mathbb{R}^n\setminus B_R} \bigg(\dfrac{1}{2}\Big(\widetilde{\varphi}_i(0)+\scaleobj{.9}{\dfrac{\tau}{\lambda(A)}}\Big) + r\,|x|\bigg)^{-(n+p)} \,d\lambda.
\]
Since $H(t) = \int_t^\infty h\,d\lambda$, $\lim_{t\to\infty}H(t) = 0$ so
\[
\lim_{\widetilde{\varphi}_i(0)\to \infty } \int_{B_R} H\bigg(\dfrac{1}{2}\Big(\widetilde{\varphi}_i(0)+\scaleobj{.9}{\dfrac{\tau}{\lambda(A)}}\Big)\bigg) \,d\lambda = 0.
\]
Likewise, $\Big(\dfrac{1}{2}\Big(\widetilde{\varphi}_i(0)+\scaleobj{.9}{\dfrac{\tau}{\lambda(A)}}\Big) + r\,|x|\Big)^{-(n+p)}$ converges to the zero function pointwise as $\widetilde{\varphi}_i(0)\to\infty$, and it is uniformly bounded above by $\Big(\scaleobj{.9}{\dfrac{\tau}{\lambda(A)}} + r\,|x|\Big)^{-(n+p)}$, so
\[
\lim_{\widetilde{\varphi}_i(0)\to \infty} C\int_{\mathbb{R}^n\setminus B_R} \bigg(\dfrac{1}{2}\Big(\widetilde{\varphi}_i(0)+\scaleobj{.9}{\dfrac{\tau}{\lambda(A)}}\Big) + r\,|x|\bigg)^{-(n+p)} \,d\lambda = 0
\]
by the dominated convergence theorem.  But we have the positive lower bound
\[
0<H(c) \leq \int_{B_R} H\bigg(\dfrac{1}{2}\Big(\widetilde{\varphi}_i(0)+\scaleobj{.9}{\dfrac{\tau}{\lambda(A)}}\Big)\bigg) \,d\lambda + C\int_{\mathbb{R}^n\setminus B_R} \bigg(\dfrac{1}{2}\Big(\widetilde{\varphi}_i(0)+\scaleobj{.9}{\dfrac{\tau}{\lambda(A)}}\Big) + r\,|x|\bigg)^{-(n+p)} \,d\lambda,
\]
so the integrals cannot go to $0$, and there must be some upper bound $C\geq\widetilde{\varphi}_i(0)$. And if $R>0$ is a constant satisfying $A\subset B_R$, then $| \nabla \widetilde{\varphi}_i(x)|\leq R$, and by integrating along lines from the origin it follows that
\[
\widetilde{\varphi}_i(x) \leq C+R\,|x|.
\]
\end{proof}

The three previous substeps complete the proof of Proposition \ref{uniformbounds}.

\subsection{Subsequence convergence and subgradient limits}

In this section we will prove Step 2 of Subsection \ref{outline}:

\begin{proposition}\label{subsequence}
Assume $A$ satisfies (\ref{Aconditions}) and the Monge--Amp\`{e}re iteration (\ref{MAiteration}) satisfies Hypotheses \ref{hypB}. Then every subsequence $\{\widetilde{\varphi}_{i'}\}$ of the translated sequence $\{\widetilde{\varphi}_i\}$ has a further subsequence $\{\widetilde{\varphi}_{i''}\}$ which converges uniformly on compact sets to some convex function $\varphi$ for which $\Int \partial \varphi(\mathbb{R}^n)=A$.
\end{proposition}

The subsequence convergence is a simple corollary of the compactness properties for locally, uniformly bounded convex functions.

\begin{lemma}
Assume $A$ satisfies (\ref{Aconditions}) and the Monge--Amp\`{e}re iteration satisfies Hypotheses \ref{hypB}.  Then every subsequence $\{\widetilde{\varphi}_{i'}\}$ has a further subsequence, uniformly convergent on compact sets.
\end{lemma}
\begin{proof}
By Proposition \ref{uniformbounds},
\[
\scaleobj{.9}{\dfrac{\tau}{\lambda(A)}} \leq \widetilde{\varphi}_{i'}(x) \leq C+R\,|x|.
\]
Thus, for every fixed $x_0\in\mathbb{R}^n$, the set $\big\{\widetilde{\varphi}_{i'}(x_0)\big\} \subset \big[\tau/\lambda(A),\,C+R\,|x_0|\,\big]$ is bounded.  By Rockafellar \cite[Theorem~10.9]{Rockafellar} there exists a further subsequence $\{\widetilde{\varphi}_{i''}\}$ which converges to a convex function $\varphi$ uniformly on compact subsets of $\mathbb{R}^n$.
\end{proof}

The convergence of the subgradients relies upon the fact that $\int_A \widetilde{\varphi}_i^*\,d\lambda=-\tau$ for each $i$.  

In general, if we only know that convex functions $f_i$ converge to $f$ uniformly on compact sets, we may have $\partial f(\mathbb{R}^n) \subsetneq \partial f_i (\mathbb{R}^n)$.  To see this, let $f$ be any convex function such that $\partial f(\mathbb{R}^n) = A$ where $A \subset B_R$.  Define
\[
f_i(x) = \max\{\, f(x), \,R\,|x| -i\, \}.
\]
Then $f_i$ converges to $f$ uniformly on compact subsets, but $\partial f_i(\mathbb{R}^n) = B_R \supsetneq A$ for all $i$.  

The proof of the subgradient limit relies upon a convex analysis lemma which is a generalization of Lemma \ref{Klartaglemma}.  Klartag proved that if $\psi$ is a convex function such that $\int_A \psi \leq 0$, then there is an upper bound for $\psi$, depending on $A$ and its minimum value, in a small ball around its minimum.  This Lemma extends the upper bound to any open set away from the boundary of $A$.

\begin{lemma}\label{Klartagvar}
Let $A\subset \mathbb{R}^n$ be convex.  For each $\epsilon>0$ define $A_\epsilon = \{\,y\mid B_\epsilon(y)\subset A\,\}$.  There exists $C>0$ depending on $\epsilon$ and $\lambda(A)$ with the following property: Let $\psi: A \to \mathbb{R}$ be convex, and assume $\int_A \psi \,d\lambda \leq 0 $.  Then
\[
\psi(y) \leq -C\,\inf_A \{\psi\} \hspace{5mm}\text{when }y\in A_\epsilon.
\]
\end{lemma}
\begin{proof}
We can make the simplifying assumption that $\lambda(A)=1$.  Define $C$ to be any constant such that $C+1 > 1/V$ where $V=\omega_n \epsilon^n/2$ is the volume of a half-ball of radius $\epsilon$. Assume by contradiction to the conclusion of the lemma
\[
\big\{\,y\in A \mid \psi(y)\leq -C\,\inf \{\psi\} \,\big\} \text{ does not contain }A_\epsilon.
\]
Since $\{\,y\mid \psi(y)\leq -C\,\inf\psi\,\}$ is the sublevel set of a convex function, it is convex.  Let $y_0$ be any point in the portion of the boundary of $\{\,y\mid \psi(y)\leq -C\,\inf\psi\,\}$ which intersects $A_\epsilon$.  Then the supporting halfspace at $y_0$ lies outside $\{\,y\mid \psi(y)\leq -C\,\inf\psi\,\}$, meaning there exists an outward normal $v$ such that 
\[
\Big(A\cap \big\{\,y \mid \langle y-y_0,v\rangle \geq 0\,\big\}\Big) \subset \{\,y \in A \mid \psi(y)\geq -C\,\inf\psi\,\}.
\]
Since $y_0\in A_\epsilon$ we can intersect the first set with $B_\epsilon(y_0)$ to get
\[
B:=\Big(B_\epsilon(y_0)\cap\big\{\,y \mid \langle y-y_0,v\rangle \geq 0\,\big\} \Big) \subset \{\,y \in A \mid \psi(y)\geq -C\,\inf\psi\,\}.
\]

\begin{center}
\includegraphics{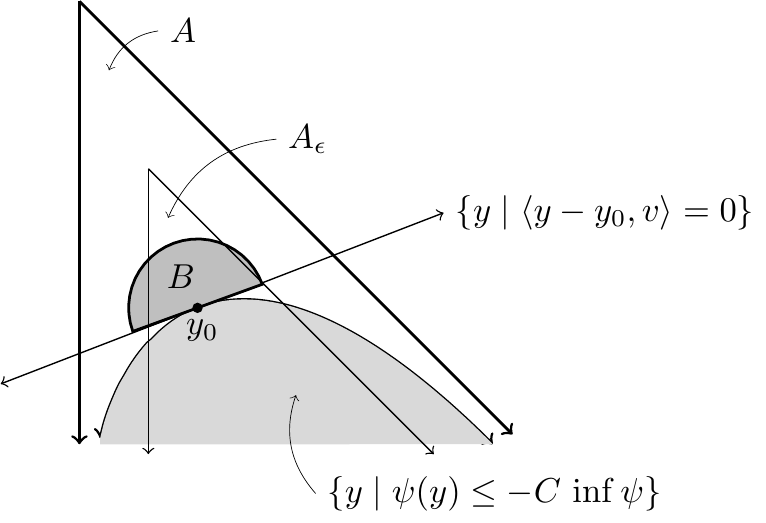}
\end{center}

Then, since $\lambda(A)=1$ 
\[
\int_A \psi(y)\,d\lambda \geq (1-V)(\inf \psi)+V(-C\,\inf\psi) = (-\inf\psi)\,\big(-1+(C+1)\,V\big).
\]
$\int_A \psi\,d\lambda\leq 0$ so $(-\inf\psi)\geq 0$.  We chose $C$ so that $C+1 > 1/V$, and it follows that $\int_A \psi\,d\lambda >0$, which is a contradiction since $\int_A \psi\,d\lambda\leq0$.
\end{proof}

\begin{lemma}\label{subgradientlimit}
Assume $A$ satisfies (\ref{Aconditions}) and the Monge--Amp\`{e}re iteration satisfies Hypotheses \ref{hypB}.  Let $\{\widetilde{\varphi}_{i'}\}$ be any subsequence which converges to $\varphi$ uniformly on compact subsets of $\mathbb{R}^n$.  Then $\Int\,\partial \varphi(\mathbb{R}^n)=A$.
\end{lemma}
\begin{proof}
First we find a refined upper bound for $\widetilde{\varphi}_i$.  Let
\[
\mathbb{1}_D = \begin{cases} 0 &y\in D\\ \infty & y \notin D \end{cases}
\]
be the convex indicator function of a convex set $D$.  Then $\mathbb{1}_D^*(x) = \sup \{\,\langle x,y\rangle\mid y\in D\,\}$ is the cone emanating from the origin with subgradient image equal to  $\overline{D}$.

By Lemma \ref{tildeprops} $\inf \widetilde{\varphi}_i^* = \widetilde{\varphi}_i^*(0) = -\widetilde{\varphi}_i(0)$, so
\[
\widetilde{\varphi}_i^* \geq \mathbb{1}_{A} +\widetilde{\varphi}_i^*(0) = \mathbb{1}_{A} -\widetilde{\varphi}_i(0).
\]
The order reversing property of the Legendre transform implies
\[
\widetilde{\varphi}_i(x) \leq \widetilde{\varphi}_i(0) + \mathbb{1}_A^*(x).
\]
Proposition \ref{uniformbounds} implies $\big\{\widetilde{\varphi}_i(0)\big\}$ is bounded, so there is a constant $C$ such that 
\begin{equation}\label{a}
\widetilde{\varphi}_i(x) \leq C + \mathbb{1}_A^*(x) \hspace{5mm}\text{for all }i.
\end{equation}

Next we find a refined lower bound for $\widetilde{\varphi}_i$.  The addition of a constant to every $\widetilde{\varphi}_i$ does not affect the convergence or the subgradients, so we can assume without loss of generality that $\int_A \widetilde{\varphi}_{i'}^* \,d\lambda = -\tau \leq 0$.  Thus we can apply Lemma \ref{Klartagvar} to show that for every $\epsilon>0$ there exists a constant $C_{\epsilon}$ such that
\[
\widetilde{\varphi}_i^*(y) \leq \mathbb{1}_{A_\epsilon}(y) - C_{\epsilon}\,\big(\inf \widetilde{\varphi}_i^*\big) = \mathbb{1}_{A_\epsilon}(y) +C_{\epsilon}\,\widetilde{\varphi}_i(0).
\]
Since $\big\{\widetilde{\varphi}_i(0)\big\}$ is bounded,
\[
\widetilde{\varphi}_i^*(y) \leq \mathbb{1}_{A_\epsilon}(y) +C_{\epsilon}.
\]
By the order reversing properties of the Legendre transform
\begin{equation}\label{b}
\widetilde{\varphi}_i(x) \geq \mathbb{1}_{A_\epsilon}^*(x) -C_{\epsilon}.
\end{equation}
Since the bounds (\ref{a}) and (\ref{b}) are independent of $i$, it follows that
\[
\mathbb{1}_{A_\epsilon}^*(x) -C_{\epsilon} \leq \varphi(x) \leq C + \mathbb{1}_A^*(x) \hspace{5mm}\text{for all }i.
\]
Thus for every $\epsilon>0$, $A_\epsilon \subset \Int\,\partial \varphi(\mathbb{R}^n) \subset A$.  Letting $\epsilon \to 0$ finishes the proof.
\end{proof}

\subsection{Convergence of the Monge--Amp\`{e}re measures}

In this section we will prove step 3 of Subsection \ref{outline}.  

Let $\{\mu_i\}$ be measures in $\mathcal{P}_1$.  Recall, $\mu_i$ converges to $\mu$ \textit{weakly}, denoted $\mu_i \Rightarrow \mu$, if
\[
\int_{\mathbb{R}^n} f\,d\mu_i \to \int_{\mathbb{R}^n} f\,d\mu \text{ for all }f\in C_b,
\]
the space of continuous, bounded functions. We will take the following proposition \cite[pg.\,96]{Villani1} as equivalent definitions of convergence in Wasserstein distance, which we denote $\mu_i \to_1 \mu$. 

\begin{proposition}\label{Wassequivalent}
Let $\{\mu_i\}_{i=1}^\infty$ and $\mu$ be probability measures in $\mathcal{P}_1$. The following are equivalent:
\begin{enumerate}[label = (\roman*)]
\item $\mu_i \to_1 \mu$.
\item $\mu_i \Rightarrow \mu$ and 
\[
\lim_{R\to\infty}\limsup_{i\to\infty} \,\bigg\{\int_{\big\{|x|\geq R\big\}} |x| \,d\mu_i\,\bigg\} =0.
\]
\item For all continuous functions $f$ with $|f(x)| \leq C\,\big(1+|x|\big)$,
\[
\int_{\mathbb{R}^n} f\,d\mu_i \to \int_{\mathbb{R}^n} f\,d\mu.
\]
\end{enumerate}
\end{proposition}

The goal of this section is to prove the following proposition:

\begin{proposition}\label{sd}
Assume $A$ satisfies (\ref{Aconditions}) and the Monge--Amp\`{e}re iteration (\ref{MAiteration}) satisfies Hypotheses \ref{hypB}.  Assume a subsequence $\{\widetilde{\varphi}_{i'}\}$ converges to $\varphi$ uniformly on compact sets. Then 
\begin{equation}\label{Wasslimit}
\scaleobj{.9}{\dfrac{\MA(\widetilde{\varphi}_{i'})}{\lambda(A)}} \rightarrow_1 \scaleobj{.9}{\dfrac{\MA(\varphi)}{\lambda(A)}}.
\end{equation}
\end{proposition}

For this step, we will simplify notation by assuming $\lambda(A)=1$.  The same proofs hold in general by replacing $\MA(\varphi)$ by $\scaleobj{.9}{\dfrac{\MA(\varphi)}{\lambda(A)}}$.  By Proposition \ref{Wassequivalent}, equation (\ref{Wasslimit}) is equivalent to
\[
\MA(\widetilde{\varphi}_{i'}) \Rightarrow \mu \hspace{3mm}\text{and}\hspace{3mm} \lim_{R\to\infty} \limsup_{i'\to\infty} \,\bigg\{\int_{\{|x|\geq R\}}|x|\,\MA(\widetilde{\varphi}_{i'}) \bigg\}=0.
\]
The weak convergence $\MA(\widetilde{\varphi}_{i'}) \Rightarrow \MA(\varphi)$ is a consequence of $\varphi_{i'}$ converging to $\varphi$ uniformly on compact sets by Trudinger--Wang \cite[Lemma 2.2]{TW}. 

The second condition, which is a tightness condition on the measures, relies upon a uniform bound $|a_{i+1}-a_i|\leq C$.  When $h(t)=e^{-t}$ this bound is the analogue of Darvas--Rubinstein \cite[Theorem~5.1]{Darvas} which establishes the bound on the distance between consecutive automorphisms from the sequence $\{g_i\}_{i\in\mathbb{N}}$ which make the K\"{a}hler--Ricci iteration converge.  First we need a convex analysis lemma.

\begin{lemma}\label{subgradientinclusion}
Let $f$ be a convex function on $B_r$, and let $a<b$ be constants such that $f(0)=a$ and $f(x)\geq b$ for $|x|=r$.  Then
\[
B_{(b-a)/r} \subset \partial f(B_r) .
\]
\end{lemma}
\begin{proof}
It is enough to show that if $|y|=(b-a)/r$, then $y\in \partial f(B_r)$.  $|y|=(b-a)/r$ implies
\[
a+\langle y,x\rangle \leq b\leq f(x) \hspace{5mm}\text{for }|x|=r.
\]
Let
\[
c=\sup_{x\in B_r} \big\{a + \langle y,x\rangle -f(x)\big\}.
\]
Considering $0$ in the supremum shows $c\geq 0$.  We claim
\begin{equation}\label{hyperplane}
a+\langle y,x\rangle -c \leq f(x)
\end{equation}
is a supporting hyperplane for $f$ for some $x_1$ such that $|x_1|<r$.  If $c=0$, then $x_1=0$ gives equality in (\ref{hyperplane}), proving the lemma.  If $c>0$ then let $x_1$ be any point in $B_r$ attaining the supremum defining $c$.  It remains to show $|x_1|<r$.  If $|x_1|=r$ then
\[
0 \geq a+ \langle y,x_1\rangle - f(x_1) = c>0,
\]
so $|x_1|$ must be less than $r$, and $y \in \partial f(x_1)$.
\end{proof}

\begin{lemma}\label{aibound}
Assume $A$ satisfies (\ref{Aconditions}) and the Monge--Amp\`{e}re iteration (\ref{MAiteration}) satisfies Hypotheses \ref{hypB}.  Then there exists  $C>0$ such that $|a_{i+1}-a_i|\leq C$ for all $i$.  
\end{lemma}
\begin{proof}
By translating both $\varphi_i$ and $\varphi_{i+1}$ we can assume $a_i=0$, so we need to prove a uniform upper bound for $|a_{i+1}|$.  In order to simplify notation, define
\[
\alpha_{i+1} := \varphi_{i+1}(a_{i+1})-\tau.
\]
By Lemma \ref{tau} $\alpha_{i+1} \geq 0$.  Moreover, equation (\ref{tauequation}) implies that if $\alpha_{i+1}=0$, then $\varphi_{i+1} = \mathbb{1}_A^*$ which would contradict the smoothness of $\varphi_{i+1}$, so $\alpha_{i+1}$ is positive.

Lemma \ref{lowerbound} implies $\varphi_{i+1}(x) \geq r\,|x-a_{i+1}| + \tau +\alpha_{i+1}/2$, and in particular
\[
\varphi_{i+1}(x)\geq \tau + (3/2)\,\alpha_{i+1} \hspace{5mm}\text{for }x\text{ such that }|x-a_{i+1}|=\alpha_{i+1}/r.
\]

\begin{center}
\includegraphics{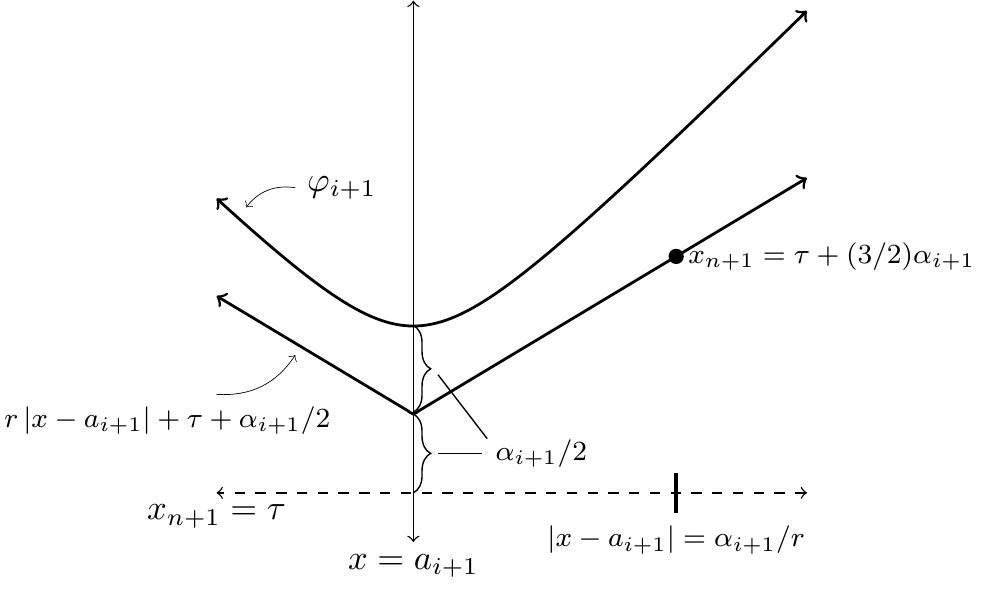}
\end{center}

Lemma \ref{subgradientinclusion} implies
\[
B_r \subset \nabla \varphi_{i+1}\big(B_{\alpha_{i+1}/r}(a_{i+1})\big),
\]
and in particular
\[
\omega_n r^n \leq \lambda \big( \nabla \varphi_{i+1}\big(B_{\alpha_{i+1}/r}(a_{i+1})\big)\big) = \int_{B_{\alpha_{i+1}/r}(a_{i+1})} \det(\nabla^2 \varphi_{i+1})\,d\lambda.
\]
Let
\[
\epsilon = \sup \,\big\{ \,\det\big(\nabla^2 \varphi_{i+1}(x)\big) \mid x \in B_{\alpha_{i+1}/r}(a_{i+1})\,\big\}.
\]
By Proposition \ref{uniformbounds} there exists a constant $C>0$ such that $\alpha_{i+1} \leq C$, so
\begin{equation}\label{epsilonbound}
\omega_n r^n \leq \epsilon\, \omega_n \Big(\scaleobj{0.9}{\dfrac{C}{r}}\Big)^n,
\end{equation}
and thus $(r^2/2C)^n \leq \epsilon$ is a uniform lower bound for $\epsilon$.  It remains to prove an upper bound for $\epsilon$ which goes to $0$ as $|a_{i+1}|\to\infty$, independently of $i$.  Since we assumed $\varphi_i(0)=\inf \varphi_i$, Proposition \ref{uniformbounds} implies
\[
\tau + r\,|x| \leq \varphi_i(x) \leq C+R\,|x|.
\]
By these bounds and hypothesis (B1) on the iteration, there is some constant $C'$ such that
\[
\det\big(\nabla^2 \varphi_{i+1}(x)\big) = \dfrac{h\circ \varphi_i}{\|h\circ \varphi_i\|_1} \leq C'\,\big(\tau+r\,|x|\big)^{-(n+p+1)}
\]
as $|x| \to \infty$.  Thus
\[
\epsilon \leq \sup \,\big\{\,C'\,\big(\tau+r\,|x|\big)^{-(n+p+1)} \mid x \in B_{C/r}(a_{i+1})\,\big\}
\]
for $|a_{i+1}|$ large enough.  Thus, $\epsilon \to 0$ as $|a_{i+1}|\to\infty$, and there is some constant $C$ such that if $|a_{i+1}| \geq C$, then $\epsilon<(r^2/2C)^n$, contradicting equation (\ref{epsilonbound}).  Thus $|a_{i+1}|$ has a uniform upper bound.
\end{proof}

\begin{lemma}\label{tight}
Assume $A$ satisfies (\ref{Aconditions}) and the Monge--Amp\`{e}re iteration (\ref{MAiteration}) satisfies Hypotheses \ref{hypB}.  Then $\{\MA(\widetilde{\varphi}_i)\}$ satisfies the tightness condition
\[
\lim_{R\to\infty} \limsup_{i\to\infty}\,\bigg\{ \int_{|x|\geq R} |x| \,\MA(\widetilde{\varphi}_i)\bigg\} = 0.
\]
\end{lemma}
\begin{proof}
By Proposition \ref{uniformbounds},
\[
\tau + r\,|x-a_i|\leq \varphi_i(x)\leq C+R\,|x-a_i|.
\]
Thus since $h$ is decreasing and $\|h\circ \varphi_i\|$ has a positive lower bound by Proposition \ref{Hfbound}, there is some constant $C'$ such that
\begin{align*}
    \det(\nabla^2 \widetilde{\varphi}_{i+1})(x) &= \det(\nabla^2 \varphi_{i+1})(x+a_{i+1})\\
    &= \dfrac{h\big(\varphi_{i}(x+a_{i+1})\big)}{\|h\circ \varphi_i\|_1} \\
    &\leq C' \,h\big(\tau + r|x-a_i+a_{i+1}|\big).
\end{align*}
And Hypothesis \ref{B1} implies
\[
h\big(\tau + r|x-a_i+a_{i+1}|\big) \leq C' \,\big(\tau + r|x-a_i+a_{i+1}|\big)^{-(n+p+1)}
\]
for $|x|\gg1$. By Lemma \ref{aibound}, $|a_{i+1}-a_i|\leq c$ for a constant $c$ independent of $i$, so $|x-(a_i-a_{i+1})|\geq |x|-|a_i-a_{i+1}|\geq |x|-c$.  Thus,
\[
 \det(\nabla^2 \widetilde{\varphi}_{i+1})(x) \leq C'\,\big(\tau-r\,c + r\,|x|\big)^{-(n+p+1)}
\]
for $|x|\gg 1$.  The tightness condition follows from
\begin{align*}
\lim_{R\to\infty}\limsup_{i\to\infty}\,\bigg\{ \int_{|x|\geq R} |x|\,\det(\nabla^2 \widetilde{\varphi}_i)\,d\lambda)\bigg\} &\leq \lim_{R\to\infty} \int_{|x|\geq R}C'\,|x|\,\big(\tau-r\,c+r\,|x|\big)^{-(n+p+1)}\,d\lambda =0\\
& \leq \lim_{R\to\infty} \int_{|x|\geq R} C''\,|x|^{-(n+p)}\,d\lambda.
\end{align*}
Since $|x|^{-(n+p)}$ is integrable for $p>0$, the limit is $0$.
\end{proof}

This concludes the proof of Proposition \ref{sd}.

\subsection{\texorpdfstring{$\varphi$}{varphi} minimizes \texorpdfstring{$g\big(\big\langle \cdot,\frac{\MA(\varphi)}{\lambda(A)} \big\rangle, \,\mathcal{F}(\cdot)\,\big)$}{g}}

In this section we will prove step 4 of Subsection \ref{outline}.

By the definition of $\mathcal{G}$,
\[
\mathcal{G}(\scaleobj{.9}{\dfrac{\MA(\varphi)}{\lambda(A)}})= \inf \{ \,  g(\langle f, \scaleobj{.9}{\dfrac{\MA(\varphi)}{\lambda(A)}}\rangle, \,\mathcal{F}(f)) \mid f\in\mathcal{C}_{\lin}\,\} \leq g(\langle f,\scaleobj{.9}{\dfrac{\MA(\varphi)}{\lambda(A)}})\rangle,\,\mathcal{F}(f))
\]
for every $f\in\mathcal{C}_{\lin}$.  If we can show that $\mathcal{G}(\scaleobj{.9}{\dfrac{\MA(\varphi)}{\lambda(A)}})= g(\langle \varphi,\scaleobj{.9}{\dfrac{\MA(\varphi)}{\lambda(A)}}\rangle,\,\mathcal{F}(\varphi))$, then step 4 will be proven.

\begin{proposition}\label{sd2}
Assume $A$ satisfies (\ref{Aconditions}) and the Monge--Amp\`{e}re iteration (\ref{MAiteration}) satisfies Hypotheses \ref{hypB}.  Assume a subsequence $\{\widetilde{\varphi}_{i'}\}$ converges to $\varphi$ uniformly on compact sets. Then 
\begin{equation}\label{equality}
\mathcal{G}\Big(\scaleobj{.9}{\dfrac{\MA(\varphi)}{\lambda(A)}}\Big) = g\Big(\Big\langle \varphi, \scaleobj{.9}{\dfrac{\MA(\varphi)}{\lambda(A)}}\Big\rangle, \,\mathcal{F}(\varphi)\Big).
\end{equation}
\end{proposition}

To prove Proposition \ref{sd2}, we first note that Proposition \ref{subsequence} implies $\Int\,\partial\varphi(\mathbb{R}^n)=A$, so $\varphi(x) \leq C\,\big(1+|x|\big)$.  Thus $\varphi \in \mathcal{C}_{\lin}$, and 
\[
\mathcal{G}\Big(\scaleobj{.9}{\dfrac{\MA(\varphi)}{\lambda(A)}}\Big) \leq g\Big(\Big\langle \varphi, \scaleobj{.9}{\dfrac{\MA(\varphi)}{\lambda(A)}}\Big\rangle, \,\mathcal{F}(\varphi)\Big).
\]

The reverse inequality relies upon the continuity of the pertinent functionals along the Monge--Amp\`{e}re iteration, which we prove in the next three lemmas.

\begin{lemma}\label{Fconverge}
Assume the Monge--Amp\`{e}re iteration (\ref{MAiteration}) satisfies Hypotheses \ref{hypB}. Assume a subsequence $\{\widetilde{\varphi}_{i'}\}$ converges to $\varphi$ uniformly on compact sets.  Then $\mathcal{F}(\widetilde{\varphi}_{i'}) \to \mathcal{F}(\varphi)$.  
\end{lemma}

\begin{proof}
By Hypothesis (B1), $h(t)$ is smooth and positive, so $H(t) = \int_t^\infty h\,d\lambda$ is smooth, positive, and strictly decreasing.  In particular, $H$ has a continuous inverse, $H^{-1}$.  Since $\mathcal{F}(f) = H^{-1}\big(\| H\circ f\|_1\big)$, the lemma will follow from showing
\[
\|H\circ \widetilde{\varphi}_{i'}\|_1 \to \|H\circ \varphi\|_1.  
\]

$\widetilde{\varphi}_{i'}$ converges to $\varphi$ uniformly on compact sets, and since $H$ is continuous, it follows that $H\circ \widetilde{\varphi}_{i'}$ converges to $H\circ \varphi$ pointwise.  By Proposition \ref{uniformbounds} $\tau/\lambda(A) + r\,|x| \leq \widetilde{\varphi}_{i'}(x) \leq C+R\,|x|$, and Hypothesis (B1) also implies $H(t) \leq C \,t^{-(n+p)}$ as $t\to \infty$, so $H \circ \widetilde{\varphi}_{i'}$ are uniformly bounded by an $L^1$ function.  Thus $\|H\circ \widetilde{\varphi}_{i'}\|_1 \to \|H\circ \varphi\|_1$ by the dominated convergence theorem.
\end{proof}

Before the next lemma, we need to say a few words about double sequences $s_{i,j}$ indexed by $i,j\in\mathbb{N}$.  We say that a double sequence $s_{i,j}$ converges to $a$ if for every $\epsilon>0$ there exists $N\in\mathbb{N}$ such that $|s_{i,j}-a|\leq \epsilon$ when $i,j\geq N$.  In particular, this implies that the diagonal sequence $s_{i,i}$ converges to $a$.  

We can also consider the iterated limits $\lim_{i\to\infty} ( \,\lim_{j\to\infty}s_{i,j})$ and $\lim_{j\to\infty}(\,\lim_{i\to\infty}s_{i,j})$.  But even if both of these limits exist, and equal the same value, the whole double sequence may not converge.  For example, consider the double sequence $s_{i,j} = \dfrac{i\,j}{i^2 + j^2}$.  Both of the iterated limits exist and equal $0$, but $s_{i,j}$ does not converge as a double sequence, which is clear from $\lim_{i\to\infty}s_{i,i}=1/2$ and $\lim_{i\to\infty}s_{i,2i}=2/5$.  

In order to deduce the convergence of the double sequence from the iterated limits, we need convergence of one iterated limit $\lim_{i\to\infty} (\,\lim_{j\to\infty} s_{i,j})=a$ and uniform convergence of the inside limit $\lim_{j\to\infty} s_{i,j}$.  We say $s_{i,j} \xrightarrow[j\to\infty]{} a_i$ uniformly if for every $\epsilon>0$ there exists $N\in \mathbb{N}$, independent of $i$, such that $j\geq N$ implies $|s_{i,j}-a_i|<\epsilon$ for all $i$.

In the next lemma we assume $\lambda(A)=1$ for notational convenience.  The result in general follows by scaling each equation by $\lambda(A)^{-1}$.

\begin{lemma}\label{angle}
Assume the Monge--Amp\`{e}re iteration (\ref{MAiteration}) satisfies Hypotheses \ref{hypB}. Assume a subsequence $\{\widetilde{\varphi}_{i'}\}$ converges to $\varphi$ uniformly on compact sets.  Then $\big\langle \widetilde{\varphi}_{i'} ,\MA(\widetilde{\varphi}_{i'}) \big\rangle \to \big\langle \varphi, \MA(\varphi) \big\rangle$.
\end{lemma}

\begin{proof}
We'll prove, more generally, that the double sequence $\big\langle \widetilde{\varphi}_{i'},\MA(\widetilde{\varphi}_{j'})\big\rangle \to \big\langle \varphi, \MA(\varphi)\big\rangle$.

Proposition \ref{sd} implies $\MA(\widetilde{\varphi}_{i'}) \to_1 \MA(\varphi)$, and by Lemma \ref{Wassequivalent} this is equivalent to
\[
\big\langle f, \MA(\widetilde{\varphi}_{i'})\big\rangle  \to \big\langle f, \MA(\varphi)\big\rangle,
\]
for all continuous $f$ such that $f(x) \leq C\,\big(1+|x|\big)$ for some constant $C$.  In particular, for every fixed $i'$
\[
\lim_{j'\to\infty}\big\langle \widetilde{\varphi}_{i'}, \MA(\widetilde{\varphi}_{j'})\big\rangle  = \big\langle  \widetilde{\varphi}_{i'}, \MA(\varphi)\big\rangle .
\]
Moreover, this converge is uniform because the uniform estimate $\tau + r\,|x| \leq \widetilde{\varphi}_{i'}(x) \leq C+R\,|x|$ implies
\[
\big|\,\big\langle \widetilde{\varphi}_{i'} ,\MA(\widetilde{\varphi}_{j'}) \big\rangle - \big\langle \widetilde{\varphi}_{i'}, \MA(\varphi) \big\rangle\, \big| \leq R\,\big|\,\big\langle\, |x| ,\MA(\widetilde{\varphi}_{j'}) \big\rangle - \big\langle\,|x|, \MA(\varphi) \big\rangle\, \big|<\epsilon
\]
for $j'\gg1$, independent of $i'$.  

Since $\widetilde{\varphi}_{i'}$ converges to $\varphi$ pointwise and $\widetilde{\varphi}_{i'}$ are uniformly bounded, 
\[
\lim_{i'\to\infty} \big\langle \widetilde{\varphi}_{i'},\MA(\varphi) \big\rangle= \big\langle \varphi, \MA(\varphi)\big\rangle
\]
by the dominated convergence theorem.  Thus the iterated sequence converges

\[
\lim_{i'\to\infty} \Big( \,\lim_{j'\to\infty} \big\langle \widetilde{\varphi}_{i'},\,\MA(\widetilde{\varphi}_{j'}) \big\rangle\,\Big) = \big\langle \varphi, \MA(\varphi)\big\rangle,
\]
and the inside limit converges uniformly.  It follows that the double sequence converges. 
\end{proof}

\begin{lemma}\label{Gusc}
Assume the Monge--Amp\`{e}re iteration (\ref{MAiteration}) satisfies Hypotheses \ref{hypB}.  Assume a subsequence $\{\widetilde{\varphi}_{i'}\}$ converges to $\varphi$ uniformly on compact sets. Then $\limsup_{i' \to \infty} \mathcal{G}\Big(\scaleobj{.9}{\dfrac{\MA(\widetilde{\varphi}_{i'})}{\lambda(A)}}\Big) \leq \mathcal{G}\Big(\scaleobj{.9}{\dfrac{\MA(\varphi)}{\lambda(A)}}\Big)$.
\end{lemma}

\begin{proof}
Recall
\[
\mathcal{G}(\mu) = \inf \big\{ \,g(\langle f,\mu \rangle,\,\mathcal{F}(f)) \mid f\in \mathcal{C}_{\lin}\, \big\}
\]
where $\mathcal{C}_{\lin} = \{\,f: \mathbb{R}^n \to (\tau,\infty) \mid f\text{ continuous, and }f(x)/(1+|x|)\text{ bounded}\,\}$.

Let $\mu_i$ and $\mu$ be any measures in $\mathcal{P}_1$ such that $\mu_i \rightarrow_1 \mu$.  Theorem \ref{Wassequivalent} implies 
\[
\lim_{i\to\infty} \langle f,\mu_i\rangle = \langle f,\mu\rangle
\]
for every fixed $f\in\mathcal{C}_{\lin}$.  Since $g$ is continuous, it follows that 
\[
\lim_{i \to \infty} g\big(\langle f,\mu_i \rangle ,\, \mathcal{F}(f)\big) = g\big(\langle f,\mu \rangle ,\, \mathcal{F}(f)\big)
\]
for every fixed $f\in \mathcal{C}_{\lin}$. Thus $\mathcal{G}$ is the infimum of continuous functionals, so $\mathcal{G}$ is upper semicontinuous, meaning
\[
\limsup_{i\to\infty}\, \mathcal{G}(\mu_i) \leq \mathcal{G}(\mu)
\]
whenever $\mu_i \rightarrow_1 \mu$.  The lemma then follows from Proposition \ref{sd} which says $\scaleobj{.9}{\dfrac{\MA(\widetilde{\varphi}_{i'})}{\lambda(A)}} \rightarrow_1 \scaleobj{.9}{\dfrac{\MA(\varphi)}{\lambda(A)}}$.
\end{proof}

\begin{lemma}\label{basiclimit}
Let $\{x_i\}$ and $\{y_i\}$ be sequences of real numbers with $\lim_{i\to\infty} x_i = x$ and $\limsup_{i\to\infty} y_i \leq y$ for $x$ and $y$ finite.  Let $g(s,t)$ be a continuous function such that $g(s,\cdot)$ is decreasing for all fixed $s$.  Then 
\[
\liminf_{i\to\infty} g(x_i,y_i) \geq g(x,y).
\]
\end{lemma}

\begin{proof}
We need to show that $\liminf_{i\to\infty} g(x_i,y_i) \geq g(x,y)-\epsilon$ for all $\epsilon>0$.  For any $\epsilon>0$ there exists a $\delta>0$ such that $g(x,y+ \delta) \geq g(x,y)-\epsilon/2$ because $g$ is continuous at $(x,y)$.  

Because $\limsup y_i \leq y$, it follows that for $i$ large enough, $y_i \leq y+\delta$. Since $g(s,\cdot)$ is decreasing for every fixed $s$, it follows that
\[
g(x_i,y_i) \geq g(x_i, y+\delta) \text{ for }i\text{ large enough.}
\]
Now using that $x_i \to x$ and the continuity of $g$ we have that
\[
g(x_i,y+\delta) \geq g(x,y+\delta) - \epsilon/2 \geq g(x,y)-\epsilon \text{ for }i\text{ large enough.}
\]
Putting the previous two equations together, and taking the $\liminf$ of both sides implies
\[
\liminf g(x_i,y_i) \geq g(x,y)-\epsilon.
\]
\end{proof}

\begin{lemma}\label{varphiequality}
Assume the Monge--Amp\`{e}re iteration (\ref{MAiteration}) satisfies Hypotheses \ref{hypB}. Assume a subsequence $\{\widetilde{\varphi}_{i'}\}$ converges to $\varphi$ uniformly on compact sets.  Then 
\[
\mathcal{G}\Big(\scaleobj{.9}{\dfrac{\MA(\varphi)}{\lambda(A)}}\Big) \geq g\Big(\Big\langle \varphi,\scaleobj{.9}{\dfrac{\MA(\varphi)}{\lambda(A)}}\Big\rangle,\mathcal{F}(\varphi)\Big).
\]
\end{lemma}
\begin{proof}
Recall equation (\ref{functionallimiteq}) which we restate for the subsequence $\widetilde{\varphi}_{i'}$
\[
\lim_{i\to\infty} \mathcal{F}(\widetilde{\varphi}_{i'}) = \lim_{i\to\infty} g\Big(\Big\langle \widetilde{\varphi}_{i'},\scaleobj{.9}{\dfrac{\MA(\widetilde{\varphi}_{i'})}{\lambda(A)}}\Big\rangle,\,\mathcal{G}\Big(\scaleobj{.9}{\dfrac{\MA(\widetilde{\varphi}_{i'})}{\lambda(A)}}\Big)\Big).
\]
By Lemma \ref{Fconverge}, $\lim_{i' \to \infty} \mathcal{F}(\widetilde{\varphi}_{i'}) = \mathcal{F}(\varphi)$, so
\[
\mathcal{F}(\varphi) = \lim_{i\to\infty} g\Big(\Big\langle \widetilde{\varphi}_{i'},\scaleobj{.9}{\dfrac{\MA(\widetilde{\varphi}_{i'})}{\lambda(A)}}\Big\rangle,\,\mathcal{G}\Big(\scaleobj{.9}{\dfrac{\MA(\widetilde{\varphi}_{i'})}{\lambda(A)}}\Big)\Big).
\]

Lemma \ref{basiclimit} applies to the limit in the above equation because $\Big\langle \widetilde{\varphi}_{i'},\,\scaleobj{.9}{\dfrac{\MA(\widetilde{\varphi}_{i'})}{\lambda(A)}} \Big\rangle \to \Big\langle \varphi,\,\scaleobj{.9}{\dfrac{\MA(\varphi)}{\lambda(A)}}\Big\rangle$ by Lemma \ref{angle} and $\limsup \mathcal{G}\Big(\scaleobj{.9}{\dfrac{\MA(\widetilde{\varphi}_{i'})}{\lambda(A)}}\Big) \leq \mathcal{G}\Big(\scaleobj{.9}{\dfrac{\MA(\varphi)}{\lambda(A)}}\Big)$ by Lemma \ref{Gusc}.  Lemma \ref{basiclimit} implies
\[
\mathcal{F}(\varphi) \geq g\Big(\Big\langle \varphi, \scaleobj{.9}{\dfrac{\MA(\varphi)}{\lambda(A)}}\Big\rangle, \mathcal{G}\Big(\scaleobj{.9}{\dfrac{\MA(\varphi)}{\lambda(A)}}\Big)\Big).
\]

Finally, by applying the decreasing function $g\Big(\Big\langle \varphi,\scaleobj{.9}{\dfrac{\MA(\varphi)}{\lambda(A)}}\Big\rangle ,\cdot\Big)$ to both sides of the previous equation, and using the fact that $g(s,g(s,t))=t$, we get
\[
g\Big(\Big\langle \varphi,\scaleobj{.9}{\dfrac{\MA(\varphi)}{\lambda(A)}}\Big\rangle,\mathcal{F}(\varphi)\Big) \leq \mathcal{G}\Big(\scaleobj{.9}{\dfrac{\MA(\varphi)}{\lambda(A)}}\Big).
\]
\end{proof}

This concludes the proof of Proposition \ref{sd2}.

\subsection{Convergence of the iteration}

In this section we will prove step 5 of Subsection \ref{outline}, thus completing the proof of Theorem \ref{introthm}.

\begin{proposition}\label{final}
Assume $A$ satisfies (\ref{Aconditions}) and the Monge--Amp\`{e}re iteration satisfies Hypotheses \ref{hypB}.  Then, $\widetilde{\varphi}_i$ converges to $\varphi$, which is a smooth solution to equation (\ref{ourpde}), in every $C^{k,\alpha}$ norm on compact sets. 
\end{proposition}

Firstly, we will prove that the limit $\varphi$ of any convergent subsequence $\widetilde{\varphi}_{i'}$ is a smooth solution to equation (\ref{ourpde}).  In Lemma \ref{limitsolution} we show  $\varphi$ satisfies the differential equation and the second boundary condition, and in Lemma \ref{limitintegral} we show that $\varphi$ satisfies the normalization $\int_A \varphi^*\,d\lambda=-\tau$.  Before proving Lemma \ref{limitsolution} we must discuss the regularity results of Caffarelli.

Consider the Monge--Amp\`{e}re equation
\begin{equation}\label{appendixMA}
\det(\nabla^2 f(x)) = g(x)
\end{equation}
for $g$ a positive function.  We say that a convex function $f$ is an \textit{Alexandrov solution} to equation (\ref{appendixMA}) if
\[
\MA(f) = g\,\lambda
\]
as Borel measures. Caffarelli \cite{Caff2} proved if $f$ is a strictly convex Alexandrov solution of equation (\ref{appendixMA}) and $g\in C^{0,\alpha}$, then $f\in C^{2,\alpha}$.

Caffarelli also proved a means of showing that Alexandrov solutions to certain Monge--Amp\`{e}re equations are strictly convex.  Let $\Omega$ be an open, convex, bounded set, and let $f:\Omega \to \mathbb{R}$ be a convex function. If $f$ is not strictly convex, then there is some supporting hyperplane $z_{n+1} = f(x) + \langle z-x, y\rangle$ which intersects the graph of $f$ at more than one point.  Denote this set by
\[
S_{y} = \{\,z\in \Omega \mid f(z) = f(x) + \langle z-x, y\rangle\,\}.
\]
$S_y$ is a convex set because it is the sublevel set of the convex function, so we can consider its \textit{extreme points} which are the points in the boundary of $S_y$ that are not convex combinations of other points in $\overline{S_y}$.  Caffarelli \cite{Caff3} proved that if $f$ is an Alexandrov solution to $\det(\nabla ^2 f) =g$ where $c^{-1} \leq g \leq c$ for some $c>0$, then the extreme points of $S_y$ must lie in the boundary of $\Omega$.  Caffarelli's result has a nice corollary when the domain  of $f$ is all of $\mathbb{R}^n$.

\begin{lemma}\label{regularitylemma}
Let $f:\mathbb{R}^n \to \mathbb{R}$ be a convex function. If $\partial f(\mathbb{R}^n)$ has nonempty interior and $\MA(f) = g$ for $g$ a positive, continuous function, then $f$ is stricly convex.
\end{lemma}
\begin{proof}
Assume $f$ is not strictly convex.  Then, for some point point $x\in\mathbb{R}^n$ and some point $y\in\partial f(x)$ the set
\[
S_{y} = \{\,z\in \mathbb{R}^n \mid f(z) = f(x) + \langle z-x, y\rangle\,\}
\]
contains more than one point.  Since the graph of $f$ lies above the supporting hyperplane $z_{n+1} = f(x) + \langle z-x, y\rangle$, it follows that $S_y$ is the sublevel set of a convex function, so it is convex.  Since $\dom(f) = \mathbb{R}^n$, it follows that $f$ is continuous, so $S_y$ is closed.  

We claim that the set $S_y$ cannot contain any extreme points.  Let $p$ be any point in the boundary of $S_y$.  Since $g$ is positive and continuous, $0<c<g(x) \leq C$ on $\overline{B_R(p)}$.  Thus $f$ is an Alexandrov solution to $0<c\leq\det(\nabla^2 f) \leq C$, so Caffarelli's theorem implies the extreme points of $S_y\cap \overline{B_R(p)}$ occur on the boundary.  Thus $p$ is not an extreme point of $S_y$, so $S_y$ has no extreme points.

Theorem 18.5.3 of Rockafellar \cite{Rockafellar} says any nonempty, closed convex set which contains no lines must contain at least one extreme point, so $S_y$ must contain a line $z+t\,u$ for $t\in \mathbb{R}$.  This implies $f(z+t\,u) = f(z) +t\,\langle u,y\rangle$, and we claim this contradicts $\partial f(\mathbb{R}^n)$ having nonempty interior.  By adding a linear function to $f$, we can assume without loss of generality that $B_r \subset \partial f(\mathbb{R}^n)$ for some small $r>0$.  This implies $f(x) \geq -C + r\,|x|$ for some constant $C$.  This lower bound contradicts $f(z+t\,u) = f(z) + t\,\langle u,y\rangle$ for either $t\to\infty$ or $t\to-\infty$, thus $f$ is strictly convex.
\end{proof}

\begin{lemma}\label{limitsolution}
Assume $A$ satisfies (\ref{Aconditions}) and the Monge--Amp\`{e}re iteration (\ref{MAiteration}) satisfies Hypotheses \ref{hypB}.    Let $\{\widetilde{\varphi}_{i'}\}$ be a subsequence which converges to $\varphi$ uniformly on compact sets.  Then $\varphi$ is a smooth solution to 
\[
    \begin{cases}
        \dfrac{\det(\nabla^2 \varphi)}{\lambda(A)} = \dfrac{h\circ \varphi}{\| h\circ \varphi \|_1} \\
        \nabla \varphi ( \mathbb{R}^n) = A.
    \end{cases}
\]
\end{lemma}

\begin{proof}
By Proposition \ref{sd} and the definition of $\mathcal{G}$
\begin{equation}\label{minimizer}
g\Big(\Big\langle \varphi,\scaleobj{.9}{\dfrac{\MA(\varphi)}{\lambda(A)}}\Big\rangle,\mathcal{F}(\varphi)\Big) = \mathcal{G}\Big(\scaleobj{.9}{\dfrac{\MA(\varphi)}{\lambda(A)}}\Big) \leq g\Big(\Big\langle f,\scaleobj{.9}{\dfrac{\MA(\varphi)}{\lambda(A)}}\Big\rangle,\mathcal{F}(f)\Big)\hspace{5mm}\text{for all }f\in\mathcal{C}_{\lin},
\end{equation}
where we recall
\[
\mathcal{C}_{\lin} = \{\,f: \mathbb{R}^n \to (\tau,\infty) \mid f\text{ continuous, and } f(x)/(1+|x|)\text{ bounded}\,\}.
\]

Consider a variation $\varphi_\epsilon = \varphi + \epsilon\,b$ for $b$ a continuous, bounded function on $\mathbb{R}^n$. For $\epsilon$ small enough, $\varphi_\epsilon\in \mathcal{C}_{\lin}$.  Hypothesis \ref{B1} says $h$ is positive and decreasing so $H(t)=\int_t^\infty h\,d\lambda$ is positive, strictly decreasing, and convex.  Thus
\[
-(t-s)\,h(s)\leq H(t)-H(s) \leq -(t-s)\,h(t),
\]
for any values $s,\,t$.  Letting $s=\varphi(x)$ and $t=\varphi_\epsilon(x)$ implies
\begin{multline*}
-\int_{\mathbb{R}^n}b(x)\,h\big(\varphi(x)\big)\,d\lambda \leq \epsilon^{-1} \left( \int_{\mathbb{R}^n} H\big(\varphi_\epsilon(x)\big)\,d\lambda - \int_{\mathbb{R}^n} H\big(\varphi(x)\big)\,d\lambda \right) \\ \leq -\int_{\mathbb{R}^n} b(x)\,h\big(\varphi_\epsilon(x)\big)\,d\lambda,
\end{multline*}
for $\epsilon>0$ with the reverse inequalities when $\epsilon<0$. The dominated convergence theorem implies
\[
\int_{\mathbb{R}^n} b(x)\,h\big(\varphi_{\epsilon}(x)\big)\,d\lambda \to \int_{\mathbb{R}^n} b(x)\,h\big( \varphi(x)\big)\,d\lambda,
\]
so
\begin{multline*}
\dfrac{d}{d\epsilon} \bigg|_{\epsilon=0} \|H\circ \varphi_{\epsilon}\|_1 = \lim_{\epsilon \to 0} \epsilon^{-1} \left( \int_{\mathbb{R}^n} H\big(\varphi_\epsilon(x)\big)\,d\lambda - \int_{\mathbb{R}^n} H\big(\varphi(x)\big)\,d\lambda \right) \\ = -\int_{\mathbb{R}^n} b(x)\,h\big(\varphi(x)\big)\,d\lambda.
\end{multline*}
Since $H$ is strictly decreasing and differentiable, it follows that $H^{-1}$ is differentiable.  Using the formula for the derivative of the inverse and the definition of $\mathcal{F}$, it follows that
\begin{multline*}
\dfrac{d}{d\epsilon}\bigg|_{\epsilon=0} \mathcal{F}(\varphi_\epsilon) = \big(H^{-1}\big)'\big(\|H\circ \varphi\|_1\big)\cdot\left(-\int_{\mathbb{R}^n} h\big(\varphi(x)\big)\,b(x)\,d\lambda\right) \\ = \dfrac{1}{h\big(\mathcal{F}(\varphi)\big)}\,\left(\int_{\mathbb{R}^n} h\big(\varphi(x)\big)\,b(x)\,d\lambda\right).
\end{multline*}
The function $\epsilon \to \Big\langle \varphi_{\epsilon}, \,\scaleobj{.9}{\dfrac{\MA(\varphi)}{\lambda(A)}}\Big\rangle$ is linear in $\epsilon$ and it has derivative $\Big\langle b, \,\scaleobj{.9}{\dfrac{\MA(\varphi)}{\lambda(A)}}\Big\rangle$.  The definition of $\mathcal{G}$ in Hypothesis \ref{B2} assumes $g(s,t)$ is differentiable, so $g\Big(\Big\langle \varphi_{\epsilon},\scaleobj{.9}{\dfrac{\MA(\varphi)}{\lambda(A)}}\Big\rangle,\mathcal{F}(\varphi_\epsilon)\Big)$ is differentiable at $\epsilon=0$.  Differentiating the right hand side of equation (\ref{minimizer}) for $f=\varphi_\epsilon$, and evaluating at its minimimum when $\epsilon=0$ implies
\begin{multline}\label{minimum}
0=\dfrac{d}{d\epsilon}\bigg|_{\epsilon=0} g\Big(\Big\langle \varphi_{\epsilon},\scaleobj{.9}{\dfrac{\MA(\varphi)}{\lambda(A)}} \Big\rangle,\mathcal{F}(\varphi_\epsilon)\Big) = \frac{\partial g}{\partial s}\Big(\Big\langle \varphi,\scaleobj{.9}{\dfrac{\MA(\varphi)}{\lambda(A)}} \Big\rangle, \mathcal{F}(\varphi)\Big)\,\Big\langle b,\scaleobj{.9}{\dfrac{\MA(\varphi)}{\lambda(A)}}\Big\rangle\\ + \frac{\partial g}{\partial t}\Big(\Big\langle \varphi,\scaleobj{.9}{\dfrac{\MA(\varphi)}{\lambda(A)}}\Big\rangle, \mathcal{F}(\varphi)\Big)\,\dfrac{1}{h\big(\mathcal{F}(\varphi)\big)}\,\left(\int_{\mathbb{R}^n} h\big(\varphi(x)\big)\,b(x)\,d\lambda\right).
\end{multline}

When $b=1$ is the constant function, then using $\partial\varphi(\mathbb{R}^n)=A$ from Proposition \ref{uniformbounds} we find
\[
\Big\langle b,\scaleobj{.9}{\dfrac{\MA(\varphi)}{\lambda(A)}}\Big\rangle = \lambda(A)^{-1}\int_{\mathbb{R}^n} \MA(\varphi) = \lambda(A)^{-1}\,\lambda \big(\partial\varphi(\mathbb{R}^n)\big)=1.
\]
So putting $b=1$ into equation (\ref{minimum}) shows
\[
\frac{\partial g}{\partial s}\Big(\Big\langle \varphi,\scaleobj{.9}{\dfrac{\MA(\varphi)}{\lambda(A)}}\Big\rangle, \mathcal{F}(\varphi)\Big) = -\frac{\partial g}{\partial t}\Big(\Big\langle \varphi,\scaleobj{.9}{\dfrac{\MA(\varphi)}{\lambda(A)}}\Big\rangle, \mathcal{F}(\varphi)\Big)\,\dfrac{1}{h\big(\mathcal{F}(\varphi)\big)}\,\left(\int_{\mathbb{R}^n} h\big(\varphi(x)\big)\,d\lambda\right).
\]
Plugging the previous equation back into equation (\ref{minimum}) implies
\[
\int_{\mathbb{R}^n} b(x)\, \scaleobj{.9}{\dfrac{\MA(\varphi)}{\lambda(A)}}  = \dfrac{\int_{\mathbb{R}^n} h\big(\varphi(x)\big)\,b(x)\,d\lambda}{\int_{\mathbb{R}^n} h\big(\varphi(x)\big)\,d\lambda},
\]
for any continuous, bounded $b$. It follows that
\begin{equation}\label{alexandrov}
\dfrac{\MA(\varphi)}{\lambda(A)} = \dfrac{h\circ \varphi}{\|h\circ\varphi\|_1}\,\lambda.
\end{equation}
so $\varphi$ solves (\ref{ourpde}) in the Alexandrov sense.  Since $h$ is positive, Lemma \ref{regularitylemma} implies $\varphi$ is strictly convex.  Since $\partial\varphi(\mathbb{R}^n)$ is bounded, it follows that $h\circ\varphi \in C^{0,1}$.  Since $\varphi$ is strictly convex and  $h\circ\varphi \in C^{0,1}$, Caffarelli \cite[Theorem~2]{Caff2} implies $\varphi\in C^{2,\alpha}$.  Then $\varphi$ is a classical solution of
\[
\dfrac{\det(\nabla^2 \varphi)}{\lambda(A)} = \dfrac{h\circ \varphi}{\|h\circ\varphi\|_1}.
\]
Since $h$ is smooth, elliptic regularity shows that $\varphi$ is in fact smooth.  Thus $\varphi$ is strictly convex, and $\partial \varphi(\mathbb{R}^n)=A$ is upgraded to $\nabla \varphi(\mathbb{R}^n) =A$.
\end{proof}

In Lemma \ref{limitintegral} we show $\varphi$ satisfies the normalization in equation (\ref{ourpde}). First we need a convex analysis lemma.

\begin{lemma}\label{legendreintregrallemma}
Let $\Omega$ be an open convex set, and let  $f:\Omega \to \mathbb{R}$ be convex and $C^1$.  Then, 
\[
\int_{\nabla f(\Omega)} f^*\,d\lambda = \int_{\Omega} \big(\langle x,\,\nabla f(x)\rangle - f(x)\big)\,\MA(f),
\]
where $\MA(f)$ is the Monge--Amp\`{e}re measure of $f$.
\end{lemma}
\begin{proof}
The function $(\nabla f)^{-1}$ may not exist, but we can still define $(\nabla f)^{-1}(U) = \{\,x\mid \nabla f(x) \in U\,\}$. For any Borel $U\subset \Omega$ the Monge--Amp\`{e}re measure satisfies
\begin{align*}
    (\nabla f)_{\#} (\MA(f))\,(U) &= \MA(f)\,\big((\nabla f)^{-1}(U)\big)\\
    &= \lambda\,\big(\, \nabla f((\nabla f)^{-1}(U)) \,\big) = \lambda \,(U),
\end{align*}
so $(\nabla f)_{\#} (\MA(f)) = \lambda$.  The definition for the pushforward of a measure implies for any measurable function $g$ on $\nabla f(\Omega)$
\[
\int_{\nabla f(\Omega)} g\,d\lambda = \int_{\Omega} (g\circ \nabla f) \,\MA(f).
\]
The result follows by setting $g = f^*$ and recalling that for $f$ which are $C^1$, $f^*(\nabla f(x)) = \langle x,\nabla f(x)\rangle -f(x)$.
\end{proof}

The proof of Lemma \ref{limitintegral} involves another double sequence argument as in Lemma \ref{angle}.

\begin{lemma}\label{limitintegral}
Under the same assumptions as Lemma \ref{limitsolution},
\[
\int_A \varphi^*\,d\lambda = -\tau.
\]
\end{lemma}
\begin{proof}
Every $\widetilde{\varphi}_{i'}$ is a smooth, convex function with $\nabla \widetilde{\varphi}_{i'}(\mathbb{R}^n)=A$, so Lemma \ref{legendreintregrallemma} implies
\[
-\tau = \int_A \widetilde{\varphi}_{i'}^*\,d\lambda = \int_{\mathbb{R}^n} \big(\langle \nabla \widetilde{\varphi}_{i'}(x),x \rangle - \widetilde{\varphi}_{i'}(x)\big)\,\det(\nabla^2 \widetilde{\varphi}_{i'})\,d\lambda.
\]
By Lemma \ref{angle}
\begin{equation}\label{r1}
\int_{\mathbb{R}^n} \widetilde{\varphi}_{i'}(x)\,\det(\nabla^2 \widetilde{\varphi}_{i'})\,d\lambda \to \int_{\mathbb{R}^n} \varphi\,\MA(\varphi),
\end{equation}
and by Lemma \ref{limitsolution} $\varphi$ is smooth, so $\MA(\varphi) = \det(\nabla^2 \varphi)\,d\lambda$.

By Rockafellar \cite[Theorem~24.5]{Rockafellar}, for every $x$ in $\mathbb{R}^n$ and every $\epsilon>0$ there exists $i_0$ such that $i'\geq i_0$ implies $\partial \widetilde{\varphi}_{i'}(x)\subset \partial \varphi(x)+ B_\epsilon$.  Since $\widetilde{\varphi}_{i'}$ and $\varphi$ are smooth, this implies $\nabla \widetilde{\varphi}_{i'}$ converges pointwise to $\nabla \varphi(x)$.  Also, $|\,\langle \nabla \widetilde{\varphi}_{i'}(x),x\rangle\,| \leq R\,|x|$, where $R$ is a constant such that $A \subset B_R$.  We consider the double sequence $\int_{\mathbb{R}^n} \big\langle \nabla \widetilde{\varphi}_{i'}(x),x\big\rangle \,\MA(\widetilde{\varphi}_{j'})$. 
\[
\lim_{j'\to\infty}\int_{\mathbb{R}^n} \big\langle \nabla \widetilde{\varphi}_{i'}(x),x\big\rangle \,\MA(\widetilde{\varphi}_{j'}) = \int_{\mathbb{R}^n}  \big\langle \nabla \widetilde{\varphi}_{i'}(x),x\big\rangle\,\MA(\varphi),
\]
because $\MA(\widetilde{\varphi}_{j'}) \rightarrow_1 \MA(\varphi)$.  The convergence is uniform because of the bound $|\,\langle \nabla \widetilde{\varphi}_{i'}(x),x\rangle\,| \leq R\,|x|$. The iterated limit of the double sequence 
\[
\lim_{i'\to\infty}\Big(\,\lim_{j'\to\infty}  \int_{\mathbb{R}^n} \big\langle \nabla \widetilde{\varphi}_{i'}(x),x\big\rangle \,\MA(\widetilde{\varphi}_{j'}) \,\Big) = \int_{\mathbb{R}^n} \big\langle \nabla\varphi(x),x\big\rangle\,\MA(\varphi),
\]
converges by the dominated convergence theorem because the uniformly bounded $\nabla \widetilde{\varphi}_{i'}$ converge pointwise to $\nabla \varphi$.  Thus the whole double sequence converges, and in particular the diagonal sequence converges:
\begin{equation}\label{r2}
\int_{\mathbb{R}^n} \big\langle \widetilde{\varphi}_{i'}(x),x\big\rangle \,\MA(\widetilde{\varphi}_{i'}) \to \int_{\mathbb{R}^n} \big\langle \nabla \varphi(x),x\big\rangle\,\MA(\varphi).
\end{equation}

The limits (\ref{r1}) and (\ref{r2}) together imply
\[
-\tau= \lim_{i' \to \infty}  \int_{\mathbb{R}^n} \Big(\big\langle \nabla \widetilde{\varphi}_{i'}(x),x \big\rangle - \widetilde{\varphi}_{i'}(x)\Big)\, \MA(\widetilde{\varphi}_{i'}) =     \int_{\mathbb{R}^n} \Big(\big\langle \nabla \varphi(x),x\big\rangle - \varphi(x)\Big)\,\MA(\varphi) = \int_A \varphi^*\,d\lambda.
\]
\end{proof}

Now we can finish the proof of Proposition \ref{final}.  

\begin{proof}
Proposition \ref{subsequence} implies that every subsequence $\widetilde{\varphi}_{i'}$ has a convergent subsequence.  If we can show that every convergent subsequence has the same unique limit, then this will imply the convergence of the whole sequence $\widetilde{\varphi}_i$. 

Let $\widetilde{\varphi}_{i'}$ be a subsequence which converges to $\varphi$ uniformly on compact sets.  By Lemmas \ref{limitsolution} and \ref{limitintegral}, $\varphi$ is a smooth solution to equation (\ref{ourpde}).  Also, since $\inf_{\mathbb{R}^n}\{\widetilde{\varphi}_i\}=\widetilde{\varphi}_i(0)$ and $\widetilde{\varphi}_{i'}$ converges to $\varphi$ uniformly on compact sets, it follows that $\inf_{\mathbb{R}^n} \{\varphi\} = \varphi(0)$.  By Hypothesis \ref{E}, solutions to the second boundary problem (\ref{ourpde}) are unique up to translation, so $\inf_{\mathbb{R}^n} \{\varphi\}=\varphi(0)$ implies $\varphi$ is the same unique limit for any subsequence $\widetilde{\varphi}_{i'}$.

Now we must show that the convergence $\widetilde{\varphi}_i \to \varphi$ extends to $C^{k,\alpha}$.  As in the proof of Lemma \ref{limitintegral}, $\nabla \widetilde{\varphi}_{i}$ converges to $\nabla \varphi$, so the smoothness of $h$ implies $h\circ \widetilde{\varphi}_i$ converges to $h\circ \varphi$ in $C^{0,1}$ on compact sets.  So in particular, the $C^{0,1}$ norm of $(h\circ \widetilde{\varphi}_i) / \|h \circ \widetilde{\varphi}_i\|$ has a uniform bound on each compact set.

Caffarelli \cite[Theorem 2]{Caff2} then implies $(h\circ \widetilde{\varphi}_i) / \|h \circ \widetilde{\varphi}_i\|$ has a uniform bound in $C^{2,\alpha}$ on a slightly smaller compact set.  The compact embeddings of H\"{o}lder spaces implies there exists convergent subsequences in $C^{2,\beta}$ for some $\beta<\alpha$.  But the limits of these subsequences are unique, so we have $C^{2,\beta}$ convergence.  Bootstrapping this argument yields $C^{k,\alpha}$ convergence on compact sets.

\end{proof}

\section{K\"{a}hler--Ricci iteration}\label{RIsection}

This section is organized as follows.  In Subsection \ref{iteration2} we prove Theorem \ref{MAiteration2} about the convergence of the Monge--Amp\`{e}re iteration (\ref{MAiteration}) with $h(t)=e^{-t}$.  In Section \ref{smoothtoric}  we give the necessary background on toric K\"{a}hler manifolds, and in Section \ref{proofofRI} we prove Theorem \ref{introthm3} about the convergence of the K\"{a}hler--Ricci iteration.  In Section \ref{Kahlerfunctionals} we show the functionals $\mathcal{F}$ and $\mathcal{G}$ roughly correspond to the Ding functional and Mabuchi K-Energy from K\"{a}hler geometry.

\subsection{The Monge--Amp\`{e}re iteration with \texorpdfstring{$h(t)=e^{-t}$}{a}}\label{iteration2}

To prove Theorem \ref{MAiteration2} about the convergence of the Monge--Amp\`{e}re iteration for $h(t)=e^{-t}$, we must verify Hypotheses \ref{hypB} and apply Theorem \ref{introthm}.

\noindent \textit{Hypothesis \ref{B1}:}

Clearly $e^{-t}$ is smooth, positive, and decreasing. Also, it is bounded by $C\,t^{-(n+p+1)}$ when $t\gg1$ and $p=1$, for example.

\noindent \textit{Hypothesis \ref{E}:}

Berman--Berndtsson \cite[Theorem~1.1]{BB} says if the barycenter of $A$ lies at the origin, then there exist smooth convex solutions $\varphi$ to
\begin{equation}\label{BB}
\begin{cases}
    \det(\nabla^2 \varphi) = e^{-\varphi}\\
    \nabla \varphi(\mathbb{R}^n) = A,
\end{cases}
\end{equation}
and they are unique up to translations by $\mathbb{R}^n$.  A convex function $\varphi$ solves equation (\ref{BB}) if and only if $\varphi+c$ solves 
\begin{equation}\label{BBnorm}
\begin{cases}
    \dfrac{\det(\nabla^2 \varphi)}{\lambda(A)} = \dfrac{e^{-\varphi}}{\|e^{-\varphi}\|_1}\\
    \nabla \varphi(\mathbb{R}^n) = A,
\end{cases}
\end{equation}
so convex solutions to equation (\ref{BBnorm}) are unique up to translation and an additive constant.  Berman--Berndtsson \cite[Theorem~1.1]{BB} also proves that if $\varphi$ is a convex solution to equation (\ref{BB}), then $\int_A \varphi^*\,d\lambda <\infty$.  Thus, the normalization $\int_A \varphi^*\,d\lambda=-\tau$ is valid for any $\tau\in \mathbb{R}$, and convex solutions to the normalized Monge--Amp\`{e}re second boundary problem (\ref{ourpde}) are unique up to translations.

\noindent \textit{Hypothesis \ref{B2}:}

We define $g(s,t)=s-t$.  To verify Hypothesis \ref{B2} we note $g$ is decreasing in $t$, and $g(s,g(s,t)) = s-(s-t)=t$.

\noindent \textit{Hypothesis \ref{B3}:}

First, we compute $\mathcal{F}$.  We integrate $H(t) = \int_t^{\infty} e^{-s}\,d\lambda = e^{-t}$, so $H^{-1}(t) = -\log(t)$, and
\[
\mathcal{F}(f) = H^{-1}\big(\|H \circ f\|_1\big) = -\log\left( \| e^{-f} \|_1\right),
\]
for any $f\in \mathcal{C}_{\lin} = \big\{f:\mathbb{R}^n \to (\tau,\infty) \mid f \text{ continuous, and } f(x)/(1+|x|)\text{ bounded}\,\big\}$.

Next we define $\mathcal{G}$ with $g(s,t)=s-t$.
\[
\mathcal{G}(\mu) = \inf \big\{ \,\langle f,\,\mu\rangle +\log \left( \|e^{-f}\|_1\right) \mid f\in\mathcal{C}_{\lin}\,\big\}.
\]

For $g(s,t)=s-t$ Hypothesis \ref{B3} says
\[
\Big\langle \varphi_{i+1},\, \scaleobj{.9}{\dfrac{\MA(\varphi_{i+1})}{\lambda(A)}} \Big\rangle - \mathcal{G}\Big(\scaleobj{.9}{\dfrac{\MA(\varphi_{i+1})}{\lambda(A)}}\Big) \leq \mathcal{F}(\varphi_i).
\]
To prove this inequality, we compute $\mathcal{G}(\mu)$ when $\mu$ has a continuous density.  Lemma \ref{RicG} can be interpreted as saying $\mathcal{G}(\mu) = -\Ent_{\lambda}(\mu)$, the relative entropy of $\mu$ with respect to Lebesgue measure $\lambda$.

\begin{lemma}\label{RicG}
Let $h$ be a continuous, nonnegative function on $\mathbb{R}^n$ such that $h\,\lambda \in \mathcal{P}_1$.  Then,
\[
\mathcal{G}(h\,\lambda ) = -\int_{\mathbb{R}^n} h\log(h)\,d\lambda.
\]
\end{lemma}
\begin{proof}

First, we show $-\int_{\mathbb{R}^n} h\log(h)\,d\lambda \leq \mathcal{G}(h\lambda)$.  It is sufficient to show 
\begin{equation}\label{logcvx}
-\int_{\mathbb{R}^n} h\log(h)\,d\lambda \leq \int_{\mathbb{R}^n}f\,h\,d\lambda + \log(\int_{\mathbb{R}^n}e^{-f}\,d\lambda)
\end{equation}
for all $f\in\mathcal{C}_{\lin}$.    The convexity of the function $x\log(x)$ implies the elementary inequality
\begin{equation}\label{elemineq}
t-s \leq t\log(t) - t\log(s) = t\log(t/s)
\end{equation}
for $t\geq 0$ and $s>0$.  Equation (\ref{elemineq}) is clearly true when $s=t$, and when $t=0$ it is true by interpreting $t\log(t)=0$ when $t=0$.  When $t>s$ the convexity of $x\log(x)$ implies
\[
1+\log(s) \leq \dfrac{t\log(t)-s\log(s)}{t-s}.
\]
Multiplying both sides by $t-s$ and simplifying implies equation (\ref{elemineq}).  When $s>t$ the convexity of $x\log(x)$ implies
\[
\dfrac{s\log(s) - t\log(s)}{s-t} \leq 1+\log(s).
\]
Multiplying both sides by $t-s$ switches the inequality and implies equation (\ref{elemineq}).  Applying equation (\ref{elemineq}) with $t=h(x)$ and $s=e^{-f(x)}/\|e^{-f}\|_1$ yields
\begin{align*}
\dfrac{e^{-f(x)}}{\|e^{-f}\|_1} -h(x) &\leq h(x)\log\big(h(x)\big) - h(x)\log\left(\dfrac{e^{-f(x)}}{\|e^{-f}\|_1} \right)\\
&= h(x)\log\big(h(x)\big) +h(x)f(x) + \log(\|e^{-f}\|_1)\, h(x)
\end{align*}
Since $h\,\lambda\in\mathcal{P}_1$ it follows that $\|h\|_1=1$, so integrating over $\mathbb{R}^n$ implies
\[
0 \leq \int_{\mathbb{R}^n}h\,\log(h)\,d\lambda + \int_{\mathbb{R}^n}f\,h\,d\lambda + \log\left(\|e^{-f}\|_1\right),
\]
which is equivalent to equation (\ref{logcvx}).

Next, we show the reverse inequality: $-\int_{\mathbb{R}^n} h\log(h)\,d\lambda \geq \mathcal{G}(h\lambda)$.  Consider
\begin{equation}\label{c2}
f_k(x) = \min\{-\log(h(x)),\, k\,(1+|x|) \}.
\end{equation}
By assumption $h$ is continuous and integrable, so $h$ is finite, which implies $f_k$ is continuous and finite.  Thus, $f_k \in \mathcal{C}_{\lin}$, and
\begin{equation}\label{fk2}
\mathcal{G}(h\,\lambda) \leq \langle f_k,h\,\lambda \rangle +\log\left(\|e^{-f_k}\|_1\right).
\end{equation}
$h\,f_k$ is an increasing sequence of functions which converge pointwise to $-h\log(h)$, so
\[
\langle f_k,h\,\lambda \rangle = \int_{\mathbb{R}^n} f_k\,h \,d\lambda \to -\int_{\mathbb{R}^n} h\log(h)\,d\lambda
\]
by the monotone convergence theorem.  

$e^{-f_k}$ converges to $h$ pointwise, and $e^{-f_k(x)}\leq e^{-f_1(x)}$ for all $k$.  Since $e^{-f_1}$ is integrable, the dominated convergence theorem implies
\[
\|e^{-f_k}\|_1 \to \|h\|_1=1.
\]
Thus $\log\left(\|e^{-f_k}\|_1\right) \to 0$, and taking the limit as $k\to\infty$ in equation (\ref{fk2}) implies
\[
\mathcal{G}(h\,\lambda) \leq -\int_{\mathbb{R}^n}h\log(h)\,d\lambda.
\]
\end{proof}

We need one further lemma concerning the condition $\int_A \varphi_{i}^*=-\tau$ along the iteration.  Lemma \ref{compare} can be thought of as an integral comparison principle for the Monge--Amp\`{e}re measure in comparison to the traditional Monge--Amp\`{e}re measure comparison principle  of Rauch--Taylor \cite{RT}.

\begin{lemma}\label{compare}
If $\varphi$ and $\psi$ are two smooth, convex functions on $\mathbb{R}^n$ such that $\nabla\varphi (\mathbb{R}^n) = \nabla \psi (\mathbb{R}^n)=A$, and $\int_A \varphi^*\,d\lambda = \int_A \psi^* \,d\lambda$, then
\[
\big\langle \varphi, \MA(\varphi) \big\rangle \leq \big\langle \psi, \MA(\varphi)\big\rangle.
\]
\end{lemma}

\begin{proof}
The equality case of Legendre duality is $\varphi(x) + \varphi^*(\nabla \varphi(x)) = \big\langle x,\nabla \varphi(x) \big\rangle$.  Integrating both sides of this equality against the Monge--Amp\`{e}re measure of $\varphi$ and using the change of variables $y = \nabla \varphi(x)$ and $x = \nabla \varphi^*(y)$ implies
\[
\int_{\mathbb{R}^n} \varphi \,\MA(\varphi) + \int_A \varphi^* \,d\lambda = \int_A \big\langle y, \nabla \varphi^*(y)\big\rangle\,d\lambda.
\]
Since $\psi$ is convex, $\langle y,x\rangle \leq \psi^*(y) + \psi(x)$ for all $x$ and $y$, so
\[
\int_A \big\langle y, \nabla \varphi^*(y)\big\rangle\,d\lambda \leq \int_A  \psi^*(y)+\psi\big( \nabla \varphi^*(y)\big)\,d\lambda = \int_{\mathbb{R}^n} \psi \,\MA(\varphi) + \int_A \psi^*\,d\lambda.
\]
Since $\int_A \varphi^*\,d\lambda = \int_A \psi^* \,d\lambda$, it follows that 
\[
\int_{\mathbb{R}^n} \varphi\,\MA(\varphi) \leq \int_{\mathbb{R}^n} \psi \,\MA(\varphi).
\]
\end{proof}

Now we can use Lemma \ref{RicG} and Lemma \ref{compare} to prove Hypothesis \ref{B3}.

\begin{lemma}\label{last}
$\Big\langle \varphi_{i+1},\, \scaleobj{.9}{\dfrac{\MA(\varphi_{i+1})}{\lambda(A)}} \Big\rangle - \mathcal{G}\Big(\scaleobj{.9}{\dfrac{\MA(\varphi_{i+1})}{\lambda(A)}}\Big) \leq \mathcal{F}(\varphi_i)$.
\end{lemma}
\begin{proof}
Since $\varphi_i$ is smooth, $\MA(\varphi_i) = \det(\nabla^2 \varphi_i)\lambda$ for all $i$.  Since $\int_A \varphi_i^*\,d\lambda = \int_A \varphi_{i+1}^*\,d\lambda$, Lemma \ref{compare} implies
\[
\int_{\mathbb{R}^n} \varphi_{i+1} \scaleobj{.9}{\dfrac{\det(\nabla^2 \varphi_{i+1})}{\lambda(A)}}\,d\lambda \leq \int_{\mathbb{R}^n}\varphi_i \scaleobj{.9}{\dfrac{\det(\nabla^2 \varphi_{i+1})}{\lambda(A)}}\,d\lambda.
\]
Subtracting $\mathcal{G}\Big(\scaleobj{.9}{\dfrac{\MA(\varphi_{i+1})}{\lambda(A)}}\Big)$ from both sides and applying the equality from Lemma \ref{RicG} yields
\begin{multline*}
\Big\langle \varphi_{i+1},\,\scaleobj{.9}{\dfrac{\MA(\varphi_{i+1})}{\lambda(A)}}\Big\rangle - \mathcal{G}\Big(\scaleobj{.9}{\dfrac{\MA(\varphi_{i+1})}{\lambda(A)}}\Big) \leq \int_{\mathbb{R}^n} \varphi_i \scaleobj{.9}{\dfrac{\det(\nabla^2 \varphi_{i+1})}{\lambda(A)}}\,d\lambda \\ + \int_{\mathbb{R}^n} \scaleobj{.9}{\dfrac{\det(\nabla^2 \varphi_{i+1})}{\lambda(A)}}\log\Big(\scaleobj{.9}{\dfrac{\det(\nabla^2 \varphi_{i+1})}{\lambda(A)}} \Big)\,d\lambda.
\end{multline*}
Since $\{\varphi_i\}$ is a Monge--Amp\`{e}re iteration solving equation (\ref{MAiteration}), it follows that  $\scaleobj{.9}{\dfrac{\det(\nabla^2 \varphi_{i+1})}{\lambda(A)}} = \dfrac{h\circ \varphi_i}{\|h\circ\varphi_i\|_1}$.  Thus, the above inequality implies
\begin{align*}
\Big\langle &\varphi_{i+1},\,\scaleobj{.9}{\dfrac{\MA(\varphi_{i+1})}{\lambda(A)}}\Big\rangle - \mathcal{G}\Big(\scaleobj{.9}{\dfrac{\MA(\varphi_{i+1})}{\lambda(A)}}\Big) \leq  \int_{\mathbb{R}^n} \varphi_i \dfrac{e^{-\varphi_i}}{\|e^{-\varphi_i}\|_1}\,d\lambda + \int_{\mathbb{R}^n} \dfrac{e^{-\varphi_i}}{\|e^{-\varphi_i}\|_1}\log\left(\dfrac{e^{-\varphi_i}}{\|e^{-\varphi_i}\|_1}\right)\,d\lambda\\
&=\|e^{-\varphi_i}\|_1^{-1} \int_{\mathbb{R}^n} \varphi_i \,e^{-\varphi_i}\,d\lambda - \|e^{-\varphi_i}\|_1^{-1} \int_{\mathbb{R}^n} \varphi_i\, e^{-\varphi_i}\,d\lambda - \log\big(\|e^{-\varphi_i}\|_1\big) \int_{\mathbb{R}^n} \dfrac{e^{-\varphi_i}}{\|e^{-\varphi_i}\|_1} \,d\lambda\\
&= -\log\big(\|e^{-\varphi_i}\|_1\big) = \mathcal{F}(\varphi_i).
\end{align*}
\end{proof}

Lemma \ref{last} concludes the proof of Hypothesis \ref{B3}, so Theorem \ref{introthm} implies Theorem \ref{MAiteration2}.

\subsection{Smooth toric Fano manifolds}\label{smoothtoric}
A compact K\"{a}hler manifold $X$ of dimension $n$ is \textit{toric} if there is an effective holomorphic action of the complex torus $\mathbb{C}^*{}^n$ with an open, dense orbit $X_0$.  Delzant \cite{Delzant} showed toric manifolds are characterized by \textit{Delzant} polytopes.  A polytope $P\subset\mathbb{R}^n$ is \textit{Delzant} if for each vertex $v$ there exists a transformation $A\in Sl_n(\mathbb{Z})$ such that
\begin{equation}\label{Delzantcondition}
A\,(P-v) \cap B_\epsilon = \{x\in\mathbb{R}^n \mid x_i\geq 0 \text{ for }i=1,\ldots,n\} \cap B_{\epsilon}.
\end{equation}
In other words, a neighborhood of each vertex of $P$ is $Sl_n(\mathbb{Z})$-equivalent to a neighborhood of the origin in the first orthant. Guillemin \cite{Guillemin} explained the connection to K\"{a}hler geometry and proved the bijective correspondence 
\begin{equation}\label{corr}
\{\text{\,smooth, polarized, toric K\"{a}hler manifolds }(X_P,L_P)\,\} \longleftrightarrow \{ \,\text{integral Delzant polytopes } P \,\}.
\end{equation}
A polarized K\"{a}hler manifold $(X_P,L_P)$ is K\"{a}hler manifold $X_P$ paired with a line bundle $L_P$.  Integral polytopes are those whose vertices lie in the integer lattice $\mathbb{Z}^n$.  We refer to \cite{Cox} for a detailed exposition on the correspondence.  

A Delzant polytope $P$ can be defined as the intersection of half-spaces by

\begin{equation}\label{dualPdef}
P = \medcap_{i=1}^M \big\{y\mid l_i(y)\geq 0\big\} \text{ for } l_i(y)=n_i\cdot y + \lambda_i,
\end{equation}
where $\lambda_i \in \mathbb{Z}$ and $n_i\in \mathbb{Z}^n$ is primitive, meaning its components are relatively prime over $\mathbb{Z}$. The vectors $\{n_i\}_{i=1}^M$ are inward pointing normals for the codimension one facets of $P$.  Under the correspondence (\ref{corr}), the codimension one facets of $P$ can be identified with the codimension one submanifolds of $X_P$ which are invariant under the $\mathbb{C}^*{}^n$ action.  We define $D_i$ to be the toric submanifold corresponding to the facet $\{l_i(y)=0\} \cap P$, and a \textit{toric divisor} is a formal linear combination of the form $\sum_{i=1}^M a_i\,D_i$.  The line bundle $L_P$ can be defined as

\[
L_P = \mathcal{O} \bigg(\sum_{i=1}^M \lambda_i\,D_i\bigg),
\]
the line bundle associated to the divisor $\sum_{i=1}^M \lambda_i \,D_i$.

We are specifically interested in the \textit{Fano} case when $L_P = -K_{X_P}$, the anticanonical line bundle of $X_P$, and we call $(X_P,L_P)$ a \textit{polarized toric Fano manifold}.  this corresponds to the case when $\lambda_i=1$ for each $i$, and for this reason we call polytopes \textit{Fano} if $\lambda_i=1$.  We have the restricted correspondence
\[
\{\,\text{polarized toric Fano manifolds }(X_P,L_P)\,\} \longleftrightarrow \{\,\text{integral Delzant Fano polytopes }P\,\}.
\] 

\subsection{K\"{a}hler Ricci iteration on toric Fano manifolds}\label{proofofRI}
Let $X$ be toric K\"{a}hler manifold, and let $X_0$ be the open dense orbit on which $\mathbb{C}^*{}^n$ acts effectively. There is a coordinate chart $\mathbb{C}^*{}^n$ on $X_0$ with coordinates $z=(z_1,\ldots,z_n)$ such that the $\mathbb{C}^*{}^n$ action is given by componentwise multiplication.  In these coordinates the real torus $(S^1)^n \subset \mathbb{C}^*{}^n$ acts by
\[
(z_1,\ldots,z_n) \mapsto (e^{\sqrt{-1}\theta_1}z_1,\ldots ,e^{\sqrt{-1}\theta_n}z_n).
\]
If we define action, angle coordinates $(x_i,\alpha_i) \in \mathbb{R}^n \times \big(\mathbb{R}/ 2\pi\mathbb{Z}\big)^n$ on the open orbit by
\[
z_i = \scaleobj{1.3}{e^{\frac{x_i}{2}+\sqrt{-1}\alpha_i}},
\]
then functions which are invariant under the action of $T^n$ only depend on the $x_i$ coordinates.  If $\omega$ is a toric K\"{a}hler metric on $X_P$, meaning its invariant under the action of $(S^1)^n$, then in it has a potential $\phi: \mathbb{R}^n \to \mathbb{R}$ such that
\begin{equation}\label{omegadef}
\omega\big|_{\mathbb{C}^*{}^n} = \sqrt{-1}\partial\overline{\partial}\,\phi = \phi_{ij} \dfrac{dz_i}{z_i}\wedge \dfrac{d\overline{z}_j}{\overline{z}_j} = \phi_{ij}\,dx_i\wedge d\alpha_j,
\end{equation}
for $\phi_{ij} = (\nabla^2 \phi)_{ij}$ the components of the real Hessian of $\phi$.  Guillemin \cite{Guillemin} proved that if $\phi$ is a potential for a smooth K\"{a}hler metric $\omega \in c_1(L_P)$ then $\phi$ is a smooth, strictly convex function with $\nabla \phi(\mathbb{R}^n) = \Int P$.

Let $(X_P,L_P)$ be a polarized, toric Fano manifold, and let $\omega$ and $\eta$ be K\"{a}hler metrics in $c_1(L_P)$ solving $\Ric(\omega)=\eta$.  If then $\phi$ and $\psi$ are open orbit potentials for $\omega$ and $\eta$ respectively, then they satisfy the second boundary problem
\begin{equation}\label{KE}
\begin{cases}
    \dfrac{\det(\nabla^2 \phi)}{\lambda(P)} = \dfrac{e^{-\psi}}{\|e^{-\psi}\|_1}\\
    \nabla \phi(\mathbb{R}^n) = \Int P.
\end{cases}
\end{equation}

Let $\{\omega_i\}_{i\in\mathbb{N}}$ be K\"{a}hler--Ricci iteration on $X_P$ a Fano manifold $X$ which satisfies
\begin{equation}\label{RIt}
\Ric(\omega_{i+1})=\omega_i.
\end{equation}
If $\{\phi_i\}_{i\in\mathbb{N}}$ are open orbit potentials for $\{\omega_i\}$, then  equation (\ref{KE}) implies
\begin{equation}\label{toricRI2}
\begin{cases}
\dfrac{\det(\nabla^2 \phi_{i+1})}{\lambda(P)} =  \dfrac{e^{-\phi_i}}{\|e^{-\phi_i}\|_1}\\
\nabla \phi_{i+1}(\mathbb{R}^n) = \Int P.
\end{cases}
\end{equation}
This is the Monge--Amp\`{e}re iteration (\ref{introiteration}) with $h(t)=e^{-t}$. We can choose any $\tau\in\mathbb{R}$ and require $\int_A \phi_i^*\,d\lambda=-\tau$ for each $\phi_i$ because the addition of a constant doesn't affect the right hand side of equation (\ref{toricRI2}).  Thus, the open orbit potentials $\{\phi_i\}_{i\in\mathbb{N}}$ for the K\"{a}hler--Ricci iteration are a normalized Monge--Amp\`{e}re iteration.

Now we can prove Theorem \ref{introthm3}.

\begin{proof}
Let $\{\phi_i\}$ be potentials for $\{\omega_i\}$ in the open orbit, and normalize their additive constants so that $\int_P \phi_i^*\,d\lambda = 0$.   As shown above, they satify equation (\ref{MAiteration}) with $h(t)=e^{-t}$, $A=\Int\,P$, and $\tau=0$.  

Wang--Zhu \cite{WZ} proved if $X_P$ admits a smooth K\"{a}hler--Einstein metric, then the barycenter of $P$ lies at the origin.  Thus, $\Int \,P$ satisfies conditions (\ref{Aconditions}).  By Theorem \ref{MAiteration2}, there exist constants $\{a_i\}$ such that $\widetilde{\phi}_i = \phi_i(x+a_i)$ converges smoothly to $\phi$, a smooth solution of equation (\ref{KE}).  The metric $\omega\big|_{\mathbb{C}^*{}^n} = \sqrt{-1}\partial \overline{\partial}\, \phi$ solves $\Ric(\omega)=\omega$ in the open orbit.  Since the barycenter lies at the origin, there is a smooth K\"{a}hler-Einstein metric which solves the same equation as $\omega$, so $\omega$ extends to a smooth K\"{a}hler-Einstein metric on all  of $X_P$. 

If $\{g_i\}_{i\in\mathbb{N}}$ are the unique automorphisms of $X_P$ which are given by $g(z_i) = e^{a_i/2}\,z_i$ on the coordinate patch associated to $P$, then the metrics $\widetilde{\omega}_i = g_i^*(\omega_i)$, which are given by
\[
\widetilde{\omega}_i = \sqrt{-1}\,\partial\overline{\partial}\,\widetilde{\phi}_i
\] 
on the open orbit, converge smoothly to $\omega$ by Theorem \ref{MAiteration2}.
\end{proof}

\subsection{Ding Functional and Mabuchi K-Energy}\label{Kahlerfunctionals}
For the definitions of all the functionals in this section we refer to equations (5) and (8) in \cite{Darvas}.  Let $(X_P,L_P)$ be a smooth, polarized toric Fano manifold, and let $\omega$ be a reference metric in $c_1(L_P)$. If $\omega_{\varphi} = \omega + \sqrt{-1}\partial\overline{\partial} \varphi$ is another K\"{a}hler metric, then the \textit{Aubin-Mabuchi functional} $\AM(\varphi)$ is defined implicitly by
\[
\dfrac{d}{dt}\bigg|_{t=0} \AM(\varphi + t\, v) = \dfrac{1}{V} \int_{X_P} v\, \MA_{\mathbb{C}}(\varphi), \hspace{5mm} \AM(0)=0,
\]
where $\MA_{\mathbb{C}}( \varphi) = (\sqrt{-1}\partial \overline{\partial}\,\varphi)^n$ is the complex Monge-Ampere operator, and $V=\int_{X_P} \omega^n$ is the volume of $X_P$.  Since $X_P$ is Fano, we can choose $\omega$ so that $\Ric(\omega)$ is a positive form. If $\varphi$ is a toric invariant function, then
\[
\MA_{\mathbb{C}}(v) = (\sqrt{-1} \partial \overline{\partial} \,v)^n = (v_{ij}\,dx_i\wedge d\alpha_j)^n = \det(\nabla^2 v)\,dx_1\wedge\cdots\wedge dx_n\wedge d\alpha_1 \wedge \cdots \wedge d\alpha_n.
\]

Since $X_P$ is toric, we can write any metric in terms of its potential in the open orbit $\mathbb{C}^*{}^n$.  We denote the reference metric $\omega$ and and any other metric $\omega_{\varphi}$ by
\[
\omega = \sqrt{-1}\partial \overline{\partial} \,\psi \hspace{5mm} \omega_{\varphi} = \sqrt{-1} \partial \overline{\partial} \,( \psi + \varphi) = \sqrt{-1}\partial \overline{\partial}\, \phi.
\]
It is convenient to define the Aubin-Mabuchi functional on the K\"{a}hler potential $\phi$ without regard to the reference metric $\omega$.  This simplification does not affect the definition because we assume $\AM(0)=0$. 
\begin{align*}
    \dfrac{d}{dt}\Big|_{t=0} \AM(\phi+t\,v) &= \dfrac{1}{V} \int_{X_P} v\, \det(\nabla^2 \phi) = dx_1 dx_n\wedge d\alpha_1 \wedge \cdots \wedge d\alpha_n\\
    &= \dfrac{1}{(2\pi)^n \,\lambda(P)} (2\pi)^n \int_{\mathbb{R}^n} v\,\det(\nabla^2 \phi)\, dx_1\wedge\cdots \wedge dx_n\\
    &= \dfrac{d}{dt}\Big|_{t=0} \Big( - \lambda(P)^{-1} \int_P \phi^*_t(y) \, dy_1 \wedge \cdots \wedge dy_n \Big),
\end{align*}
where we used the change of variables $y = \nabla \phi(x)$ and the first variation formula for the Legendre transform in the last equality.  We refer the reader to \cite[pg.~85]{Yanirthesis} for a proof and exposition of the variations of the Legendre transform.  Thus, for a convex function $\phi$, the Aubin--Mabuchi functional of the metric $\sqrt{-1} \partial \overline{\partial}\,\phi$ is given by 
\[
\AM(\phi)=-\dfrac{1}{\lambda(P)} \int_P \phi^*\,d\lambda.
\]

Now we return to the reference metric $\omega$.  Define $f_\omega \in C^\infty(X_P)$ to be the unique function satisfying
\[
\sqrt{-1}\partial \overline{\partial}\, f_\omega = \Ric(\omega) -\omega,\hspace{5mm}  \int_{X_P}e^{f_\omega} \omega^n = V.
\]
Let $\omega = \sqrt{-1} \partial\overline{\partial}\,\psi$ in the open orbit coordinates, and assume $\omega$ is toric invariant.  We can add a constant to $\psi$ so that
\[
f_\omega = -\log\big(\det(\nabla^2 \psi)\big) - \psi + \langle a,x \rangle
\]
for some constant $a\in\mathbb{R}^n$, which we will show must equal $0$.  Since $\omega$ was chosen so that $\Ric(\omega)$ is a positive form in $2\pi\,c_1(X_P)$, it follows that both $-\log\big(\det(\nabla^2 \psi)\big)$ and $\psi$ are potentials for K\"{a}hler metrics in the same class.  In particular, they both have gradient images equal to $\Int\,P$.  If $a\neq 0$, then $f_\omega$ is unbounded because the gradient of $f_\omega$ equals $a$ as $|x|\to\infty$.  The unboundedness of $f_\omega$ contradicts the fact that $f_\omega$ is a smooth function on $X_P$, so $a=0$.

The \textit{Ding functional} $\mathcal{D}(\omega_{\varphi})$ is given by
\[
\mathcal{D}(\omega_{\varphi}) = -\AM(\varphi) -\log \dfrac{1}{V}\int_{X_P} e^{f_\omega - \varphi}\,\omega^n.
\]
Since $\omega_{\varphi} = \sqrt{-1}\partial \overline{\partial} \,(\psi + \varphi) = \sqrt{-1} \partial \overline{\partial} \,\phi$ we will write $\mathcal{D}$ as a function of the potential $\phi$,
\begin{align*}
\mathcal{D}(\phi) &= \lambda(P)^{-1} \int_P \phi^*\,d\lambda - \log \dfrac{1}{\lambda(P)} \int_{\mathbb{R}^n} \scaleobj{1.2}{e^{-\log\,\det(\nabla^2 \psi) - \psi  - (\phi - \psi)}}\,\det(\nabla^2 \psi) \,d\lambda\\
&= \lambda(P)^{-1} \int_P \phi^*\,d\lambda - \log \dfrac{1}{\lambda(P)} \,\int_{\mathbb{R}^n} e^{-\phi} \,d\lambda\\
&= \lambda(P)^{-1} \int_P \phi^*\,d\lambda  -\log \bigg( \int_{\mathbb{R}^n} e^{-\phi}\,d\lambda \bigg) + \log\big(\lambda(P)\big).
\end{align*}
Thus, on the toric variety $X_P$ the Ding functional is related to the functional $\mathcal{F}$ by
\[
\mathcal{D}(\phi) = \mathcal{F}(\phi) + \lambda(P)^{-1} \int_P \phi^*\,d\lambda + \log\big(\lambda(P)\big).
\]

The \textit{Mabuchi K-energy} is given by 
\[
\mathcal{K}(\phi) = \lambda(P)^{-1} \bigg( \int_{\mathbb{R}^n} \phi\,\det(\nabla^2 \phi)\,d\lambda + \int_{\mathbb{R}^n} \det(\nabla^2 \phi)\,\log\big(\det(\nabla^2 \phi)\big)\,d\lambda + \int_P \phi^*\,d\lambda \bigg),
\]
as is shown in Section 4.2 of Berman and Berndttson \cite{BB} and Section 2.4 of Donaldson \cite{Donaldson}.  In Section \ref{iteration2} we showed 
\[
\mathcal{G}\Big(\scaleobj{.9}{\dfrac{\MA(\varphi)}{\lambda(P)}}\Big) = -\int_{\mathbb{R}^n} \scaleobj{.9}{\dfrac{\det(\nabla^2 \phi)}{\lambda(P)}}\,\log\Big(\scaleobj{.9}{\dfrac{\det(\nabla^2 \phi)}{\lambda(P)}}\Big)\,d\lambda.
\]
Thus,
\[
\mathcal{K}(\phi) = \Big\langle \phi,\,\scaleobj{.9}{\dfrac{\MA(\phi)}{\lambda(P)}}\Big\rangle - \mathcal{G}\Big(\scaleobj{.9}{\dfrac{\MA(\phi)}{\lambda(P)}}\Big) +\log\big(\lambda(P)\big) + \int_P \phi^*\,d\lambda .
\]

If $\{\phi_i\}$ are potentials for $\{\omega_i\}$ solving the Ricci iteration (\ref{RIt}), then Lemma \ref{functionallimit} implies
\[
\mathcal{F}(\phi_i) \geq \Big\langle \phi_{i+1},\,\scaleobj{.9}{\dfrac{\MA(\phi_{i+1})}{\lambda(P)}}\Big\rangle - \mathcal{G}\Big(\scaleobj{.9}{\dfrac{\MA(\phi_{i+1})}{\lambda(P)}}\Big) \geq \mathcal{F}(\phi_{i+1}).
\]
Adding $\int_P \phi^* \,d\lambda + \log\big(\lambda(P)\big)$ to every term we get for $X_P$ a toric K\"{a}hler manifold, along the Ricci iteration 
\[
\mathcal{D}(\phi_i) \geq \mathcal{K}(\phi_{i+1}) \geq \mathcal{D}(\phi_{i+1}).
\]
This inequality is a special case of a more general fact from K\"{a}hler geometry, that along the Ricci iteration
\[
\mathcal{D}(\omega_i) \geq \mathcal{K}(\omega_{i+1}) \geq \mathcal{D}(\omega_{i+1}),
\]
which is shown in Rubinstein \cite{Rubinstein}.

\section{Affine iteration}\label{AIsection}

This section is organized as follows.  In Subsection \ref{iteration1} we prove Theorem \ref{MAiteration1} about the convergence of the Monge--Amp\`{e}re iteration (\ref{MAiteration}) with $h(t) = t^{-(n+p+1)}$ for $p>0$.   In Subsections \ref{affineimmersions}\,--\,\ref{affinespheres} we give the necessary background on affine differential geometry.  In Subsection \ref{leggraphs} we discuss the specific example of affine immersions which are graphs of Legendre transforms, and in Subsection
\ref{affineiterationsection} we define the affine iteration on these graph immersions and prove Theorem \ref{affineiterationconverges} about its convergence.

\subsection{The Monge--Amp\`{e}re iteration for \texorpdfstring{$h(t)=t^{-(n+p+1)}$}{a}}\label{iteration1}

To prove Theorem \ref{MAiteration1} about the convergence of the Monge--Amp\`{e}re iteration for $h(t)=t^{-(n+p+1)}$ when $p>0$ we must verify Hypotheses \ref{hypB} and apply Theorem \ref{introthm}.

\noindent \textit{Hypothesis \ref{B1}:}

When $p>0$ Hypothesis \ref{B1} is satisfied trivially.

\noindent \textit{Hypothesis \ref{E}:}

Klartag proved \cite[Theorem~3.10]{Klartag}, that the second boundary problem
\begin{equation}\label{AHE}
    \begin{cases}
            \dfrac{\det(\nabla^2 \varphi)}{\lambda(A)} =\dfrac{ \varphi^{-(n+p+1)}}{\|\varphi^{-(n+p+1)}\|_1}\\
            \nabla \varphi( \mathbb{R}^n) = A
    \end{cases}
\end{equation}
has smooth, strictly convex solutions, unique up to translations $\varphi(x-a)$ and dilations $\varphi_a(x) = a\,\varphi(x/a)$ if and only if $A$ has barycenter at the origin.  

\begin{lemma}\label{negativeintegral}
If $\varphi$ is a smooth, convex solution to equation (\ref{AHE}) for $p>0$, then $\int_A \varphi^*\,d\lambda<0$.
\end{lemma}
\begin{proof}
The Legendre transform of $\varphi$ is given by $\varphi^*\big(\nabla \varphi(x)\big) = \langle x,\nabla \varphi(x)\rangle - \varphi(x)$.  Thus, the change of variables $y= \nabla\varphi(x)$ implies
\[
    \int_A \varphi^*\,d\lambda = \int_{\mathbb{R}^n} \big( \langle x,\nabla \varphi(x) \rangle  - \varphi(x) \big)\,\det(\nabla^2 \varphi)\,d\lambda.
\]
Since $\varphi$ satisfies equation (\ref{AHE}),
\[
\int_{\mathbb{R}^n}\big( \langle x,\nabla \varphi(x)\rangle - \varphi(x) \big)\,\det(\nabla^2 \varphi)\,d\lambda
= C\int_{\mathbb{R}^n} \langle x,\,\nabla \varphi(x) \rangle \,\varphi(x)^{-(n+p+1)}\,d\lambda - C\int_{\mathbb{R}^n} \varphi^{-(n+p)}\,d\lambda
\]
for some constant $C>0$.  We can simplify the first integral on the second line by an integration by parts.  We restrict to the ball of radius $R$ to show
\begin{align*}
    \int_{B_R} \langle x,\nabla \varphi(x) \rangle \,\varphi^{-(n+p)}\,d\lambda &= \dfrac{-1}{n+p} \int_{B_R} \langle x, \nabla (\varphi^{-(n+p)})\rangle \,d\lambda\\
    &= \dfrac{n}{n+p} \int_{B_R} \varphi^{-(n+p)}\,d\lambda - \dfrac{1}{n+p} \int_{\partial B_R} \varphi^{-(n+p)}\, \Big\langle x, \dfrac{x}{|x|} \Big\rangle\,dS
\end{align*}
where $dS$ is the surface measure on the sphere of radius $R$.  Since $\varphi$ is positive and convex there are positive constants $c$ and $r$ such that $\varphi(x) \geq c+ r\,|x|$.  Thus, we can bound the integral
\[
\int_{\partial B_R} \varphi^{-(n+p)}\, \Big\langle x, \dfrac{x}{|x|} \Big\rangle\,dS \leq (c+r\,R)^{-(n+p)}\,R\,R^{n-1}\,\omega_n.
\]
This bound goes to $0$ as $R\to\infty$ because $p>0$.  Thus, we know 
\[
 C\int_{\mathbb{R}^n} \langle x,\,\nabla \varphi(x) \rangle \,\varphi(x)^{-(n+p+1)}\,d\lambda = C\,\dfrac{n}{n+p} \int_{\mathbb{R}^n} \varphi^{-(n+p)}\,d\lambda.
\]
Returning to the integral of the Legendre transform we see
\[
\int_A \varphi^*\,d\lambda = -C\,\dfrac{p}{n+p} \int_{\mathbb{R}^n} \varphi^{-(n+p)}\,d\lambda <0.
\]

\end{proof}

The Legendre transform of the dilation satisfies
\[
\varphi_a^*(y) = \sup_{x\in\mathbb{R}^n} \{\,\langle x,y\rangle - \varphi_a(x)\,\} = a\,\sup_{x\in\mathbb{R}^n} \{\,\langle x/a, y\rangle - \varphi(x/a)\,\} = a\,\varphi^*(y).
\]
If $\varphi$ is a solution to equation (\ref{AHE}) then $\int_A \varphi^*\,d\lambda<0$ by Lemma \ref{negativeintegral}, so specifying $\int_A \varphi^*\,d\lambda = -\tau<0$ is equivalent to specifying a unique dilation.  Thus the second boundary problem (\ref{ourpde}) has smooth, strictly convex solutions unique up to translations.

\noindent \textit{Hypothesis \ref{B2}:}

We define $g(s,t)=s/t$.  To verify Hypothesis \ref{B2} we note $g$ is decreasing in $t$ when $s$ is positive, and $g(s,g(s,t)) = s/(s/t)=t$.

\noindent \textit{Hypothesis \ref{B3}:}

We let $s=n+p$ for notational convenience and define $\mathcal{F}_s$ to be the functional $\mathcal{F}$ associated to $h(t)=t^{-(s+1)}$.
First, we compute $\mathcal{F}_s$.  We integrate $H(t) = \int_t^{\infty} x^{-(s+1)}\,d\lambda = s^{-1}\,t^{-s}$, so $H^{-1}(t) = (s\,t)^{-1/s}$. Thus,
\[
\mathcal{F}_s(f) = H^{-1}\big(\|H \circ f\|_1\big) = \bigg( s \int_{\mathbb{R}^n} \frac{1}{s}\,f^{-s}\,d\lambda \bigg)^{-1/s} = \|f\|_{-s}
\]
for any $f\in \mathcal{C}_{\lin} = \{f:\mathbb{R}^n \to (\tau,\infty) \mid f \text{ continuous, and } f(x)/(1+|x|)\text{ bounded}\,\}$.

Next, we compute the definition of $\mathcal{G}_s$ defined in equation (\ref{Gdef}) with $g(s,t)=s/t$:
\[
\mathcal{G}_s(\mu) = \inf \big\{ \,\langle f,\,\mu\rangle \,\|f\|_{-s}^{-1} \mid f\in\mathcal{C}_{\lin}\,\big\}.
\]
For $g(s,t)=s/t$ Hypothesis \ref{B3} says
\[
\Big\langle \varphi_{i+1},\, \scaleobj{.9}{\dfrac{\MA(\varphi_{i+1})}{\lambda(A)}} \Big\rangle \,\mathcal{G}_s \Big(\scaleobj{.9}{\dfrac{\MA(\varphi_{i+1})}{\lambda(A)}}\Big)^{-1} \leq \mathcal{F}_s(\varphi_i).
\]
To prove this inequality, we compute $\mathcal{G}_s(\mu)$ when $\mu$ has a continuous density.

\begin{proposition}\label{Gequality}
Let $h$ be a continuous, nonnegative function on $\mathbb{R}^n$.  If $h\,\lambda$ is in $\mathcal{P}_1$, then
\[
\mathcal{G}_s(h\,\lambda) = \|h\|_{\frac{s}{s+1}}.
\]
\end{proposition}
\begin{proof}
This proof relies on the reverse H\"{o}lder inequality which states that if $p\in (0,1)$ and $q<0$ are numbers such that $1/p + 1/q =1$, and $f>0$,  $h\geq 0$ are functions on $\mathbb{R}^n$, then

\begin{equation}\label{reverseholder}
\|h\|_p \, \|f\|_q \leq \int_{\mathbb{R}^n} f\,h\,d\lambda.
\end{equation}
Applying (\ref{reverseholder}) with $f\in\mathcal{C}_{\lin}$, $h$ defined in the proposition, $p=s/(s+1)$, and $q=-s$ shows
\[
\|h \|_{\frac{s}{s+1}} \,\|f\|_{-s} \leq \langle f, \,h\lambda \rangle.
\]
Since $f\in\mathcal{C}_{\lin}$ was arbitrary, it follows that 
\begin{equation}\label{firstinequality}
\|h\|_{\frac{s}{s+1}} \leq \inf\{\,\langle f,\,h\lambda \rangle \,\|f\|_{-s}^{-1}\mid f\in\mathcal{C}_{\lin}\,\}=\mathcal{G}_s(h\,\lambda).
\end{equation}

To prove the reverse inequality consider
\begin{equation}\label{c1}
f_k(x) = \min\{h(x)^{-1/(s+1)},\, k\,(1+|x|) \}.
\end{equation}
By assumption $h$ is continuous and integrable, so $h$ is finite, which implies $f_k$ is continuous and positive.  Thus, $f_k \in \mathcal{C}_{\lin}$, and
\[
\mathcal{G}_s(\mu) \leq \langle f_k,\mu \rangle \,\mathcal{F}_s(f_k)^{-1}.
\]
$f_k$ is an increasing sequence of positive functions which converge pointwise to $g_\mu^{-1/(s+1)}$, so
\[
\langle f_k,\mu \rangle = \int_{\mathbb{R}^n} f_k\,h \,d\lambda \to \int_{\mathbb{R}^n} h^{\frac{s}{s+1}}\,d\lambda = \|h\|_{\frac{s}{s+1}}^{\frac{s+1}{s}}
\]
by the monotone convergence theorem.  

$f_k^{-s}$ converges pointwise to $h^{\frac{s}{s+1}}$, and $f_k^{-s}$ is dominated by $f_1^{-s}$ for all $k$.  $f_1^{-s}$ is integrable because it is the maximum of $h ^{\frac{s}{s+1}}$ and  $(1+|x|)^{-s}$ which are both positive and integrable.  $h^{\frac{s}{s+1}}$ is integrable by  (\ref{firstinequality}) and the finiteness of $\mathcal{G}_s$, and $(1+|x|)^{-s}$ is integrable because $s>n$.  Thus,
\begin{equation}\label{c3}
\|f_k\|_{-s}^{-1} \to \left(\int_{\mathbb{R}^n} h^{\frac{s}{s+1}}\,d\lambda \right)^{1/s}
\end{equation}
by the dominated convergence theorem.  Together, (\ref{c1}) and (\ref{c3}) imply
\[
\mathcal{G}_s(h\,\lambda) \leq \|h\|_{\frac{s}{s+1}}.
\]
\end{proof}

Now we use Lemmas \ref{compare} and \ref{Gequality} to prove Hypothesis \ref{B3}.

\begin{lemma}\label{rev4}
$\mathcal{F}_s (\varphi_i) \geq \Big\langle \scaleobj{.9}{\dfrac{\MA(\varphi_{i+1})}{\lambda(A)}},\varphi_{i+1} \Big\rangle \,\mathcal{G}_s \Big(\scaleobj{.9}{\dfrac{\MA(\varphi_{i+1})}{\lambda(A)}}\Big)^{-1}$.
\end{lemma}
\begin{proof} 
 $\varphi_{i+1}$ is smooth, so $\MA(\varphi_{i+1}) = \det(\nabla^2 \varphi_{i+1})\,\lambda$, and since $\{\varphi_i\}$ is a Monge--Amp\`{e}re iteration satisfying equation (\ref{MAiteration})
\[
\int_{\mathbb{R}^n} |x|\,\MA(\varphi_{i+1}) = \| \varphi_i \|_{-(s+1)}^{s+1}\,\int_{\mathbb{R}^n} |x| \, \varphi_i(x)^{-(s+1)} \,d\lambda < \infty
\]
because $\varphi_i$ has linear growth and $s>n$.  Thus $\scaleobj{.9}{\dfrac{\MA(\varphi_{i+1})}{\lambda(A)}} \in \mathcal{P}_1$, and Proposition \ref{Gequality} implies
\[
\mathcal{G}_s \Big(\scaleobj{.9}{\dfrac{\MA(\varphi_{i+1})}{\lambda(A)}}\Big) = \Big\| \scaleobj{.9}{\dfrac{\det (\nabla^2 \varphi_{i+1})}{\lambda(A)}} \Big\|_{\frac{s}{s+1}}.
\]
Since $\int_A \varphi_i^* \,d\lambda = \int_A \varphi_{i+1}^* \,d\lambda$, Lemma \ref{compare} implies
\[
\langle \varphi_{i+1},\,\MA(\varphi_{i+1}) \rangle \leq \langle \varphi_i, \MA(\varphi_{i+1})\rangle.
\]
By equation (\ref{MAiteration}) defining the Monge--Amp\`{e}re iteration,
\[
\Big\langle \varphi_i, \scaleobj{.9}{\dfrac{\MA(\varphi_{i+1})}{\lambda(A)}} \Big\rangle = \left( \int_{\mathbb{R}^n} \varphi_i^{-(s+1)} \,d\lambda\right)^{-1} \, \int_{\mathbb{R}^n} \varphi_i^{-s}\,d\lambda.
\]
Combining the previous equations shows
\begin{align*}
\Big\langle\varphi_{i+1},\,& \scaleobj{.9}{\dfrac{\MA(\varphi_{i+1})}{\lambda(A)}} \Big\rangle \,\mathcal{G}_s \Big(\scaleobj{.9}{\dfrac{\MA(\varphi_{i+1})}{\lambda(A)}} \Big)^{-1} \\ &\leq \left( \int_{\mathbb{R}^n} \varphi_i^{-(s+1)} \,d\lambda\right)^{-1} \, \left(\int_{\mathbb{R}^n} \varphi_i^{-s}\,d\lambda\right) \, \Big\| \scaleobj{.9}{\dfrac{\det(\nabla^2 \varphi_{i+1})}{\lambda(A)}}\Big\|_{\frac{s}{s+1}}^{-1}\\
&= \left( \int_{\mathbb{R}^n} \varphi_i^{-(s+1)} \,d\lambda\right)^{-1} \, \left(\int_{\mathbb{R}^n} \varphi_i^{-s}\,d\lambda\right)\\ & \hspace{3cm} \left( \int_{\mathbb{R}^n} \varphi_i^{-s}\,d\lambda \right)^{-\frac{s+1}{s}} \,\left( \int_{\mathbb{R}^n} \varphi_i^{-(s+1)}\,d\lambda \right)\\
&=\mathcal{F}_s(\varphi_i).
\end{align*}
\end{proof}

Lemma \ref{rev4} concludes the proof of Hypothesis \ref{B3}, so Theorem \ref{introthm} implies Theorem \ref{MAiteration1}.

\subsection{Affine immersions}\label{affineimmersions}

Affine differential geometry is concerned with properties of submanifolds of Euclidean space which are invariant under volume preserving linear transformations $x\mapsto A\,x$ for $A\in Sl_n\mathbb{R}$.  Consider a smooth immersion $f$ of a manifold $M$ as a hypersurface in $\mathbb{R}^{n+1}$.

\[
f:M^n \to \mathbb{R}^{n+1}
\]
In order to decompose vectors in $T\,\mathbb{R}^{n+1}$ into a component tangent to $f(M)$ and a component transversal to $f(M)$ we need to choose a transversal vector field which we will denote by $\xi$ and pointwise by $M \ni x \mapsto \xi_x$.  In Riemannian geometry the transversal vector field is chosen to be the Euclidean unit normal, but this is only invariant under orthogonal transformations, and not all volume preserving transformations.  

We seek to find a unique \textit{affine normal} $\xi$ which is invariant under all volume preserving linear transformations.  Our strategy is to define a metric, connection, and volume forms based on an arbitrary transversal vector field $\xi$ and then to show under a unique choice of $\xi$ these objects are invariant under volume preserving linear transformations. 

Let $\xi \in T\,\mathbb{R}^{n+1} \big|_{f(M)}$, be an arbitrary vector field transverse to $f(M)$ whose pointwise value we'll denote by $\xi_x \in T_{f(x)} \mathbb{R}^{n+1}$.  We decompose $T\,\mathbb{R}^{n+1}$ as

\[
T_{f(x)} \mathbb{R}^{n+1} = f_* (T_x M) + \text{span} \{\xi_x\}.
\]

\noindent We can differentiate vector fields on $T\,\mathbb{R}^{n+1} \big|_{f(M)}$ using the flat connection $D$ on $\mathbb{R}^{n+1}$ by extending them to a neighborhood of $f(M)$.  

\begin{equation} \label{gauss1}
D_X f_*(Y) = f_*(\nabla _X Y) + h(X,Y)\,\xi,
\end{equation}

\noindent defines a torsion free connection $\nabla$ on $TM$ and a symmetric bilinear form $h$ on $TM$.  Differentiating the vector field $\xi$ 

\begin{equation} \label{gauss2}
D_X \xi = -f_*(S\,X) + \tau(X)\,\xi,
\end{equation}

\noindent defines an endomorphism $S$ of $TM$ and a one form $\tau$.  These definitions are invariant under affine transformations, so if $\widetilde{f}=v+Lf(x)$ and $\widetilde{\xi} = L\,\xi$ for $v\in \mathbb{R}^n$ and $L \in Gl_n\,\mathbb{R}$, then $\widetilde{h}=h$, $\widetilde{\nabla}=\nabla$, $\widetilde{S}=S$ and $\widetilde{\tau}=\tau$.  The converse is also true, as shown in Nomizu--Sasaki \cite[Proposition~1.3]{Nomizu}:

\begin{proposition}
Let $f,\widetilde{f}:M \to \mathbb{R}^{n+1}$ be two immersions with transversal vector fields $\xi$ and $\widetilde{\xi}$ respectively.  If $\nabla=\widetilde{\nabla}$, $h=\widetilde{h}$, $S=\widetilde{S}$, and $\tau = \widetilde{\tau}$ then there is an affine transformation $A$ such that $\widetilde{f} = Af$. 
\end{proposition}

It is important to note that $h$, $\nabla$, $S$, and $\tau$ depend on both the immersion $f$ and the choice of transversal vector $\xi$.  Before proceeding, we will compute a simple example

\begin{exmp}
Let $f(x) = (x,F(x))$ for some smooth function $F$.  Let $\{e_i\}_{i=1}^{n+1}$ be the standard basis for $T\,\mathbb{R}^{n+1}$.  Let $\xi = e_{n+1}$  be the transversal vector field. Since $\xi$ is constant, equation (\ref{gauss2}) implies $S=0$ and $\tau=0$.  Let $\{\partial_i=\partial/\partial x_i \}_{i=1}^n$ be a basis of vector fields on $M=\mathbb{R}^n$.  Thus

\[
f_*(\partial_i) = e_i + F_i\,e_{n+1}
\]

\noindent where $F_i=\partial F/\partial x_i$.  

\[
D_{\partial_i} f_*(\partial_j) = F_{ij}\,e_{n+1}
\]

\noindent so equation (\ref{gauss1}) implies $h(\partial_i,\partial_j) = F_{ij}$, and $\nabla_{\partial_i}\partial_j=0$.
\end{exmp}

In the previous example, $h$ is a positive-definite bilinear form if and only if $\nabla^2 F>0$.  In particular, if $F$ is a strongly convex function, then $h$ is positive-definite.  We will only consider the case when $f:M \hookrightarrow \mathbb{R}^{n+1}$ is a \textit{nondegenerate convex immersion}, meaning $f(M)$ is locally the graph of a convex function, and $h$ is positive-definite at every point.

Let $f:M \to \mathbb{R}^{n+1}$ be a nondegenerate convex immersions.  The induced positive-definite bilinear form $h$ is called the \textit{affine metric}.  The induced connection $\nabla$ is called the \textit{affine connection}, and $S$ is called the \textit{affine shape operator}.  All of these objects are \textit{equiaffine}, meaning they are invariant under volume preserving affine transformations of $\mathbb{R}^{n+1}$.  More precisely, if $\widetilde{f}(x) = v + L\,f(x)$ and $\widetilde{\xi}= L\,\xi$ for $L\in SL_{n+1}\mathbb{R}$, then the objects defined by $(f,\xi)$ and $(\widetilde{f},\widetilde{\xi})$ are equal.

In order to define a unique, equiaffine normal vector, we introduce two volume forms on $M$.  The first is the intrinsic volume form determined by the affine metric $h$.  
\begin{equation}\label{volumeform1}
\omega_h(X_1,\ldots,X_n) = \det(h(X_i,X_j)_{ij})^{1/2}.
\end{equation}

\noindent The second volume form is an extrinsic one which is defined in terms of the transversal vector field $\xi$.  

\begin{equation}\label{volumeform2}
\theta(X_1,\ldots,X_n) = \det(f_*(X_1),\ldots,f_*(X_n),\xi).
\end{equation}

\noindent As before, we consider two embeddings $f$ and $\widetilde{f}$ which are affinely equivalent so that $\widetilde{f}=v+L\,f$ and $\widetilde{\xi}=L\,\xi$ for $v\in\mathbb{R}^{n+1}$ and $L\in GL_{n+1}\,\mathbb{R}$.  $\omega_{\widetilde{h}}=\omega_h$ for all $L\in Gl_{n+1}\,\mathbb{R}$, but  $\widetilde{\theta}=det(L)\,\theta$ so $\theta$ is only invariant for $L\in Sl_{n+1}\,\mathbb{R}$.  Both $\omega_h$ and $\theta$ are equiaffine.

The proof of the following proposition is in Nomizu--Sasaki \cite[Theorem~3.1]{Nomizu}.

\begin{proposition}\label{affinenormalexists}
Let $f:M\to\mathbb{R}^{n+1}$ be a strictly convex immersion.  There exists a unique vector field $\xi\in T\,\mathbb{R}^{n+1}\big|_{f(M)}$ such that $\xi$ points to the concave side of $f(M)$, and the induced affine structure satisfies $\tau=0$ and $\theta=\omega_h$.  
\end{proposition}

The unique transversal vector field $\xi$ from Proposition \ref{affinenormalexists} is called the \textit{affine normal}. Nondegenerate convex immersions $f:M\to \mathbb{R}^{n+1}$ paired with their unique affine normals $\xi$ are called \textit{affine immersions}.  Affine immersions induce an equiaffine structure which consists of an affine metric $h$, an affine connection $\nabla$, and an affine shape operator $S$ which satisfy the equations

\[
D_X f_*(Y) = f_*(\nabla_X Y) + h(X,Y)\,\xi \hspace{1cm}\text{and}\hspace{1cm} D_X \xi = -f_*(S(X)).
\]
In this affine immersion we have equality between the volume forms $\theta$ and $\omega_h$.  We define the \textit{affine surface area measure} $\Omega$ to be integration over this volume form.  For any Borel set $U\subset M$ lying in a coordinate patch $x$ of $M$

\begin{equation}\label{ASA1}
\Omega(U) = \int_U \omega_h = \int_U \det(h_{ij})^{1/2} \,dx_1 \wedge\cdots\wedge dx_n.
\end{equation}

\subsection{Alternate formulations of affine surface area}\label{ASAsection}

In Section \ref{affineimmersions} we defined the affine surface area measure in equation (\ref{ASA1}) as integration against the Riemannian volume form of the affine metric.  In this section we will give two more equivalent definitions of affine surface area. 

Let  $f:M \to \mathbb{R}^{n+1}$ be a smooth, nondegenerate convex immersion. For any Borel $U\subset M$ we define the \textit{cone measure} by
\[
\mu(U) = (n+1) \,\lambda\big(\,\cvx(\,0,f(U))\,\big),
\]
where $\lambda$ is Lebesgue measure on $\mathbb{R}^{n+1}$ and $\cvx(0,f(U))$ is the convex hull of the origin and $f(U)$.  The cone measure is $(n+1)$ times the volume of the cone over $f(U)$ with vertex at the origin. 

Let $N(x)$ denote the Euclidean unit normal to $f(M)$ at $f(x)$, with the orientation chosen so that $\langle \xi_x,N(x)\rangle \geq0$. Let $dV_{f(M)} = f^*(\iota_{N}\,dV_{\mathbb{R}^{n+1}})$ be the induced volume form of $M$, and let $dV_{S^n}$ be the standard volume form on the sphere of radius $1$.  The Euclidean unit normal can be thought of as the Gauss map $N:M\to S^n$. We can define the Gaussian curvature, $\kappa:S^n \to \mathbb{R}$, of the immersion $f$ as a function of the Euclidean unit normals to $f(M)$ by the formula
\[
dV_{f(M)} = N^*\Big( \frac{1}{\kappa}\,dV_{S^n}\Big).
\]

If the position vector $f(x)$ is transversal to $f(M)$, then we can define the \textit{support function} of $f$ by 
\begin{equation}\label{support}
\rho(x) = \langle f(x), N(x)\rangle,
\end{equation}
and $\mu$ has the following integral representation:
\[
\mu(U) = \int_U |\rho(x)|\,dV_{f(M)} = \int_{N(U)} \big|\rho(N^{-1}(z))\big|\,\kappa^{-1}(z)\,dV_{S^n}.
\]

Since $\mu$ is defined as the volume of a cone in $\mathbb{R}^{n+1}$ with its vertex at the origin, it is $SL_{n+1}$ invariant, but in contrast to the affine surface area measure, the cone measure is not translation invariant.

Next we will show the affine surface area (\ref{ASA1}) can be expressed in terms of the Gaussian curvature.  First we need the following lemma whose proof can be found in Nomizu--Sasaki \cite[pg~48]{Nomizu}.

\begin{lemma}\label{affinenormalgauss}
Let $f:M\to\mathbb{R}^{n+1}$ be a nondegenerate convex affine immersion with affine normal $\xi$.  Let $N_x$ be the Euclidean unit normal pointing to the concave side of $f(M)$, and let $\kappa(N_x)$ be the Gaussian curvature of $f(M)$ at $f(x)$ .  Then
\[
\langle \xi_x ,N_x\rangle = \kappa^{1/(n+2)}(N_x).
\]
 
\end{lemma}

\begin{proposition}\label{equivalentasa}
Let $f:M \to \mathbb{R}^{n+1}$ be a nondegenerate convex affine immersion, and let $\Omega$ be the affine surface area measure defined in equation (\ref{ASA1}).  Then for every Borel $U\subset M$,
\begin{equation}\label{ASA2}
\Omega(U) = \int_A \kappa(N(x))^{1/(n+2)}\,dV_{f(M)} = \int_{N(U))} \kappa^{-(n+1)/(n+2)}\,dV_{S^n}.
\end{equation}
\end{proposition}

\begin{proof}
Since $\Omega(U) = \int_U \omega_h$, the first inequality is equivalent to showing $\theta = \kappa(N(x))\,dV_{f(M)}$.  Since $\xi$ is the affine normal, the volume form $\omega_h$ is equal to the volume form $\theta$ defined in equation(\ref{volumeform2}) as
\[
\theta = f^*(\iota_{\xi} \,dV_{\mathbb{R}^{n+1}}),
\]
where $\iota_{\xi} \,dV_{\mathbb{R}^{n+1}}$ is the contraction of the Euclidean volume form along the affine normal $\xi$.  The contraction only depends on the component of $\xi$ which is perpendicular to $f(M)$, so it follows that
\[
\theta = \langle \xi, N\rangle\, f^*(\iota_{N} \,dV_{\mathbb{R}^n}) = \langle \xi, N\rangle \,dV_{f(M)}.
\]
Lemma \ref{affinenormalgauss} shows that $\langle \xi, N\rangle = \kappa^{1/(n+2)}$, so
\[
\theta = \kappa(N(x))^{1/(n+2)}\,dV_{f(M)}.
\]
The last equality follows from the Riemannian geometry fact that $N^*(\kappa^{-1}\,dV_{S^n}) = dV_{f(M)}$.
\end{proof}

Lastly we give a dual definition of affine surface area which was introduced by Lutwak in \cite{Lutwak1}:  
\begin{equation}\label{weakASA}
\Omega(U) = \inf \{\, \left(\int_{N(U)} \alpha(z)\,\kappa(z)^{-1}\,dV_{S^n} \right)^{\frac{n+1}{n+2}}\,\left(\int_{N(U)} \alpha(z)^{-(n+1)}\,dV_{S^n} \right)^{\frac{1}{n+2}} \mid \alpha \in C_+(S^n)\,\},
\end{equation}
where $C_+(S^n)$ are continuous, positive functions on $S^n$.

We refer to Lutwak \cite{Lutwak1} for the equivalence between definition (\ref{weakASA}) and (\ref{ASA2}), though we will prove the equivalence in the case of immersions which are graphs of Legendre transforms in Subsection \ref{ASAleggraphs}.  Lutwak also realized that the two integrals inside the infimum of defintion (\ref{weakASA}) can be interpreted in the nonsmooth case.  When $f(M)$ is the boundary of any convex set with the origin in its interior, then Lutwak showed the first integral is a power of the mixed volume of $f(M)$ and the polar set $Q^\circ$, where $Q$ is the convex set whose radial definition function is $\alpha^{-1}$.  The second integral is a power of the volume of $Q$.

\subsection{Affine spheres and dual affine immersions}\label{affinespheres}

An affine immersion is called an \textit{affine spheres} if the shape operator $S=\gamma\,I$ for some constant $\gamma$.  Affine spheres are grouped into three families based on the sign of $\gamma$.  The case $\gamma<0$ is called \textit{hyperbolic}, $\gamma=0$ is called parabolic, and $\gamma>0$ is called elliptic.  Affine spheres also have a geometric interpretation in terms of the affine normals.  The proof of the following proposition can be found in Nomizu--Sasaki \cite[Proposition~3.5]{Nomizu}.

\begin{proposition}
Let $f:M\to\mathbb{R}^{n+1}$ be a convex immersion, $\xi$ be the affine normal, and $S$ be the induced shape operator.  Then $S=0$ if and only if the $\xi$ are all parallel.  And $S=\gamma I$ with $\gamma\neq 0$ if and only if $f(x)+\gamma^{-1}\xi_x=y_0$ a fixed point which is called the center of the affine sphere.
\end{proposition}

\begin{exmp}
Let $M=S^n$ be the sphere which we think of as $\{\,x\in\mathbb{R}^{n+1} \mid |x|=R\,\}$.  We can define the immersion of $S^n$ as a sphere of radius $R$ by $f(x) = R\,x$.   Lemma \ref{affinenormalgauss} can be used to show the affine normal is given by
\[
\xi = \kappa^{1/(n+2)}\,N +f_*(Z)
\]
for some vector field $Z$ on $S^n$ which satisfies $h(Z,X) = -X(\kappa^{1/(n+2)})$ for every vector field $X$ on $S^n$.  Since the immersion $f$ has constant Gaussian curvature $\kappa=R^{-n}$ it follows that $Z=0$.  Thus,
\[
\xi(x) = \kappa^{1/(n+2)}\,N(x) = R^{-n/(n+2)}\,N(x) = -R^{(2n+2)(n+2)}f(x)
\]
which shows $f$ is an affine sphere with center at the origin.  Since affine spheres are invariant under volume preserving affine transformations, this shows all ellipses are affine spheres.
\end{exmp}

It turns out that the only complete elliptic affine spheres are ellipses, but there are many more incomplete elliptic affine spheres. 

Affine spheres naturally arise as dual pairs.  Recall that for any affine immersion $f:M\hookrightarrow \mathbb{R}^{n+1}$ we defined the support function by
\[
\rho(x) = \langle f(x),N(x)\rangle,
\]
where $N(x)$ is the Euclidean normal to $f(M)$ at $f(x)$.  We define the \textit{dual affine immersion} $\nu: M \hookrightarrow \mathbb{R}^{n+1}$ by
\[
\nu(x) = \rho(x)^{-1}\,N(x).
\]
Since the Euclidean normal to the immersion $\nu(x)$ is given by $f(x)/|f(x)|$ it follows that the dual of $\nu$ if $f$ again. The affine geometry of dual affine immersions is explained more fully in Nomizu--Sasaki \cite{Nomizu}.  The following proposition was known to Calabi, and first written down by Gigena \cite{Gigena}.

\begin{proposition}\label{dualaffinesphere}
If $f(M)$ is an affine sphere with its center at the origin, then the dual immersion $\nu(M)$ is also an affine sphere with its center at the origin.  The dual affine spheres are of the same type (elliptic, parabolic, or hyperbolic).
\end{proposition}

\subsection{Graphs of Legendre transforms}\label{leggraphs}

A large class of examplse of affine immersions are graphs of convex functions.  Specifically, we'll study the graphs of Legendre transforms of convex functions with bounded gradient image.

Let $\phi: \mathbb{R}^n \to \mathbb{R}$ be a smooth, strictly convex function.  We define the \textit{Legendre graph immersion} $f_\phi: \mathbb{R}^n \hookrightarrow \mathbb{R}^{n+1}$ by
\begin{equation}\label{fimmersion}
f_\phi(x) = \big(\nabla \phi(x), \langle x,\nabla \phi(x) \rangle  -\phi(x) \big) = \big(\nabla \phi(x) , \,\phi^*(\nabla \phi(x))\big).
\end{equation}

When $\nabla \phi(\mathbb{R}^n) = A$ then we say $f_\phi$ is a \textit{Legendre graph immersion over $A$}, to capture that $f_\phi(\mathbb{R}^n)$ is the graph of the Legendre transform $\phi^*$ over the domain $A$.

\begin{lemma}\label{Legaffinenormal}
Let $\phi$ is a smooth, strictly convex function.  If $f_\phi$ is the Legendre graph immersion defined in (\ref{fimmersion}), then the affine normal of $f_\phi$ is given by 
\[
\xi = -\big(\nabla \psi(x),\big\langle x,\nabla \psi(x) \big\rangle -\psi(x)\big) \hspace{5mm}\text{where}\hspace{5mm} \psi(x) = \det\big(\nabla ^2 \phi(x)\big)^{-1/(n+2)}.
\]
\end{lemma}
\begin{proof}

By Lemma \ref{affinenormalexists} we must verify that $\xi$ points to the concave side of $f(\mathbb{R}^n)$, and the induced affine structure satisfies $\tau=0$ and $\theta=\omega_h$.  

Since $\phi$ is strictly convex, it follows that $\psi>0$.  To verify that $\xi$ points the concave side of $f$, it is enough to show this at one point. Since $\phi$ is strictlyl convex, $\psi$ is positive, so $\xi(0) = (-\nabla \psi (0), \psi(0))$ has positive $e_{n+1}$ component. Since $f(\mathbb{R}^n)$ is the graph of a convex function, it follows that $\xi$ points to the concave side of $f(\mathbb{R}^n)$.

Now we compute the affine metric, shape operator, and the volume forms $\theta$ and $\omega_h$.  Let $\partial_i = \partial/\partial x_i$, and let $\{e_i\}_{i=1}^{n+1}$ be the standard basis for $T\mathbb{R}^{n+1}$.  Then
\[
f_*(\partial_j) = \phi_{jk}\,e_k + \phi_{jk}\,x_k\,e_{n+1}
\]
where repeated indices are summed from $1$ to $n$.  Differentiating once more shows
\begin{align}
D_{\partial_i} f_*(\partial_j) &= \phi_{ijk}\,e_k + (\phi_{ijk}\,x_k + \phi_{ij})\,e_{n+1} \nonumber \\
&=\left( \phi_{ijl}\,\phi^{lk}\,-(n+2)^{-1}\, \phi_{ij}\,\phi^{lm}\,\phi_{lmp}\,\phi^{kp}\right) \,f_*(\partial_k) + \psi^{-1}\,\phi_{ij} \,\xi, \label{secondderiv}
\end{align}
where $\phi^{ij}$ are the components of $(\nabla^2 \phi)^{-1}$, which is well defined becuase $\phi$ is strictly convex.  To verify the second inequality, we must expand $\psi^{-1}\,\phi_{ij} \,\xi$.  Utilizing the identity $\psi_i = -(n+2)^{-1} \,\psi \,\phi^{kl}\,\phi_{kli}$ shows

\begin{align*}
    \psi^{-1}\,\phi_{ij} \,\xi &= -\psi^{-1}\,\phi_{ij} \,(  \psi_k\,e_k + (\langle x,\nabla \psi\rangle - \psi )  \,e_{n+1}) \\
    &= (n+2)^{-1} \,\phi_{ij}\,\phi^{pq}\,\phi_{pqk}\,e_k\, + \left( (n+2)^{-1} \phi_{ij} \,\phi^{pq}\,\phi_{pqk}\,x_k + \phi_{ij} \right) \,e_{n+1}.
\end{align*}
Plugging this into equation (\ref{secondderiv}) along with the definition of $f_*(\partial_k)$ verifies the equality.  The definition of $h$ given in  (\ref{gauss1}), along with equation (\ref{secondderiv}) imply
\begin{equation}\label{fh}
h_{ij} := h(\partial_i,\partial_j) = \psi^{-1} \,\phi_{ij}.
\end{equation}
Now we differentiate $\xi$ to verify that $\tau=0$.
\[
D_{\partial_i} \xi = -\psi_{ij}\,e_j -  \psi_{ij}\,x_j \,e_{n+1} = - \psi_{ij} \,\phi^{jk}\,f_*(\partial_k).
\]
The definition of $S$ given in (\ref{gauss2}) along with the previous equation imply
\begin{equation}\label{Sdef}
S(\partial_{i}) = \psi_{ij}\,\phi^{jk}\,e_k \hspace{5mm}\text{and}\hspace{5mm} \tau=0.
\end{equation}
The last step to verifying $\xi$ is the affine normal is to show the two volume forms $\omega_h$ and $\theta$ defined in equations (\ref{volumeform1}) and (\ref{volumeform2}) are equal.  It is enough to verify their equality on the basis $\{\partial_i\}$.  
\[
\omega_h( \partial_1,\cdots,\partial_n) = \det(h_{ij})^{1/2} = \big(\psi^{-n}\,det(\phi_{ij})\big)^{1/2} = \psi^{-(n+1)},
\]
because $\det(\nabla^2 \phi) = \psi^{-(n+2)}$.  
\begin{align*}
\theta(\partial_1,\cdots,\partial_n) &= \det \begin{pmatrix} f_*(\partial_i) & \cdots & f_*(\partial_n) & \xi \end{pmatrix} \\
&= \det \begin{pmatrix} \nabla^2 \phi & -\nabla \psi\\ x^T\,\nabla^2 \phi & \psi - \langle x,\nabla \psi \rangle \end{pmatrix}\\
&= \det(\nabla^2 \phi) \,\big( \psi - \langle x,\nabla \psi\rangle + x^T \,\nabla^2 \phi \,(\nabla^2 \phi)^{-1}\,\nabla \psi\big) = \psi^{-(n+1)},
\end{align*}
where we've used the formula
\[
\det \begin{pmatrix} A & u \\ v^T & a \end{pmatrix} = \det(A) \,\big(a-v^T\,A^{-1}\,u\big).
\]
Thus $\omega_h = \theta$, and $\xi$ is the affine normal.
\end{proof}

Now we will derive a formula for the dual immersion $\nu:\mathbb{R}^n \hookrightarrow \mathbb{R}^{n+1}$ and find conditions for both $f(\mathbb{R}^n)$ and $\nu(\mathbb{R}^n)$ to be elliptic affine spheres.  First we must derive a formula for the Euclidean normal to $f(\mathbb{R}^n)$, which we will denote $N_f$.  If we let $y=\nabla \phi(x)$, then the image of the immersion $f$ is a graph of the form $\big(y,\phi^*(y)\big)$.  The normal to this graph is
\begin{equation}\label{fnormal}
N_f(x)=\bigg(\dfrac{\nabla \phi^*(y)}{\sqrt{1+|\nabla \phi^*(y)|^2}},\dfrac{-1}{\sqrt{1+|\nabla \phi^*(y)|^2}}\bigg) = \bigg(\dfrac{x}{\sqrt{1+|x|^2}},\dfrac{-1}{\sqrt{1+|x|^2}}\bigg).
\end{equation}
Then the support function $\rho_f(x) = \langle N_f(x),f(x)\rangle$ is given by
\[
\rho_f(x) = \dfrac{\phi(x)}{\sqrt{1+|x|^2}}.
\]
And finally, the dual affine immersion $\nu(x) = \rho_f(x)^{-1}\,N_f(x)$ is given by
\begin{equation}\label{nuimmersion}
\nu(x) = \bigg( \dfrac{x}{\phi(x)},\dfrac{-1}{\phi(x)} \bigg).
\end{equation}

\begin{lemma}\label{anchor}
Let $\phi$ be a smooth, positive, strictly convex satisfying $\nabla \phi(\mathbb{R}^n)=A$, a bounded convex set with the center in its interior.  Then the boundary of the dual immersion $\nu_{\phi}(\mathbb{R}^n)$ equals 
\[
\partial A^\circ \times \{0\}
\]
where $A^\circ = \{\,y\in\mathbb{R}^n \mid \langle x,y\rangle \leq 1 \text{ for all }x\in A\,\}$ is the polar set of $A$, and $\partial A^\circ$ is its boundary.
\end{lemma}
\begin{proof}
Since $\phi$ is a strictly convex function, the boundary of $\nu_{\phi}(\mathbb{R}^n)$ equals
\[
\{\,\lim_{r\to\infty} \nu_{\phi}(r\,x) \mid x\in S^{n-1}\,\}.
\]
Since the origin lies in the interior of $A$, the convex function $\phi$ is proper, so
\[
\lim_{r\to\infty} \dfrac{-1}{\phi(r\,x)} = 0.
\]
Thus, it remains to show 
\[
\lim_{r\to\infty} \dfrac{r\,x}{\phi(r\,x)} \in \partial A^\circ.
\]

For every $\epsilon>0$ define 
\[
A_{\epsilon} = \{\,y\in A \mid B_\epsilon(y) \subset A \,\}.
\]
Since $\nabla \phi(\mathbb{R}^n) = A$ it follows that for every $\epsilon>0$ there is a constant $C_\epsilon$ such that
\[
C_{\epsilon} + \sup_{y\in A_{\epsilon}} \{\,\langle x,y\rangle\,\} \leq \phi(x) \leq \phi(0) + \sup_{y\in A} \{\,\langle x,y\rangle\,\}.
\]
It follows that for $r$ large enough
\[
\dfrac{r}{\phi(0) + \sup_{y\in A} \{\,\langle r\,x,y\rangle\,\}} \leq \dfrac{r}{\phi(r\,x)} \leq \dfrac{r}{C_{\epsilon} + \sup_{y\in A_{\epsilon}} \{\,\langle r\,x,y\rangle\,\}}.
\]
Taking the limit as $r\to\infty$ implies
\[
\dfrac{1}{\sup_{y\in A} \{\,\langle r\,x,y\rangle\,\}} \leq \lim_{r\to\infty} \dfrac{r}{\phi(r\,x)} \leq \dfrac{1}{\sup_{y\in A_{\epsilon}} \{\,\langle r\,x,y\rangle\,\}}.
\]
Since $\epsilon$ was arbitrary, it follows that
\[
\lim_{r\to\infty} \dfrac{r\,x}{\phi(r\,x)} = \dfrac{x}{\sup_{y\in A} \{\,\langle x,y\rangle\,\}}.
\]

For all $y\in A$
\[
\Big\langle \dfrac{x}{\sup_{y\in A} \{\,\langle x,y\rangle\,\}},\,y \Big\rangle \leq 1,
\]
so $\dfrac{x}{\sup_{y\in A} \{\,\langle x,y\rangle\,\}}\in A^\circ$.  For arbitrarily small $\epsilon$ the point $(1+\epsilon) \dfrac{x}{\sup_{y\in A} \{\,\langle x,y\rangle\,\}}$ does not lie in $A^\circ$, so $\dfrac{x}{\sup_{y\in A} \{\,\langle x,y\rangle\,\}}$ must lie in the boundary of $A^\circ$.
\end{proof}

In the language of Klartag \cite{Klartag} we say the immersion $\nu_{\phi}$ has \textit{anchor} $A^\circ$.  We depict a Legendre graph immersion in Figure \ref{LGIfig} and its dual immersion in Figure \ref{dualfig}.  The dashed lines in Figure \ref{LGIfig} form the boundary of the convex set $A$ which the immersion is a graph over.  The dashed lines in Figure \ref{dualfig} form $\partial A^\circ \times \{0\}$, which is the boundary of the dual immersion, as guaranteed by Lemma \ref{anchor}. 

\begin{figure}[h]
\centering
\begin{minipage}{.5\textwidth}
  \centering
  \includegraphics[width=.88\linewidth]{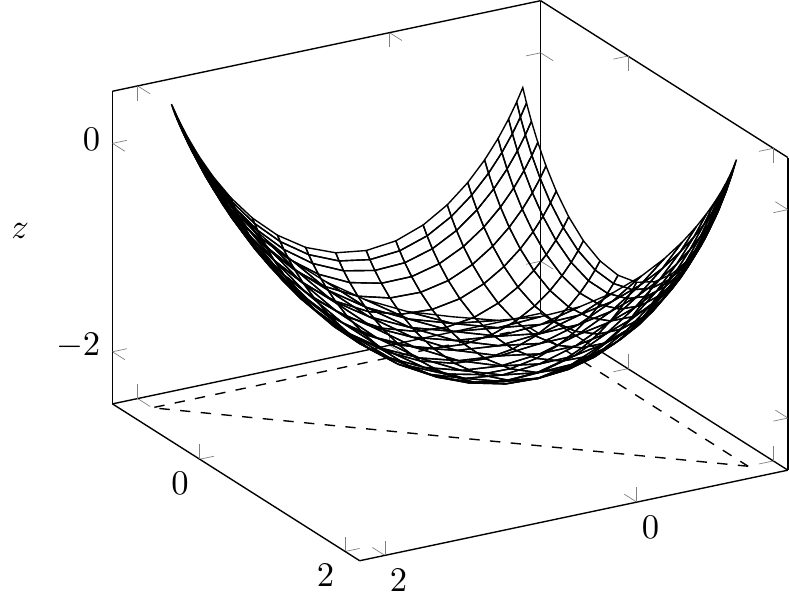}
  \captionof{figure}{A Legendre graph immersion over $A$}
  \label{LGIfig}
\end{minipage}%
\begin{minipage}{.5\textwidth}
  \centering
  \includegraphics[width=.975\linewidth]{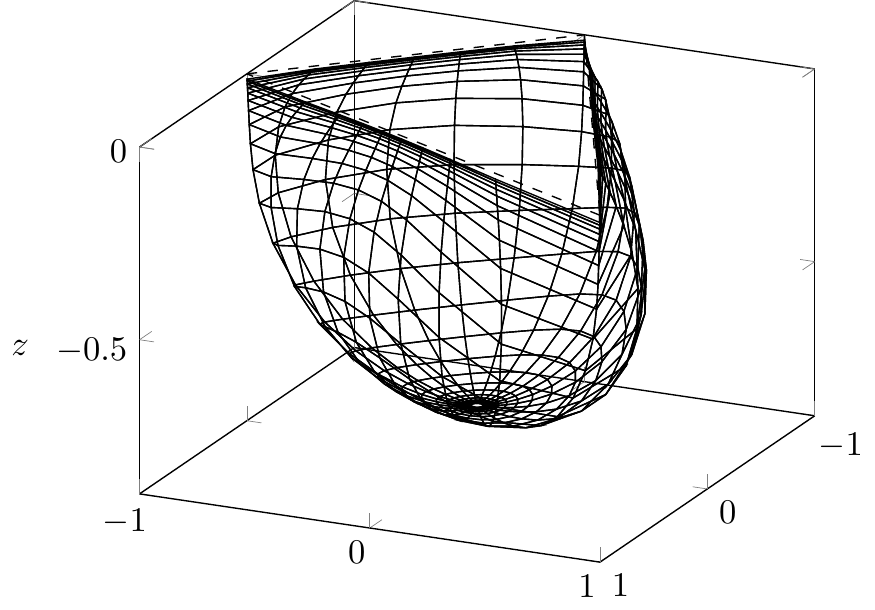}
  \captionof{figure}{The dual immersion with anchor $A^\circ$}
  \label{dualfig}
\end{minipage}
\end{figure}

\begin{proposition}\label{affinesphereequation}
Let $A\subset\mathbb{R}^n$ satisfy conditions (\ref{Aconditions}), and let $\phi$ be a strictly convex function solving the Monge--Amp\`{e}re second boundary problem (\ref{ourpde}) with $h(t)=t^{-(n+2)}$ and $\tau>0$. Then the immersion $f$ defined in equation (\ref{fimmersion}) and its dual affine immersion $\nu$ defined  in equation (\ref{nuimmersion}) are elliptic affine spheres.
\end{proposition}
\begin{proof}
$f(\mathbb{R}^n)$ is an elliptic affine sphere if and only if the affine shape operator $S$ of $f$ satisfies $S=\gamma I$ for $\gamma>0$.  
In equation (\ref{Sdef}) we computed 
\[
S(\partial_{i}) = \psi_{ij}\,\phi^{jk}\,e_k
\]
where $\psi = \det(\nabla^2 \phi)^{-1/(n+2)}$.  Since $\phi$ solves equation (\ref{ourpde}) it follows that
\[
\psi = \lambda(A)^{-1/(n+2)}\,\|\phi\|_{-(n+2)}^{-1}\,\phi,
\]
so the shape operator simplifies to
\[
S(\partial_i) = \lambda(A)^{-1/(n+2)}\, \|\phi\|_{-(n+2)}^{-1} \,\phi_{ij}\,\phi^{jk}\,e_k =  \lambda(A)^{-1/(n+2)}\,\|\phi\|_{-(n+2)}^{-1} \,\partial_i.
\]
Thus, $f(\mathbb{R}^n)$ is an elliptic affine sphere with $\gamma = \lambda(A)^{-1/(n+2)}\,\|\phi\|_{-(n+2)}^{-1}>0$.  Proposition \ref{dualaffinesphere} implies that the dual immersion $\nu(\mathbb{R}^n)$ is also an elliptic affine sphere.
\end{proof}

\subsection{Affine surface area of Legendre graph immersions}\label{ASAleggraphs}

In this  to verify that the three definitions of the affine surface area are equivalent. 

\noindent \textit{Affine surface area definition 1:}

Equation (\ref{ASA1}) defines the affine surface area as
\[
\Omega_f(U) = \int_U \omega_h = \int_U \det(h_{ij})^{1/2}\,dx_1\wedge\cdots\wedge dx_n.
\]
We derived the affine metric $h$ of $f$ in equation (\ref{fh}) to be $h_{ij} := h(\partial_i,\partial_j) = \psi^{-1} \,\phi_{ij}$, where $\psi = \det(\nabla^2 \phi)^{-1/(n+2)}$.  Thus, $\det(h_{ij}) = \psi^{-n}\,\det(\nabla^2 \phi) = \psi^{-(2n+2)}$, and the affine surface are measure is given by
\begin{equation}\label{Legasa}
\Omega(U) = \int_U \psi^{-(n+1)}\,d\lambda = \int_U \det(\nabla^2 \phi)^{\frac{n+1}{n+2}}\,d\lambda.
\end{equation}

\noindent \textit{Affine surface area definition 2:}

Equation (\ref{ASA2}) defines the affine surface area as 
\[
\Omega_f(U) = \int_U \kappa_f\big(N_f(x)\big)^{1/(n+2)}\,dV_{f(\mathbb{R}^n)}.
\]
Since $f(\mathbb{R}^n)$ is the graph of $\phi^*$ we can use the coordinate $y=\nabla \phi$ and the formula for the volume form of a graph of a function to find
\[
dV_{f(\mathbb{R}^n)} = \big(1+|\nabla \phi^*(y)|^2\big)^{1/2}\,dy = \big(1+|x|^2\big)^{1/2}\,\det(\nabla^2 \phi(x))\,dx
\]
where $dx=dx_1\wedge\cdots\wedge dx_n$ is shorthand for the standard volume form on $M=\mathbb{R}^n$.  We derived the formula for the Gauss map $N_f$ in equation (\ref{fnormal}), and it can be used to compute
\[
N_f^*(dV_{S^n}) = \big(1+|x|^2\big)^{-(n+1)/2}\,dx.
\]
Now we use the identity $dV_f = \kappa_f\big(N_f(x)\big)^{-1}\,N_f^*(dV_{S^n})$ to show
\[
\kappa_f\big(N_f(x)\big) = \big(1+|x|^2\big)^{-(n+2)/2}\,\det\big(\nabla^2 \phi(x)\big)^{-1}.
\]
Putting the formulas for $\kappa_f\big(N_f(x)\big)$ and $dV_f$ into equation (\ref{ASA2}) yields
\begin{align*}
\Omega_f(U) &= \int_U \Big( \big(1+|x|^2\big)^{-(n+2)/2}\,\det\big(\nabla^2 \phi(x)\big)^{-1} \Big)^{1/(n+2)}\,\big(1+|x|^2\big)^{1/2}\,\det\big(\nabla^2 \phi(x)\big)\,dx  \\&= \int_U \det(\nabla^2 \phi)^{\frac{n+1}{n+2}}\,d\lambda.
\end{align*}

\noindent \textit{Affine surface are definition 3:}
 
Equation (\ref{weakASA}) gives the dual definition of affine surface area as 
\[
\Omega(U) = \inf \bigg\{\, \bigg(\int_{N(U)} \alpha(z)\,\kappa(z)^{-1}\,dV_{S^n} \bigg)^{\frac{n+1}{n+2}}\,\bigg(\int_{N(U)} \alpha(z)^{-(n+1)}\,dV_{S^n} \bigg)^{\frac{1}{n+2}} \mid \alpha \in C_+(S^n)\,\bigg\},
\]
where $C_+(S^n)$ are continuous, positive functions on $S^n$.  Using the formulas for $N_f$, $\kappa_f$, and $N_f^*(dV_{S^n})$ derived above, it follows that
\begin{align*}
\Omega(U)^{\frac{n+2}{n+1}} &= \inf_{\alpha \in C_+(S^n)} \bigg\{\, \bigg(\int_U \alpha\big(N_f(x)\big)\,\kappa_f\big(N_f(x)\big)^{-1}\,N_f^*(dV_{S^n}) \bigg) \\
&\hspace{5.2cm}\bigg(\int_U \alpha\big(N_f(x)\big)^{-(n+1)}\,N_f^*(dV_{S^n}) \bigg)^{\frac{1}{n+1}} \,\bigg\}\\
&= \inf_{\alpha \in C_+(S^n)} \bigg\{\, \bigg(\int_U \widehat{\alpha}(x)\,\det\big(\nabla^2 \phi(x)\big)\,dx \bigg)\bigg(\int_U \widehat{\alpha}(x)^{-(n+1)}\,dx\,\bigg)^{\frac{1}{n+1}}\bigg\},
\end{align*}
where we define $\widehat{\alpha}\in C_+(\mathbb{R}^n)$ by
\[
\widehat{\alpha}(x) = \sqrt{1+|x|^2}\,\alpha \bigg(\frac{x}{\sqrt{1+|x|^2}},\frac{-1}{\sqrt{1+|x|^2}}\bigg).
\]
Since $\alpha\in C_+(S^n)$, it follows that $0<\alpha(z) \leq C$.  Thus, $\widehat{\alpha}$ is continuous, and 
\[
\widehat{\alpha}(x) \leq C\,\sqrt{1+|x|^2} \leq C\,(1+|x|),
\] 
so $\widehat{\alpha}\in \mathcal{C}_{\lin}  = \{\,f:\mathbb{R}^n \to (0,\infty) \mid f \text{ is continuous, and }f(x)/(1+|x|) \text{ is bounded} \,\}$.
This characterization of $\widehat{\alpha}$ implies
\[
\Omega_f(\mathbb{R}^n)^{\frac{n+2}{n+1}}  = \lambda(A)\,\inf\Big\{ \,\Big\langle \widehat{\alpha},\scaleobj{.9}{\dfrac{\MA(\phi)}{\lambda(A)}} \Big\rangle\,\|\widehat{\alpha}\|_{-(n+1)}^{-1} \mid \widehat{\alpha}\in\mathcal{C}_{\lin}\,\Big\}=\lambda(A)\,\mathcal{G}\Big(\scaleobj{.9}{\dfrac{\MA(\phi)}{\lambda(A)}}\Big).
\]
Proposition \ref{Gequality} implies $\mathcal{G}\Big(\scaleobj{.9}{\dfrac{\MA(\phi)}{\lambda(A)}}\Big) = \Big\|\scaleobj{.9}{\dfrac{\det(\nabla^2 \phi)}{\lambda(A)}}\Big\|_{\frac{n+1}{n+2}}$, so the previous equation again verifies that
\[
\Omega_f(\mathbb{R}^n) = \int_{\mathbb{R}^n} \det(\nabla^2 \phi)^{\frac{n+1}{n+2}}\,d\lambda.
\]
Thus, Proposition \ref{Gequality} can be viewed as a proof of Lutwak's theorem that the dual definition of the affine surface area is equal to the usual definition in the case when the immersion $f$ is given by equation (\ref{fimmersion}).

The last observation we want to  make about the immersion $f$ is the interpretation of the functional $\mathcal{F}(\phi) = \|\phi\|_{-(n+1)}$.  
\begin{lemma}
Let $\phi$ be a smooth, positive, strictly convex function. If $f$ is the immersion defined in equation (\ref{fimmersion}), and $\nu$ is its dual immersion defined  in equation (\ref{nuimmersion}), then 
\[
\mathcal{F}(\phi)^{-(n+1)} = \mu_\nu (\mathbb{R}^n),
\]
where $\mu_{\nu}$ is the cone measure of the dual immersion.
\end{lemma}
\begin{proof}
The cone measure of the dual immersion is given by
\[
d\mu_{\nu} = \rho_{\nu}\,dV_{\nu},
\]
where $\rho_{\nu}$ is the support function of $\nu$, and $dV_{\nu}$ is the induced volume form of the immersion.  Since $\nu$ is dual to $f$, it follows that its normal vector field $N_{\nu}(x) = f(x)/|f(x)|$.  The definitions of the dual map and the support function imply
\[
\rho_{\nu} = \langle N_{\nu},\nu\rangle = \left\langle \dfrac{f}{|f|}, \dfrac{N_f}{\langle N_f, f\rangle} \right\rangle = |f|^{-1}.
\]
A computation shows that
\[
dV_{\nu} = \nu^*(\iota_{N_{\nu}} dV_{\mathbb{R}^{n+1}}) = \phi^{-(n+1)} |f|\,dx.
\]
Thus,
\[
\mathcal{F}(\phi)^{-(n+1)} = \int_{\mathbb{R}^n} \phi^{-(n+1)}\,d\lambda = \mu_{\nu}(\mathbb{R}^n).
\]
\end{proof}

\subsection{Affine iteration}\label{affineiterationsection}

Fix an open, bounded, convex set $A\subset\mathbb{R}^n$.  We begin by proving that for any Legendre graph immersion $f_0$ over $A$, there exists a unique affine iteration $\{f_i\}_{i\in\mathbb{N}}$ over $A$.  It suffices to show that for every affine immersion $f_\psi$ over $A$, there exists an affine immersion $f_\phi$ satisfying
\[
\xi_{\phi}(x) = -c\,f_\psi(x),
\]
for some constant $c>0$.

\begin{proposition}\label{prescribednormal}
Let $A$ be an open, convex set with the origin contained in its interior.  If $\psi:\mathbb{R}^n \to \mathbb{R}$ is a smooth, positive, strictly convex function such that $\nabla \psi(\mathbb{R}^n)=A$, and if $f_\psi$ is its Legendre graph immersion over $A$, then there is a Legendre graph immersion $f_\phi$ over $A$, unique up to an additive constant, which satisfies
\[
\xi_{\phi}(x) = -c\,f_\psi(x),
\]
for some $c>0$.
\end{proposition}

\begin{proof}
Let $B_r$ be some small ball around the origin which is contained in $A$.  Since $\nabla \psi(\mathbb{R}^n) = A$, it follows that $\psi(x) \geq r\,|x| -c$ for some constant $c$.  Since $\psi$ is also positive, this lower bound implies $\psi^{-(n+2)}$ is integrable. It follows that 
\[
\dfrac{\psi(x)^{-(n+2)}}{\int_{\mathbb{R}^n} \psi^{-(n+2)}\,d\lambda}\,\lambda
\]
is a smooth, positive probability measure on $\mathbb{R}^n$.  McCann proved in \cite{McCann} that the equation
\[
\begin{cases}
    \dfrac{\det(\nabla^2 \phi)}{\lambda(A)} =\dfrac{\psi(x)^{-(n+2)}}{\int_{\mathbb{R}^n} \psi^{-(n+2)}\,d\lambda}\\
    \nabla \phi(\mathbb{R}^n) = A
\end{cases}
\]
has convex solutions which are unique up to an additive constant.  Since $\psi$ is smooth and positive, it follows from the regularity theory of Caffarelli \cite{Caffarelli} that $\phi$ is smooth, and thus strictly convex.  Thus, $f_\phi$ as defined in equation (\ref{fimmersion}) is a Legendre graph immersion over $A$.  

Since $\det(\nabla^2 \phi) = c\,\psi$ for $c=\lambda(A)^{-1/(n+2)}\|\psi\|_{-(n+2)}^{-1}$, Proposition \ref{Legaffinenormal} implies the affine normal of $f_\phi$ is given by
\[
\xi(x) = -c\,\big(\nabla \psi(x), \psi^*(\nabla \psi(x))\big) = -c\,f_\psi(x).
\]
Thus $f_\phi$ solves the prescribed affine normal problem, and $\phi$ is unique up to the addition of a constant.
\end{proof}

If we choose the additive constant so that $\phi>0$, then we can solve the applied affine normal problem for $f_\phi$ and create a sequence of Legendre graph immersions.  If the barycenter of $A$ is the origin, then the integral condition $\int_A \phi^*\,d\lambda = -\tau<0$ forces $\phi(x)>0$.  

Now we prove, Theorem \ref{affineiterationconverges}.

\begin{proof}
The functions $\{\phi_i\}$ which define the Legendre graph immersions $\{f_i\}$ are smooth solutions to the Monge--Amp\`{e}re iteration (\ref{MAiteration}) with $h(t)=t^{-(n+2)}$.  Theorem \ref{MAiteration1} with $p=1$ implies there exists a sequence of constants $\{a_i\}$ such that $\widetilde{\phi}_i(x) = \phi_i(x+a_i)$ converges smoothly to $\phi$, which is a smooth, convex solution of the Monge--Amp\`{e}re equation (\ref{ourpde}).  By Proposition \ref{affinesphereequation}, the Legendre graph immersion $f_\phi$ associated to $\phi$ is an elliptic affinesphere.

The Legendre transform of the translated sequence satisfy $\widetilde{\phi}_i^*(y) = \phi_i^*(y) - \langle a,y\rangle$.  Since the image of Legendre graph immersions are simply the graph of the Legendre transform, it follows that
\[
\widetilde{f}_i(\mathbb{R}^n) = \begin{pmatrix} I & 0 \\ a_i^T & 1 \end{pmatrix} f_i(\mathbb{R}^n).
\]
Thus if we define $M_i = \begin{pmatrix} I & 0 \\ a_i^T & 1 \end{pmatrix}$, then it follows that $M_i\,f_i(\mathbb{R}^n)$ converges smoothly to $f_\phi(\mathbb{R}^n)$ which is an elliptic affine hemisphere with its center at the origin.
\end{proof}

Finally, we would like to remark about the decreasing functionals along the affine iteration.  As a consequence of Lemma \ref{functionallimit} and the fact that Hypotheses \ref{hypB} hold for $h(t)=t^{-(n+2)}$ as shown in Section \ref{iteration1}, we have the inequalities
\[
\mathcal{F}(\phi_i) \geq \Big\langle \phi_{i+1},\scaleobj{.9}{\dfrac{\MA(\phi_{i+1})}{\lambda(A)}}\Big\rangle \,\mathcal{G}\Big(\scaleobj{.9}{\dfrac{\MA(\phi_{i+1})}{\lambda(A)}}\Big)^{-1} \geq \mathcal{F}(\phi_{i+1})
\]
when $\phi_i$ solves the Monge--Amp\`{e}re iteration with $h(t)=t^{-(n+2)}$.  Due to the calculation of these functionals in Section \ref{leggraphs}, this is equivalent to 
\[
\scaleobj{1.2}{
\mu_{\nu_i}(\mathbb{R}^n)^{-\frac{1}{n+1}} \geq \mu_{f_{i+1}}(\mathbb{R}^n)\,\Omega_{f_{i+1}}(\mathbb{R}^n)^{-\frac{n+2}{n+1}} \geq \mu_{\nu_{i+1}}(\mathbb{R}^n)^{-\frac{1}{n+1}}}.
\]

The right inequality, which comes from the definition of $\mathcal{G}$ and the equivalence of normal and dual definitions of affine surface area, is a special case of the inequality in Lutwak \cite[Lemma~AB]{Lutwak3}.  The left inequality says that the reverse inequality holds along the affine iteration, and that equality is achieved when the immersion is an affine sphere.

\medskip

\bibliographystyle{plain}
\bibliography{sources}

\end{document}